\documentclass{article} 
\usepackage{nips2015,times}
\usepackage{hyperref}
\usepackage{url}

\usepackage{pifont}
\usepackage{bm}
\usepackage{amsmath}
\usepackage{amssymb}
\usepackage{algorithm}
\usepackage{algorithmic}
\usepackage{xfrac}
\usepackage{graphicx}
\usepackage{tensor}
\usepackage{float}
\usepackage{wrapfig}
\usepackage{color}
\usepackage{framed}
\usepackage{eurosym}
\usepackage{supertabular}
\usepackage{colortbl}
\usepackage{sidecap}
\newenvironment{claim}[1]{\par\noindent\underline{Claim:}\space#1}{}

\usepackage{times}
\usepackage{amssymb}
\usepackage{amsmath}
\usepackage{mathtools}
\usepackage{graphicx} 
\usepackage{subfigure} 
\usepackage{amsthm}

\newenvironment{definition}[1][Definition]{\begin{trivlist}
\item[\hskip \labelsep {\bfseries #1}]}{\end{trivlist}}
\usepackage{algorithm}
\usepackage{algorithmic}

\newtheorem{theorem}{Theorem}
\newtheorem{assumption}{Assumption}
\newtheorem{lemma}{Lemma}
\newtheorem{corollary}{Corollary}

\title{Distributed SDDM Solvers: Theory \& Applications}

\author{
Rasul Tutunov \\
Department of Computer and Information Science,\\
University of Pennsylvania, \\
\texttt{tutunov@seas.upenn.edu} \\
\And
Haitham Bou-Ammar \\
Department of Computer and Information Science, \\
University of Pennsylvania, \\
Department of Operations Research and Financial Engineering,\\
Princeton University, \\
\texttt{haithamb@seas.upenn.edu} \\
\AND
Ali Jadbabaie,  \\
Department of Electrical and Systems Engineering,\\
University of Pennsylvania, \\
\texttt{jadbabaie@seas.upenn.edu} \\
}

\nipsfinalcopy 

\begin{document}

\maketitle

\begin{abstract}
In this paper, we propose distributed solvers for systems of linear equations given by symmetric diagonally dominant M-matrices based on the parallel solver of Spielman and Peng. We propose two versions of the solvers, where in the first, full communication in the network is required, while in the second communication is restricted to the R-Hop neighborhood between nodes for some $R \geq 1$. We rigorously analyze the convergence and convergence rates of our solvers, showing that our methods are capable of outperforming state-of-the-art techniques. 

Having developed such solvers, we then contribute by proposing an accurate distributed Newton method for network flow optimization. Exploiting the sparsity pattern of the dual Hessian, we propose a Newton method for network flow optimization that is both faster and more accurate than state-of-the-art techniques. Our method utilizes the distributed SDDM solvers for determining the Newton direction up to any arbitrary precision $\epsilon >0$. We analyze the properties of our algorithm and show superlinear convergence within a neighborhood of the optimal. Finally, in a set of experiments conducted on randomly generated and barbell networks, we demonstrate that our approach is capable of significantly outperforming state-of-the-art techniques.
\end{abstract}
\section{Introduction}

Solving systems of linear equations given by symmetric diagonally matrices (SDD) is of interest to researchers in a variety of fields. Such constructs, for example, are used to determine solutions to partial differential equations~\cite{c18} and computations of maximum flows in graphs~\cite{c19,c20}. Other application domains include machine learning~\cite{c21,c22}, and computer vision~\cite{c34}\footnote{This research is supported in parts by by ONR grant Number N00014-12-1-0997 and AFOSR grant FA9550-13-1-0097.}. 

Much interest has been devoted to determining fast algorithms for solving SDD systems. Recently, Spielman and Teng~\cite{c24}, utilized the multi-level framework of~\cite{c25}, pre-conditioners~\cite{c26}, and spectral graph sparsifiers~\cite{c27}, to propose a nearly linear-time algorithm for solving SDD systems. Further exploiting these ingredients, Koutis \textit{et. al}~\cite{c28, c29} developed an even faster algorithm for acquiring $\epsilon$-close solutions to SDD linear systems. Further improvements have been discovered by Kelner \textit{et. al}~\cite{c30}, where their algorithm relied on only spanning-trees and eliminated the need for graph sparsifiers and the multi-level framework.  

Motivated by applications, much progress has been made in developing  parallel versions of these algorithms. Koutis and Miller~\cite{c31} proposed an algorithm requiring nearly-linear work (i.e., total number of operations executed by a computation) and $m^{\sfrac{1}{6}}$ depth (i.e., longest chain of sequential dependencies in the computation) for planar graphs. This was then extended to general graphs in~\cite{c32} leading to depth close to $m^{\sfrac{1}{3}}$. Since then, Peng and Spielman~\cite{c11} have proposed an efficient parallel solver requiring nearly-linear work and poly-logarithmic depth without the need for low-stretch spanning trees. Their algorithm, which we provide a distribute construction for, requires sparse approximate inverse chains~\cite{c11} which facilitates the solution of the SDD system. 

Less progress, on the other hand, has been made on the distributed version of these solvers. Contrary to the parallel setting, memory is not shared and is rather distributed in the sense that each unit abides by its own memory restrictions. Furthermore, communication in a distributed setting fundamentally relies on message passing through communication links. Current methods, e.g., Jacobi iteration~\cite{c16,c17}, can be used for such distributed solutions but require substantial complexity. In~\cite{c33}, the authors propose a gossiping framework for acquiring a solution to SDDM systems in a distributed fashion. 
Recent work~\cite{c34} considers a local and asynchronous solution for solving systems of linear equations, where they acquire a bound on the number of needed multiplication proportional to the degree and condition number of the graph for one component of the solution vector.

\textbf{Contributions:} In this paper, we propose a fast distributed solver for linear equations given by symmetric diagonally dominant M-Matrices. Our approach distributes the parallel solver in~\cite{c11} by considering a specific approximated inverse chain which can be computed efficiently in a distributed fashion. We develop two versions of the solver. The first, requires full communication in the network, while the second is restricted to R-Hop neighborhood of nodes for some $R \geq 1$. Similar to the work in~\cite{c11}, our algorithms operate in two phases. In the first, a ``crude'' solution to the system of equations is retuned, while in the second a distributed R-Hop restricted pre-conditioner is proposed to drive the ``crude'' solution to an $\epsilon$-approximate one for any $\epsilon > 0$. Due to the distributed nature of the setting considered, the direct application of the sparsfier and pre-conditioner of Peng and Spielman~\cite{c11} is difficult due to the need of global information. Consequently, we propose a new sparse inverse chain which can be computed in a decentralized fashion for determining the solution to the SDDM system. 

Interestingly, due to the involvement of powers of matrices with eigenvalues less than one, our inverse chain is substantially shorter compared to that in~\cite{c11}. This leads us to a distributed SDDM solver with lower computational complexity compared to state-of-the-art methods.  Specifically, our algorithm's complexity is given by  \[\mathcal{O}\left(n^{3}\frac{\bm{\alpha}}{R}\frac{\bm{W}_{\text{max}}}{\bm{W}_{\text{min}}}\log\left(\frac{1}{\epsilon}\right)\right),\] with $n$ being the number of nodes in graph $\mathcal{G}$, $\bm{W}_{\text{max}}$ and $\bm{W}_{\text{min}}$ denoting the largest and smaller weights of the edges in $\mathcal{G}$, respectively, $\bm{\alpha}=\min\left\{n,\frac{d_{\text{max}}^{R+1}-1}{d_{\text{max}}-1}\right\}$ representing the upper bound on the size of the R-Hop neighborhood $\forall \bm{v} \in \mathcal{V}$, and $\epsilon \in (0,\frac{1}{2}]$ being the precision parameter. Furthermore, our approach improves current linear methods by a factor of $\log n$ and by a factor of the degree  compared to~ \cite{c34} for each component of the solution vector.

Having developed such distributed solvers, we next contribute by proposing an accurate distributed Newton method for network flow optimization. Exploiting the sparsity pattern of the dual Hessian, we propose a Newton method for network optimization that is both faster and more accurate than state-of-the-art techniques. Our method utilizes the proposed SDDM distributed solvers to approximate the Newton direction up to any arbitrary $\epsilon >0$. The resulting algorithm is an efficient and accurate distributed second-order method which performs almost identically to exact Newton. We analyze the properties of the proposed algorithm and show that, similar to conventional Newton methods, superlinear convergence within a neighborhood of the optimal value is attained.  We finally demonstrate the effectiveness of the approach in a set of experiments on randomly generated and Barbell networks.

\section{The parallel SDDM Solver}\label{Sec:ParallelSolver}
We now review the parallel solver for symmetric diagonally dominant (SDD) linear systems~\cite{c11}. 

\subsection{Problem Setting}\label{Sec:ProbSetting}
As detailed in~\cite{c11}, SDDM solvers consider the following system of linear equations:
\begin{equation}\label{lin_sys}
\bm{M}_{0}\bm{x} = \bm{b}_{0}
\end{equation}
where $\bm{M}_{0}$ is a Symmetric Diagonally Dominant M-Matrix (SDDM). Namely, $\bm{M}_{0}$ is symmetric positive definite with non-positive off diagonal elements, such that for all $i=1,2,\ldots, n$:
\begin{equation*}
\left[\bm{M}_{0}\right]_{ii} \ge -\sum_{j=1, j\ne i}^{n}\left[\bm{M}_{0}\right]_{ij}.
\end{equation*}
The system of Equations in~\ref{lin_sys} can be interpreted as representing an undirected weighted graph, $\mathcal{G}$, with $\bm{M}_{0}$ being its Laplacian. Namely, $\mathcal{G} = \left(\mathcal{V},\mathcal{E},\bm{W}\right)$, with $\mathcal{V}$ representing the set of nodes, $\mathcal{E}$ denoting the edges, and $\bm{W}$ representing the weighted graph adjacency. Nodes $\bm{v}_i$ and $\bm{v}_j$ are connected with an edge $\bm{e}=\left(i,j\right)$ iff $\bm{W}_{ij}> 0$, where: 
\begin{equation*}
\bm{W}_{ij} = -\left[\bm{M}_{0}\right]_{ij}.
\end{equation*}
Following~\cite{c11}, we seek $\epsilon$-approximate solutions to $\bm{x}^{\star}$, being the exact solution of $\bm{M}_{0}\bm{x}=\bm{b}_{0}$, defined as:
\begin{definition}[$\epsilon-$ Approximate Solution]
Let $\bm{x}^{\star}\in \mathbb{R}^{n}$ be the solution of $\bm{M}_{0}\bm{x}=\bm{b}_{0}$. A vector $\tilde{\bm{x}}\in \mathbb{R}^{n}$ is called an $\epsilon-$ approximate solution, if:
\begin{equation}
\left|\left|\bm{x}^{\star} - \tilde{\bm{x}}\right|\right|_{\bm{M}_{0}} \le \epsilon\left|\left|\bm{x}^{\star}\right|\right|_{\bm{M}_{0}}, \ \ \ \text{where $\left|\left|\bm{u}\right|\right|^{2}_{\bm{M}_0} = \bm{u}^{\mathsf{T}}\bm{M}_{0}\bm{u}$.
}
\end{equation}
\end{definition}

The R-hop neighbourhood of node $\bm{v}_{k}$ is defined as $\mathbb{N}_{R}\left(\bm{v}_{k}\right) = \{\bm{v}\in \mathcal{V}: \text{dist}\left(\bm{v}_{k}, \bm{v}\right)\le R\}$. We also make use of the diameter of a graph, $\mathcal{G}$, defined as $\text{diam}\left(\mathcal{G}\right) = \max_{\bm{v}_{i},\bm{v}_{j}\in \mathcal{V}}\text{dist}\left(\bm{v}_i,\bm{v}_j\right)$. 

\begin{definition}[Sparsity Pattern]
We say that a matrix $\bm{A} \in \mathbb{R}^{n\times n}$ has a sparsity pattern corresponding to the R-hop neighborhood if $\bm{A}_{ij} = 0$ for all $i = 1,\ldots, n$ and for all $j$ such that $\bm{v}_j\notin \mathbb{N}_{R}\left(\bm{v}_i\right)$. 
\end{definition}

We will denote the spectral radius of a matrix $\bm{A}$ by $\rho\left(\bm{A}\right) = \max{\left|\bm{\lambda}_i\right|}$, where $\bm{\lambda}_i$ represents an eigenvalue of the matrix $\bm{A}$. Furthermore, we will make use of the condition number\footnote{Please note that in the case of the graph Laplacian,  the condition number is defined as the ratio of the largest to the smallest nonzero eigenvalues.}, $\kappa\left(\bm{A}\right)$ of a matrix $\bm{A}$ defined as $\kappa=\left|\frac{\bm{\lambda}_{\text{max}}\left(\bm{A}\right)}{\bm{\lambda}_{\text{min}}\left(\bm{A}\right)}\right|$.  In~\cite{c11} it is shown that the condition number of the graph Laplacian is at most \[\mathcal{O}\left(n^{3}\frac{\bm{W}_{\text{max}}}{\bm{W}_{\text{min}}}\right),\] where $\bm{W}_{\text{max}}$ and $\bm{W}_{\text{min}}$ represent the largest and the smallest edge weights in $\mathcal{G}$. Finally, the condition number of a sub-matrix of the Laplacian is at most $\mathcal{O}\left(n^4\frac{\bm{W}_{\text{max}}}{\bm{W}_{\text{min}}}\right)$, see~\cite{c11}.

\subsection{Standard Splittings \& Approximations}\label{Sec:Standard}
Our first contribution is a distributed version of the parallel solver for SDDM systems of equations previously proposed in~\cite{c11}. Before detailing our solver, however, we next introduce basic mathematical machinery needed for developing the parallel solver of~\cite{c11}. The parallel solver commences by considering the standard splitting of the symmetric matrix $\bm{M}_{0}$: 
\begin{definition}
The standard splitting of a symmetric matrix $\bm{M}_{0}$ is: 
\begin{equation}
\bm{M}_{0} = \bm{D}_{0} - \bm{A}_{0}.
\end{equation}
\end{definition}
Here, $\bm{D}_0$ is a diagonal matrix consisting of the diagonal elements in $\bm{M}_{0}$ such that: 
\begin{equation*}
\left[\bm{D}_{0}\right]_{ii} = \left[\bm{M}_{0}\right]_{ii} \ \ \forall i=1,2,\dots, n.
\end{equation*}
Furthermore, $\bm{A}_0$ is a non-negative symmetric matrix such that:
\begin{align*}
\left[\bm{A}_0\right]_{ij}=
\begin{cases}
-\left[\bm{M}_{0}\right]_{ij} &: \text{if $i \ne j$,} \\
0 &: \text{otherwise.}
\end{cases}
\end{align*}

To quantify the quality of the acquired solutions, we define two additional mathematical constructs. First, the Loewner ordering is defined as: 
\begin{definition}
Let $\mathcal{\bm{S}}_{(n)}$ be the space of $n \times n$-symmetric matrices. The Loewner ordering $\preceq$ is a partial order on $\mathcal{\bm{S}}_{(n)}$ such that $\bm{Y}\preceq \bm{X}$ if and only if $\bm{X} - \bm{Y}$ is positive semidefinite.
\end{definition}

Having defined the Loewner order, we next define the notion of approximation for matrices ``$\approx_{\alpha}$'': 
\begin{definition}\label{Def:Ordering}
Let $\bm{X}$ and $\bm{Y}$ be positive semidefinite symmetric matrices. Then $\bm{X}\approx_{\alpha} \bm{Y}$ if and only iff
\begin{equation}
e^{-\alpha}\bm{X} \preceq \bm{Y} \preceq e^{\alpha}\bm{X}
\end{equation}
with $\bm{A}\preceq \bm{B}$ meaning $\bm{B} - \bm{A}$ is positive semidefinite.
\end{definition}

Based on the above definitions, the following lemma represents the basic characteristics of the $\approx_{\alpha}$ operator:
\begin{lemma}~\cite{c11}\label{approx_lemma_facts}
Let $\bm{X},\bm{Y},\bm{Z}$ and, $\bm{Q}$ be symmetric positive semi definite matrices. Then
\begin{enumerate}
\item[] (1) If $\bm{X}\approx_{\alpha} \bm{Y}$, then $\bm{X} + \bm{Z} \approx_{\alpha} \bm{Y} + \bm{Z}$, 
\item[] (2) If $\bm{X}\approx_{\alpha} \bm{Y}$ and $\bm{Z}\approx_{\alpha} \bm{Q}$, then $\bm{X} + \bm{Z} \approx_{\alpha} \bm{Y} + \bm{Q}$, 
\item[] (3) If $\bm{X}\approx_{\alpha_1} \bm{Y}$ and $\bm{Y} \approx_{\alpha_2} \bm{Z}$, then $\bm{X} \approx_{\alpha_1 + \alpha_2} \bm{Z}$
\item[] (4) If $\bm{X}$, and $\bm{Y}$ are non singular and $\bm{X}\approx_{\alpha} \bm{Y}$, then $\bm{X}^{-1}\approx_{\alpha} \bm{Y}^{-1}$, 
\item[] (5) If $\bm{X}\approx_{\alpha} \bm{Y}$ and $\bm{V}$ is a matrix, then $\bm{V}^{\mathsf{T}}\bm{X}\bm{V}\approx_{\alpha}\bm{V}^{\mathsf{T}}\bm{Y}\bm{V}$.
\end{enumerate}
\end{lemma}

Since the parallel solver returns an approximation, $\bm{Z}_{0}$, to $\bm{M}_{0}^{-1}$ (see Section~\ref{Sec:ParallelSolver}), the following lemma shows that ``good'' approximations to $\bm{M}^{-1}_0$ guarantee ``good'' approximate solutions to $\bm{M}_{0}\bm{x}=\bm{b}_{0}$.
\begin{lemma}\label{lemma_approx_matrix_inverse}
Let $\bm{Z}_0\approx_{\epsilon}\bm{M}^{-1}_0$, and $\tilde{\bm{x}} = \bm{Z}_0\bm{b}_0$, then $\tilde{\bm{x}}$ is $\sqrt{2^{\epsilon}(e^{\epsilon} - 1)}$ approximate solution to $\bm{M}_{0}\bm{x}=\bm{b}_{0}$.
\end{lemma}
\begin{proof}
The proof can be found in the appendix. 
\end{proof}

\subsection{The Solver}\label{Parallel:SDDM}\label{Sec:ParrallelSolver}
The parallel SDDM solver proposed in~\cite{c11} is a parallelized technique for solving the problem of Section~\ref{Sec:ProbSetting}. It makes use of inverse approximated chains (see Definition~\ref{Def:InvChain}) to determine $\tilde{\bm{x}}$ and can be split in two steps. In the first, Algorithm~\ref{Algo:Inv}, a ``crude'' approximation, $\bm{x}_{0}$, of $\bm{\tilde{x}}$ is returned. $\bm{x}_{0}$ is driven to the $\epsilon$-close solution, $\tilde{\bm{x}}$, using Richardson Preconditioning in Algorithm~\ref{Algo:Inv2}. Before we proceed, we start with the following two Lemmas which enable the definition of inverse chain approximation. 

\begin{lemma}~\cite{c11}\label{SDDM_splitting_lemma}
If $\bm{M} = \bm{D} - \bm{A}$ is an SDDM matrix, with $\bm{D}$ being positive diagonal, and $\bm{A}$ denoting a non-negative symmetric matrix, then $\bm{D} - \bm{A}\bm{D}^{-1}\bm{A}$ is also SDDM.
\end{lemma}

\begin{lemma}~\cite{c11}\label{approx_inverse_formulae_lemma}
Let $\bm{M} =\bm{D} - \bm{A}$ be an SDDM matrix, where $\bm{D}$ is positive diagonal and $\bm{A}$ a symmetric matrix. Then 
\begin{align}\label{Inv_of_SDDM}
\left(\bm{D}-\bm{A}\right)^{-1} &= \frac{1}{2}\Big[\bm{D}^{-1} + \left(\bm{I} + \bm{D}^{-1}\bm{A}\right)\left(\bm{D} - \bm{A}\bm{D}^{-1}\bm{A}\right)^{-1} \left(\bm{I} + \bm{A}\bm{D}^{-1}\right)\Big].
\end{align}
\end{lemma}


Given the results in Lemmas \ref{SDDM_splitting_lemma} and \ref{approx_inverse_formulae_lemma}, we now can consider inverse approximated chains of $\bm{M}_0$:
\begin{definition}\label{Def:InvChain}
Let $\mathcal{C} = \{\bm{M}_0, \bm{M}_1, \ldots, \bm{M}_d\}$  be a collection of SDDM matrices such that $\bm{M}_i = \bm{D}_i - \bm{A}_i$, with $\bm{D}_i$ a positive diagonal matrix, and $\bm{A}_i$ denoting a non-negative symmetric matrix. Then $\mathcal{C}$ is an inverse approximated chain if there exists  positive real numbers $\epsilon_0, \epsilon_1, \ldots, \epsilon_d$ such that: 
\begin{enumerate}
\item[](1) $\bm{D}_i - \bm{A}_i \approx_{\epsilon_{i-1}} \bm{D}_{i-1} - \bm{A}_{i-1}\bm{D}^{-1}_{i-1}\bm{A}_{i-1} \ \ \forall i = 1,\ldots, d$, 
\item[] (2) $\bm{D}_i \approx_{\epsilon_{i-1}}\bm{D}_{i-1}$, and 
\item[] (3) $\bm{D}_d\approx_{\epsilon_d} \bm{D}_d - \bm{A}_d$. 
\end{enumerate}
\end{definition}
 
It is shown in~\cite{c11} that an approximate inverse chain allows for ``crude'' solutions to the system of linear equations in $\bm{D}_{0}-\bm{A}_{0}$ in time proportional to the number of non-zeros entries in the matrices in the inverse chain. Such a procedure is summarized in the following algorithm: 

\begin{algorithm}\label{Algo:CrudeParallel}
  \caption{$\text{ParallelRSolve}\left(\bm{M}_0,\bm{M}_1,\ldots, \bm{M}_d, \bm{b}_0\right)$}
  \label{Algo:Inv}
  \begin{algorithmic}[1]
	\STATE \textbf{Input}: Inverse approximated chain, $\{\bm{M}_0,\bm{M}_1,\ldots, \bm{M}_d\}$, and $\bm{b}_0$ being 
	\STATE \textbf{Output}: The ``crude'' approximation, $\bm{x}_0$, of $\bm{x}^{\star}$    
    \FOR  {$i=1$ to $d$} 
    	\STATE $\bm{b}_i = \left(\bm{I}+\bm{A}_{i-1}\bm{D}^{-1}_{i-1}\right)\bm{b}_{i-1}$
    \ENDFOR 
    \STATE $\bm{x}_d = \bm{D}^{-1}_d\bm{b}_d$
    \FOR  {$i=d-1$ to $0$} 
    	\STATE $\bm{x}_{i} = \frac{1}{2}\left[\bm{D}^{-1}_i\bm{b}_{i} + \left(\bm{I} + \bm{D}^{-1}_i\bm{A}_i\right)\bm{x}_{i+1}\right]$
    \ENDFOR 
    \STATE \textbf{return} $x_0$ 
  \end{algorithmic}
\end{algorithm}
On a high level, Algorithm~\ref{Algo:CrudeParallel} operates in two phases. In the first (i.e., lines~3-5) a forward loop (up-to the length of the inverse chain $d$) computes intermediate vectors $\bm{b}_{i}$ as:
\begin{equation}
\label{Eq:bUpdate}
\bm{b}_{i}=\left(\bm{I}+\bm{A}_{i-1}\bm{D}_{i-1}^{-1}\right)\bm{b}_{i-1},
\end{equation}
for $i=\{1,\dots,d\}$. These can then be used to compute the ``crude'' solution $\bm{x}_{0}$ using a ``backward'' loop (i.e., lines~7-9). Consequently, the crude solution is computed iteratively backwards as: 

\begin{equation*}
\bm{x}_{i} = \frac{1}{2}\left[\bm{D}^{-1}_i\bm{b}_{i} + \left(\bm{I} + \bm{D}^{-1}_i\bm{A}_i\right)\bm{x}_{i+1}\right], 
\end{equation*}
with $\bm{x}_{d}=\bm{D}_{d}^{-1}\bm{b}_{d}$ and $\bm{b}_{i}$ as defined in Equation~\ref{Eq:bUpdate}. The quality of the ``crude'' solution returned by the Algorithm is quantified in the following lemma:
\begin{lemma}~\cite{c11}\label{Rude_Alg_guarantee_Lemma}
Let $\{\bm{M}_0, \bm{M}_1,\ldots, \bm{M}_d\}$ be the inverse approximated chain and denote $\bm{Z}_0$ be the operator defined by $\text{ParallelRSolve}\left(\bm{M}_0, \bm{M}_1,\ldots, \bm{M}_d, \bm{b}_0\right)$, namely, $\bm{x}_0 = \bm{Z}_0\bm{b}_0$. Then
\begin{equation}\label{appr_inv_express}
\bm{Z}_0\approx_{\sum_{i=0}^{d}\epsilon_i}\bm{M}^{-1}_0
\end{equation}
\end{lemma}
Having returned a ``crude'' solution to $\bm{M}_{0}\bm{x}=\bm{b}_{0}$, the authors in~\cite{c11} obtain arbitrary close solutions using the \emph{preconditioned Richardson iterative scheme}. The first step in the exact solver is the usage of Algorithm~\ref{Algo:CrudeParallel} to obtain the ``crude'' solution $\chi$. This is then updated through the loop in lines~4-8 to obtain an $\epsilon$-close solution to $\bm{x}^{\star}$, see Algorithm~\ref{Algo:Inv2}.
\begin{algorithm}[h!]
  \caption{$\text{ParallelESolve}\left(\bm{M}_0, \bm{M}_1,\ldots, \bm{M}_d,  \bm{b}_0, \epsilon\right)$}
  \label{Algo:Inv2}
  \begin{algorithmic}[1]
	\STATE \textbf{Input}: Inverse approximated chain $\{\bm{M}_0,\bm{M}_1,\ldots, \bm{M}_d\}$, $\bm{b}_0$, and $\epsilon$. 
	\STATE \textbf{Output}: $\epsilon$ close approximation, $\tilde{\bm{x}}$, of $\bm{x}^*$
	\STATE \textbf{Initialize}: $\bm{y}_0 = 0$; \\
	$\chi = \text{ParallelRSolve}\left(\bm{M}_0,\bm{M}_1,\ldots, \bm{M}_d, \bm{b}_0\right)$ (i.e., Algorithm~\ref{Algo:Inv})
    \FOR  {$k=1$ to $q$}
    	\STATE $\bm{u}_{k}^{(1)} = \bm{M}_0\bm{y}_{k-1}$
    	\STATE $\bm{u}_{k}^{(2)} = \text{ParallelRSolve}\left(\bm{M}_0, \bm{M}_1,\ldots, \bm{M}_d, \bm{u}_{k}^{(1)}\right)$
    	\STATE $\bm{y}_{k} = \bm{y}_{k-1} - \bm{u}_{k}^{(2)} + \chi$ 
    \ENDFOR 
    \STATE $\tilde{\bm{x}} = \bm{y}_q$
    \STATE \textbf{return} $\tilde{\bm{x}}$ 
  \end{algorithmic}
\end{algorithm}
Following the analysis in~\cite{c11}, Lemma~\ref{Exact_Alg_guarantee_lemma} provides the iteration count needed by Algorithm~\ref{Algo:Inv2} to arrive at $\tilde{\bm{x}}$:
\begin{lemma}~\cite{c11}\label{Exact_Alg_guarantee_lemma}
Let $\{\bm{M}_0, \bm{M}_1\ldots \bm{M}_d\}$ be an inverse approximated chain such that $\sum_{i=1}^{d}\epsilon_i < \frac{1}{3}\ln2$. Then $\text{ParallelESolve}\left(\bm{M}_0, \bm{M}_1,\ldots, \bm{M}_d, \bm{b}_0, \epsilon\right)$ requires $q$ iterations to arrive at an $\epsilon$ close solution of $\bm{x}^{\star}$ with:  
$q = \mathcal{O}\left(\log\frac{1}{\epsilon}\right)$.
\end{lemma}

\section{Distributed SDDM Solvers}
Having introduced the parallel solver, next we detail our first contribution by proposing a distributed solver for SDDM linear systems. In particular, we develop two versions. The first, requires full communication in the network, while the second restricts communication to the R-Hop neighborhood increasing its applicability. Not only our solver improves the computational complexity of distributed methods for system of equations represented by an SDDM matrix, but can also be applied to a variety of fields including distributed Newton methods, computer vision, among others.

To compute the solution of SDDM systems in a distributed fashion, we follow a similar strategy to that of~\cite{c11} with major differences. Our distributed solver requires two steps to arrive at an $\epsilon$-close approximation to $\bm{x}^{\star}$. Similar to~\cite{c11}, the first step adopts an inverse approximated chain to determine a ``crude'' solution to $\bm{x}^{\star}$. The inverse chain proposed in~\cite{c11} can not be computed in a distributed fashion rendering its immediate application to our setting difficult. Hence, we par-ways with~\cite{c11} by proposing an inverse chain which can be computed in a distributed fashion. This chain, defined in Section~\ref{Sec:Chain}, enables the distributed computation of both a crude and exact solution to $\bm{M}_{0}\bm{x}=\bm{b}_{0}$. Interestingly, due to the involvement of matrices with eigenvalues less than 1, the length, $d$, of our inverse chain is substantially shorter compared to that of~\cite{c11}, allowing for fast and efficient distributed solvers. Given the crude solution, the second step computes an $\epsilon$-close approximation to $\bm{x}^{\star}$.  This is achieved by proposing a distributed version of the Richardson pre-conditioning scheme. Definitely, this step is also similar in spirit to that in~\cite{c11}, but generalizes the aforementioned authors' work into a distributed setting and allows for $\epsilon$-close approximation to $\bm{x}^{\star}$ for \emph{any arbitrary} $\epsilon > 0$. Main results on the full communication version of the solver are summarized in the following theorem: 

\begin{theorem}\label{Main_Theorem}
There exists a distributed algorithm,
\[\mathcal{A}\left(\{[\bm{M}_0]_{k1},\ldots [\bm{M}_0]_{kn}\}, [\bm{b}_0]_k, \epsilon\right),\] that computes $\epsilon$-close approximations to the solution of $\bm{M}_{0}\bm{x}=\bm{b}_{0}$ in 
$\mathcal{O}\left(n^2\log\kappa \log\left(\frac{1}{\epsilon}\right)\right)$ time steps, with $n$ the number of nodes in $\mathcal{G}$, $\kappa$ the condition number of  $\bm{M}_0$, and $[\bm{M}_{0}]_{k\cdot}$ the $k^{th}$ row of $\bm{M}_{0}$, as well as $\epsilon\in \left(0, \frac{1}{2}\right]$ representing the precision parameter. 
\end{theorem}

The above distributed algorithms require no knowledge of the graph's topology, but do require the information from all other nodes (i.e., full communication) for computing solutions to $\bm{M}_{0}\bm{x}=\bm{b}_{0}$. In a variety of real-world applications (e.g., smart-grids, transportation) load, capacity, money and resource restrictions pose problems for such a requirement. Consequently, we extend the previous solvers to an R-Hop version in which communication is restricted to the R-Hop neighborhood between nodes for some $R\geq 1$. Again we follow a two-step strategy, where in the first we compute the crude solution and in the second an $\epsilon$-close approximation to $\bm{x}_{0}$ is determined using an ``R-Hop restricted'' Richardson pre-conditioner.  These results are captured in the following theorem: 

\begin{theorem}\label{Main_TheoremTwo}
There is a decentralized algorithm, 
\[\mathcal{A}(\{[\bm{M}_0]_{k1},\ldots [\bm{M}_0]_{kn}\}, [\bm{b}_0]_k, R, \epsilon),\] that uses only $R$-Hop communication between the nodes and computes $\epsilon$-close solutions to $\bm{M}_{0}\bm{x}=\bm{b}_{0}$ in \[\mathcal{O}\left(\left(\frac{\alpha\kappa}{R} + \alpha Rd_{max}\right) \log\left(\frac{1}{\epsilon}\right)\right)\] time steps, with $n$ being the number of nodes in $\mathcal{G}$, $d_{max}$ denoting the maximal degree, $\kappa$ the condition number of $\bm{M}_0$, and $\alpha = \min\left\{n, \frac{\left(d^{R+1}_{\text{max}} - 1\right)}{\left(d_{\text{max}} - 1\right)}\right\}$ representing the upper bound on the size of the R-hop neighborhood $\forall \bm{v}\in \mathcal{V}$, and $\epsilon\in (0, \frac{1}{2}]$ being the precision parameter. 
\end{theorem}

The remainder of the section details the above distributed solvers and provides rigorous theoretical guarantees on the convergence and convergence rates of each of the algorithms. We start by describing solvers requiring full network communication and then detail the R-Hop restricted versions. 

\subsection{Full Communication Distributed Solvers}
As mentioned previously, our strategy for a distributed implementation of the parallel solver in~\cite{c11} requires two steps. In the first a ``crude'' solution is returned, while in the second an $\epsilon$-close approximation (for any arbitrary $\epsilon >0$) to $\bm{x}_{0}$ is computed. 

\subsubsection{``Crude'' Distributed SDDM Solvers}\label{Sec:Chain}
The distributed crude solver, represented in Algorithm~\ref{Algo:DisRudeApprox}, resembles similarities to the parallel one of~\cite{c11} with major differences. On a high level, Algorithm~\ref{Algo:DisRudeApprox} operates in two distributed phases. In the first, a forward loop computes intermediate $\bm{b}$ vectors which are then used to update the crude solution of $\bm{M}_{0}\bm{x}=\bm{b}_{0}$. The crucial difference to~\cite{c11}, however, is the distributed nature of these computations. Precisely, the algorithm is responsible for determining the crude solution for each node $\bm{v}_{k} \in \mathcal{V}$. Due to such distributed nature, the inverse approximated chain used in~\cite{c11} is inapplicable to our setting. Therefore, the second crucial difference to the parallel SDDM solver is the introduction of a new chain which can be computed in a distributed fashion. Starting from $\bm{M}_{0} = \bm{D}_{0} - \bm{A}_{0}$, our ``crude'' distributed solver makes use of the following collection as the inverse approximated chain: 
\begin{equation}\label{Eq:InverseChain}
\mathcal{C} = \{\bm{A}_0,\bm{D}_0,\bm{A}_1, \bm{D}_1,\ldots, \bm{A}_d, \bm{D}_d\},
\end{equation}
where 
$\bm{D}_{k}=\bm{D}_{0}$, \text{and} $\bm{A}_{k}=\bm{D}_{0}\left(\bm{D}_{0}^{-1}\bm{A}_{0}\right)^{2^{k}}$, 
for $k=\{1,\dots, d\}$ with $d$ being the length of the inverse chain. Note that since the magnitude of the eigenvalues of $\bm{D}^{-1}_0\bm{A}_0$ is strictly less than 1, $\left(\bm{D}^{-1}_0\bm{A}_0\right)^{2^{k}}$ tends to zero as $k$ increases which reduces the length of the chain needed. This length is explicitly computed in Section~\ref{Sec:Length} for attaining $\epsilon$-close approximations to $\bm{x}^{\star}$. 

It is relatively easy to verify that $\mathcal{C}$ is an inverse approximated chain, since: 
\begin{enumerate}
\item[] (1) $\bm{D}_i - \bm{A}_i \approx_{\epsilon_{i-1}} \bm{D}_{i-1} - \bm{A}_{i-1}\bm{D}^{-1}_{i-1}\bm{A}_{i-1}$ with $\epsilon_i = 0$ for $i = 1,\ldots, d$,
\item[] (2)  $\bm{D}_{i}\approx_{\epsilon_{i-1}}\bm{D}_{i-1}$ with $\epsilon_i = 0$ for $i = 1,\ldots, d$, and 
\item[] (3) $\bm{D}_d\approx_{\epsilon_d}\bm{D}_d - \bm{A}_d$.
\end{enumerate}

Algorithm~\ref{Algo:DisRudeApprox} returns the $k^{th}$ component of the approximate solution vector, $[\bm{x}_{0}]_{k}$. As inputs it requires the inverse chain of Equation~\ref{Eq:InverseChain}, the $k^{th}$ component of $\bm{b}_{0}$, and the length of the inverse chain. Namely, each node, $\bm{v}_k\in \mathcal{V}$, receives the $k^{th}$ row of $\bm{M}_0$, the $k^{th}$ value of $\bm{b}_{0}$ (i.e., $[\bm{b}_0]_k$), and the length of the inverse approximated chain $d$ and then operates in two parts. In the first (i.e., lines~1-8) a forward loop computes the $k^{th}$ component of $\bm{b}$ exploiting the distributed inverse chain, while in the second a backward loop (lines~9-17) is responsible for computing the $k^{th}$ component of the``crude'' solution $[\bm{x}_{0}]_{k}$ which is then returned. Essentially, in both the forward and backward loops each of the $\bm{b}$ and $\bm{x}$ vectors are computed in a distributed fashion based on the relevant components of the matrices, explaining the usage of $\mathbb{N}_{\cdot}$ loops in Algorithm~\ref{Algo:DisRudeApprox}. 

\begin{algorithm}[t!]
  \caption{
 		 \hspace{0em} $\text{DistrRSolve}\Big(\left\{[\bm{M}_0]_{k1},\ldots, [\bm{M}_0]_{kn}\right\}, 	 [\bm{b}_0]_k, d\Big)$
	   }
	   \label{Algo:DisRudeApprox}
  \begin{algorithmic}[1]
	\STATE \textbf{Part One: Computing $[\bm{b}_{i}]_{k}$}
	\STATE $[\bm{b}_1]_k = [\bm{b}_0]_k + \sum_{j: \bm{v}_j\in \mathbb{N}_{1}\left(\bm{v}_k\right)}[\bm{A}_0\bm{D}^{-1}_0]_{kj}[\bm{b}_{0}]_j$	
	\FOR {$i = 2$ to $d$}
		\FOR {$j: \bm{v}_j \in \mathbb{N}_{2^{i-1}}\left(\bm{v}_k\right)$}
			\STATE$
			\left[(\bm{A}_0\bm{D}^{-1}_0)^{2^{i-1}}\right]_{kj} =  \sum_{r=1}^{n}\frac{[\bm{D}_0]_{rr}}{[\bm{D}_0]_{jj}}\left[(\bm{A}_0\bm{D}^{-1}_0)^{2^{i-2}}\right]_{kr} \left[(\bm{A}_0\bm{D}^{-1}_0)^{2^{i-2}}\right]_{jr}$
		\ENDFOR 
		\STATE $[\bm{b}_i]_k = [\bm{b}_{i-1}]_k + \sum_{j: \bm{v}_j \in \mathbb{N}_{2^{i-1}}\left(\bm{v}_k\right)}\left[(\bm{A}_0\bm{D}^{-1}_0)^{2^{i-1}}\right]_{kj}[\bm{b}_{i-1}]_{j}$
	\ENDFOR
	\\\hrulefill
	\STATE \textbf{Part Two: Computing $[\bm{x}_{0}]_{k}$}
	\STATE $[\bm{x}_d]_k = \sfrac{[\bm{b}_d]_k}{[\bm{D}_0]_{kk}}$
	\FOR {$i = d - 1$ to $1$}
		\FOR {$j: \bm{v}_j \in \mathbb{N}_{2^i}(\bm{v}_k)$}
			\STATE$
			\left[(\bm{D}^{-1}_0\bm{A}_0)^{2^i}\right]_{kj} = 
			\sum_{r=1}^{n}\frac{[\bm{D}_0]_{jj}}{[\bm{D}_0]_{rr}}\left[(\bm{D}^{-1}_0\bm{A}_0)^{2^{i-1}}\right]_{kr} \left[(\bm{D}^{-1}_0\bm{A}_0)^{2^{i-1}}\right]_{jr}$
		\ENDFOR 
		\STATE$
		[\bm{x}_i]_k =  \frac{[\bm{b}_i]_k}{2[\bm{D}_0]_{kk}} + \frac{[\bm{x}_{i+1}]_{k+1}}{2} +\frac{1}{2}\sum_{j: \bm{v}_j \in \mathbb{N}_{2^{i}}(\bm{v}_k)}\left[(\bm{D}^{-1}_0\bm{A}_0)^{2^i}\right]_{kj}[\bm{x}_{i+1}]_j$
	\ENDFOR 
	\STATE $[\bm{x}_0]_k = \frac{[\bm{b}_0]_k}{2[\bm{D}_{0}]_{kk}} + \frac{[\bm{x}_{1}]_k}{2} + \frac{1}{2}\sum_{j:\bm{v}_j\in \mathbb{N}_{1}(\bm{v}_k)}[\bm{D}^{-1}_0\bm{A}_0]_{kj}[\bm{x}_1]_j$
    \STATE \textbf{return:} $[\bm{x}_0]_k$ 
  \end{algorithmic}
\end{algorithm}

\textbf{Theoretical Guarantees of Algorithm~\ref{Algo:DisRudeApprox}:} Due to the modifications made to the original parallel solver, new theoretical analysis quantifying convergence and accuracy of the returned ``crude'' solution is needed. We show that $\text{DistrRSolve}$ computes the $k^{th}$ component of the ``crude'' approximation of $\bm{x}^{\star}$ and provide time complexity analysis. These results are summarized in the following lemma\footnote{For ease of presentation, we leave the proof of the lemma to the appendix.}: 

\begin{lemma}\label{Rude_Dec_Alg_guarantee_Lemma}
Let $\bm{M}_0 = \bm{D}_0 - \bm{A}_0$ be the standard splitting of $\bm{M}_{0}$. Let $\bm{Z}^{\prime}_0$ be the operator defined by $\text{DistrRSolve}([\{[\bm{M}_0]_{k1},\ldots, [\bm{M}_0]_{kn}\}, [\bm{b}_0]_k, d)$ (i.e., $\bm{x}_0 = \bm{Z}^{\prime}_0\bm{b}_0$). Then
\begin{equation*}
\bm{Z}^{\prime}_0\approx_{\epsilon_d} \bm{M}^{-1}_0.
\end{equation*}
Moreover, Algorithm~\ref{Algo:DisRudeApprox} requires  $\mathcal{O}\left(dn^2\right)$ time steps. 
\end{lemma}

In words, Lemma~\ref{Rude_Dec_Alg_guarantee_Lemma} states that Algorithm~\ref{Algo:DisRudeApprox} requires $\mathcal{O}\left(dn^2\right)$ to arrive at an $\epsilon_{d}$ approximation to the real inverse $\bm{M}_{0}^{-1}$, where this approximation is quantified using Definition~\ref{Def:Ordering}: 
\begin{equation*}
e^{-\epsilon_{d}} \bm{Z}_{0}^{\prime}\preceq \bm{M}_{0}^{-1}\preceq e^{\epsilon_{d}}\bm{Z}^{\prime}_{0}.
\end{equation*}

$\bm{Z}_{0}^{\prime}$ can then be used to compute the crude solution as $\bm{x}_{0}=\bm{Z}_{0}^{\prime}\bm{b}$. Note that the accuracy of approximating $\bm{M}_{0}$ is limited to $\epsilon_{d}$ motivating the need for an ``exact'' distributed solver reducing the error to any $\epsilon>0$.

\subsubsection{``Exact'' Distributed SDDM Solvers}
Having introduced $\text{DistrRSolve}$, we are now ready to present a distributed version of Algorithm~\ref{Algo:Inv2} which enables the computation of $\epsilon$ close solutions for $\bm{M}_{0}\bm{x}=\bm{b}_{0}$. Contrary to the work of~\cite{c11}, our algorithm is capable of acquiring solutions up to any arbitrary $\epsilon > 0$. Similar to $\text{DistrRSolve}$, each node $\bm{v}_k\in \mathcal{V}$ receives the $k^{th}$ row of $\bm{M}_0$, $[\bm{b}_0]_k$, $d$ and a precision parameter $\epsilon$ as inputs. Node $\bm{v}_k$ then computes the $k^{th}$ component of the $\epsilon$  close approximation of $\bm{x}^{\star}$ by using $\text{DistrRSolve}$ as a sub-routine and updates the solution iteratively as shown in lines~2-6 in Algorithm~\ref{Alg_ExactBla}.

\begin{algorithm}[h!]
  \caption{$\text{DistrESolve}\left(\{[\bm{M}_0]_{k1},\ldots, [\bm{M}_0]_{kn}\}, [\bm{b}_0]_k, d, \epsilon\right)$}
  \begin{algorithmic}[1]
\STATE \textbf{Initialize}: $[\bm{y}_0]_k = 0$; 
$[\chi]_k = \text{DistrRSolve}\left(\{[\bm{M}_0]_{k1},\ldots, [\bm{M}_0]_{kn}\}, [\bm{b}_0]_k, d\right)$ (i.e., Algorithm~\ref{Algo:DisRudeApprox})	
	\FOR {$t=1$ to $q$}
		\STATE $\left[\bm{u}^{(1)}_{t}\right]_k = [\bm{D}_0]_{kk}[\bm{y}_{t-1}]_k - \sum_{j: \bm{v}_j\in \mathbb{N}_{1}(\bm{v}_k)}[\bm{A}_{0}]_{kj}[\bm{y}_{t-1}]_j$
		\STATE $\left[\bm{u}^{(2)}_{t}\right]_k = \text{DistrRSolve}(\{[\bm{M}_0]_{k1},\ldots, [\bm{M}_0]_{kn}\}, \left[\bm{u}^{(1)}_{t}\right]_k, d,)$
		\STATE $[\bm{y}_t]_k = [\bm{y}_{t-1}]_k - \left[\bm{u}^{(2)}_{t}\right]_k + [\chi]_k$ 
	\ENDFOR 
    \STATE $[\tilde{\bm{x}}]_k = [\bm{y}_q]_k$
    \STATE \textbf{return} $[\tilde{\bm{x}}]_k$ 
  \end{algorithmic}
  \label{Alg_ExactBla}  
\end{algorithm}

\textbf{Analysis of Algorithm~\ref{Alg_ExactBla}:} Here, we again provide the theoretical analysis needed for quantifying the convergence and computational time of the exact algorithm for returning $\epsilon$-close approximation to $\bm{x}^{\star}$. The following lemma shows that $\text{DistrESolve}$ computes the $k^{th}$ component of the $\epsilon$-close approximation of $x^{\star}$:
\begin{lemma}\label{Dist_Exact_algorithm_guarantee_lemma}
Let $\bm{M}_0 = \bm{D}_0 - \bm{A}_0$ be the standard splitting. Further, let $\epsilon_d < \frac{1}{3}\ln2$ in the nverse approximated chain $\mathcal{C} = \{\bm{A}_0,\bm{D}_0,\bm{A}_1, \bm{D}_1,\ldots, \bm{A}_d, \bm{D}_d\}$. Then $\text{DistrESolve}(\{[\bm{M}_0]_{k1},\ldots, [\bm{M}_0]_{kn}\}, [b_0]_k, d, \epsilon)$ requires $\mathcal{O}\left(\log\frac{1}{\epsilon}\right)$ iterations to return the $k^{th}$ component of the $\epsilon$ close approximation for $\bm{x}^{\star}$.
\end{lemma}

The above lemma proofs that the algorithm requires $\mathcal{O}(\log \frac{1}{\epsilon})$ iterations for attaining for returning the $k^{th}$ of the $\epsilon$-close approximation to $\bm{x}^{\star}$. Consequently, the overall complexity can be summarized as: 
\begin{lemma}\label{time_complexity_of_distresolve}
Let $\bm{M}_0 =\bm{D}_0 - \bm{A}_0$ be the standard splitting. Further, let $\epsilon_d < \frac{1}{3}\ln2$ in the inverse approximated chain $\mathcal{C} = \{\bm{A}_0,\bm{D}_0,\bm{A}_1, \bm{D}_1,\ldots, \bm{A}_d, \bm{D}_d\}$. Then, $\text{DistrESolve}(\{[\bm{M}_0]_{k1},\ldots, [\bm{M}_0]_{kn}\}, [\bm{b}_0]_k, d, \epsilon)$ requires $\mathcal{O}\left(dn^2\log(\frac{1}{\epsilon})\right)$ time steps.
\end{lemma}

\subsubsection{Length of the Inverse Chain}\label{Sec:Length}
Both introduced algorithms depend on the length of the inverse approximated chain, $d$. Here, we provide an analysis to determine the value of $d$ which guarantees $\epsilon_d < \frac{1}{3}\ln2$ in $\mathcal{C} = \{\bm{A}_0,\bm{D}_0,\bm{A}_1, \bm{D}_1,\ldots, \bm{A}_d, \bm{D}_d\}$: 
\begin{lemma}\label{eps_d_lemma}
Let $\bm{M}_0 = \bm{D}_0 - \bm{A}_0$ be the standard splitting and $\kappa$ denote the condition number of $\bm{M}_0$. Consider the inverse approximated chain 
\begin{equation*}
\mathcal{C} = \{\bm{A}_0,\bm{D}_0,\bm{A}_1, \bm{D}_1,\ldots, \bm{A}_d, \bm{D}_d\},
\end{equation*}
with $d =  \lceil \log \left(2\ln\left(\frac{\sqrt[3]{2}}{\sqrt[3]{2} - 1}\right)\kappa\right)\rceil$, then $
\bm{D}_0\approx_{\epsilon_d} \bm{D}_0 - \bm{D}_0\left(\bm{D}^{-1}_0\bm{A}_0\right)^{2^d},$ 
with $\epsilon_d < \frac{1}{3}\ln2$.
\end{lemma}

\begin{proof}
The proof will be given as a collection of claims:

\begin{claim}
Let $\kappa$ be the condition number of $\bm{M}_0 = \bm{D}_0 - \bm{A}_0$, and $\{\lambda_i\}^{n}_{i=1}$ denote the eigenvalues of $\bm{D}^{-1}_0\bm{A}_0$. Then,
$|\lambda_i| \le 1 - \frac{1}{\kappa}$,  
for all $i=1,\ldots, n$
\end{claim}
\begin{proof}
See Appendix. 
\end{proof}

Notice that if $\lambda_i$ represented an eigenvalue of $\bm{D}^{-1}_0\bm{A}_0$, then $\lambda^{r}_i$ is an eigenvalue of $\left(\bm{D}^{-1}_0\bm{A}_0\right)^r$ for all $r\in \mathbb{N}$. Therefore, we have 
\begin{equation}\label{spect_radius}
\rho\left(\left(\bm{D}^{-1}_0\bm{A}_0\right)^{2^d}\right)\le \left(1 - \frac{1}{\kappa}\right)^{2^d}
\end{equation}

\begin{claim}
Let $\bm{M}$ be an SDDM matrix and consider the splitting $\bm{M} =\bm{D} -\bm{A}$, with $\bm{D}$ being non negative diagonal and $\bm{A}$ being symmetric non negative. Further, assume that the eigenvalues of $\bm{D}^{-1}\bm{A}$ lie between $-\alpha$ and $\beta$. Then, 
\[
(1 - \beta)\bm{D} \preceq \bm{D} -\bm{A} \preceq (1 + \alpha)\bm{D}.\] 
\end{claim}
\begin{proof}
See Appendix. 
\end{proof}

Combining the above results, gives
\[\left[1 - \left(1 - \frac{1}{\kappa}\right)^{2^d}\right]\bm{D}_d\preceq \bm{D}_d - \bm{A}_d\preceq\left[1 + \left(1 - \frac{1}{\kappa}\right)^{2^d}\right]\bm{D}_d. \]
Hence, to guarantee that $\bm{D}_d\approx_{\epsilon_d}\bm{D}_d - \bm{A}_d$, the following system must be satisfied: 
\begin{align*}
e^{-\epsilon_d} &\le 1 - \left(1 - \frac{1}{\kappa}\right)^{2^d},  \ \ \ \ \text{and} \ \ \ \ e^{\epsilon_d} \ge 1 + \left(1 - \frac{1}{\kappa}\right)^{2^d}.
\end{align*}
Introducing $\gamma$ for $\left(1 - \frac{1}{\kappa}\right)^{2^d}$, we arrive at: 
\begin{align*}
\epsilon_d &\ge \ln\left(\frac{1}{1 - \gamma}\right),  \ \ \ \ \text{and} \ \ \ \ \epsilon_d \ge \ln(1 + \gamma).
\end{align*}
Hence, $\epsilon_d \ge \max\left\lbrace \ln\left(\frac{1}{1 - \gamma}\right), \ln(1 + \gamma)\right\rbrace = \ln\left(\frac{1}{1 - \gamma}\right)$. 
Now, notice that if $d = \lceil \log c \kappa\rceil$ then,
$\gamma = \left(1 - \frac{1}{\kappa}\right)^{2^d} =  \left(1 - \frac{1}{\kappa}\right)^{{c}\kappa} \le \frac{1}{e^c}$. 
Hence, $\ln\left(\frac{1}{1 - \gamma}\right) \le \ln\left(\frac{e^c}{e^c - 1}\right)$. This gives $c = \lceil 2\ln\left(\frac{\sqrt[3]{2}}{\sqrt[3]{2} - 1}\right)\rceil$, implying $\epsilon_d = \ln\left(\frac{e^c}{e^c - 1}\right) < \frac{1}{3}\ln2$.
\end{proof}

Using the above results the time complexity of $\text{DistrESolve}$ with $d =  \lceil \log \left(2\ln\left(\frac{\sqrt[3]{2}}{\sqrt[3]{2} - 1}\right)\kappa\right)\rceil$ is $\mathcal{O}\left(n^2\log\kappa \log(\frac{1}{\epsilon})\right)$ times steps, which concludes the proof of Theorem \ref{Main_Theorem}.

\subsection{R-Hop Distributed Solvers}
The above version of the distributed solver requires no knowledge of the graph's topology, but does require the information from all other nodes. Next, we will outline an R-Hop version of the algorithm in which communication is restricted to the R-Hop neighborhood between nodes. Due to such communication constraints, the R-Hop solver is general enough to be applied in a variety of fields including but not limited to, network flow problems (see Section~\ref{Sec:Newton}). Along with Theorem~\ref{Main_TheoremTwo}, the following corollary summarizes the results of the R-Hop distributed solver: 

\begin{corollary}\label{CorollaryTwo}
Let $\bm{M}_0$ be the weighted Laplacian of $\mathcal{G} = \left(\mathcal{V},\mathcal{E},\bm{W}\right)$. There exists a decentralized algorithm that uses only $R$-hop communication between nodes and computes $\epsilon$ close solutions of $\bm{M}_{0}\bm{x}=\bm{b}_{0}$ in
 $\mathcal{O}\left(\frac{n^3\alpha}{R}\frac{\bm{W}_{\text{max}}}{\bm{W}_{\text{min}}}\log(\frac{1}{\epsilon})\right)$ time steps, with $n$ being the number of nodes in $\mathcal{G}$, $\bm{W}_{\text{max}}, \bm{W}_{\text{min}}$ denoting the largest and the smallest weights of edges in $\mathcal{G}$, respectively, $\alpha = \min\left\{n, \frac{\left(d^{R+1}_{\text{max}} - 1\right)}{\left(d_{\text{max}} - 1\right)}\right\}$ representing the upper bound on the size of the R-hop neighborhood $\forall \bm{v}\in \mathcal{V}$, and $\epsilon\in (0, \frac{1}{2}]$ being the precision parameter. 
\end{corollary}

Similar to development of the full communication solver, the R-Hop version also requires two steps to attain the $\epsilon$-close approximation to $\bm{x}^{\star}$, i.e., the ``crude R-Hop'' and the ``exact R-Hop'' solutions. 

\begin{algorithm}[t!]
  \caption{$\text{RDistRSolve}\left(\mathcal{C}, [\bm{b}_0]_k, d, R\right)$}
  \label{Algo:DistRHop}
  \begin{algorithmic}[1]
	\STATE \textbf{Part One:}	
%
	\STATE $\{[\bm{C}_0]_{k1},\ldots,[\bm{C}_0]_{kn}\} = \text{f}_0\left([\bm{M}_0]_{k1},\ldots, [\bm{M}_0]_{kn}, R\right)$
	\STATE $\{[\bm{C}_1]_{k1},\ldots,[\bm{C}_1]_{kn}\} = \text{f}_1\left([\bm{M}_0]_{k1},\ldots, [\bm{M}_0]_{kn}, R\right)$
\\\hrulefill	
	\STATE \textbf{Part Two:}
	\FOR {$i=1$ to $d$}
	\STATE $l_{i-1} = \sfrac{2^{i-1}}{R}$
	\STATE $[\bm{u}^{(i-1)}_{1}]_k = [\bm{A}_0\bm{D}^{-1}_0\bm{b}_{i-1}]_k$
	\STATE $[\bm{u}^{(i-1)}_{l_{i-1}}]_k = \text{f}_2([\bm{u}^{(i-1)}_{1}]_k)$
			\STATE $[\bm{b}_i]_k = [\bm{b}_{i-1}]_k + [\bm{u}^{(i-1)}_{l_{i-1}}]_k$\\
		
	\ENDFOR 
		\\\hrulefill
	\STATE \textbf{Part Three:}
	\STATE $[\bm{x}_d]_k = \sfrac{[\bm{b}_d]_k}{[\bm{D}_0]_{kk}}$
	\FOR {$i=d-1$ to $1$}
	\STATE $l_i = \sfrac{2^i}{R}$
	\STATE $[\bm{\eta}^{(i+1)}_{1}]_k = [\bm{D}^{-1}_0\bm{A}_0\bm{x}_{i+1}]_k$
	\STATE $\left[\bm{\eta}^{i+1}_{l_i}\right]_k  = \text{f}_3([\bm{\eta}^{(i+1)}_{1}]_k)$

%

			\STATE $[\bm{x}_i]_k = \frac{1}{2}\left[\frac{[\bm{b}_i]_k}{[\bm{D}_0]_{kk}}  + [\bm{x}_{i+1}]_k + [\bm{\eta}^{i+1}_{l_i}]_k \right]$\\
	\ENDFOR 
	\STATE $[\bm{x}_0]_k = \frac{1}{2}\left[\frac{[\bm{b}_0]_k}{[\bm{D}_0]_{kk}} + [\bm{x}_1]_k + [\bm{D}^{-1}_0\bm{A}_0\bm{x}_1]_k \right]$
    \STATE \textbf{return} $[\bm{x}_0]_k$ 
  \end{algorithmic}
\end{algorithm}

\subsubsection{``Crude'' R-Hop Distributed Solver}
The ``crude R-Hop'' solver uses the same inverse approximated chain as that of the full communication version (see Equation~\ref{Eq:InverseChain}) to acquire a ``crude'' approximation for the $k^{th}$ component $\bm{x}_{0}$ while only requiring R-Hop communication between the nodes. Algorithm~\ref{Algo:DistRHop} represents the ``crude'' R-Hop solver requiring the inverse chain, $k^{th}$ component of $\bm{b}_{0}$, length of the inverse chain $d$, and the communication bound $R$ as inputs. Namely, each node $\bm{v}_{k} \in \mathcal{V}$ receives the $k^{th}$ row of $\bm{M}_{0}$ , $k^{th}$ component, $[\bm{b}_{0}]_{k}$, of $\bm{b}_{0}$, the length of the inverse chain, $d$, and the local communication bound\footnote{For simplicity, $R$ is assumed to be in the order of powers of 2, i.e., $R=2^{\rho}$.}  $R$ as inputs to output the $k^{th}$ component of the ``crude'' approximation to $\bm{x}^{\star}$.  Algorithm~\ref{Algo:DistRHop} operates in three major parts. Due to the need of the R-powers of $\bm{A}_{0}\bm{D}_{0}^{-1}$ and $\bm{D}_{0}\bm{A}_{0}^{-1}$, the first step is to compute such matrices in a distributed manner.

\begin{algorithm}[t!]
  \caption{$\text{f}_0\left([\bm{M}_0]_{k1},\ldots, [\bm{M}_0]_{kn}, R\right)$}
  \label{Alg_4}
  \begin{algorithmic}[1]
	\FOR {$l = 1$ to $R - 1$}
		\FOR {$j$ s.t.$\bm{v}_j\in \mathbb{N}_{l+1}(\bm{v}_k)$}
			\STATE
			$\left[(\bm{A}_0\bm{D}^{-1}_0)^{l+1}\right]_{kj} = \sum\limits_{r:\bm{v}_r\in \mathbb{N}_1(v_j)}\frac{[\bm{D}_0]_{rr}}{[\bm{D}_0]_{jj}}[(\bm{A}_0\bm{D}^{-1}_0)^l]_{kr}[\bm{A}_0\bm{D}^{-1}_0]_{jr}$
		\ENDFOR
	\ENDFOR
	\STATE \textbf{return} $\bm{c}_0 = \{[(\bm{A}_0\bm{D}^{-1}_0)^{R}]_{k1},\ldots,[(\bm{A}_0\bm{D}^{-1}_0)^{R}]_{kn} \}$
  \end{algorithmic}
\end{algorithm}

Given the inverse chain and the communication bound, R, $\text{f}_{0}(\cdot)$ and $\text{f}_{1}(\cdot)$ serve this cause as detailed in Algorithms~\ref{Alg_4} and~\ref{Alg_5}, respectively. Essentially, these algorithms execute  multiplications needed for determining $\left(\bm{A}\bm{D}_{0}^{-1}\right)^{R}$ and $\left(\bm{D}\bm{A}_{0}^{-1}\right)^{R}$ in a distributed fashion looping over the relevant hops of the network. For a node, $\bm{v}_{k} \in \mathcal{V}$, the $k^{th}$ component of these powers are returned to Algorithm~\ref{Algo:DistRHop} as $[\bm{C}_{0}]_{ki}=\left[\left(\bm{A}_{0}\bm{D}_{0}^{-1}\right)^{R}\right]_{ki}$ and $[\bm{C}_{1}]_{ki}=\left[\left(\bm{D}_{0}\bm{A}_{0}^{-1}\right)^{R}\right]_{ki}$ for $i=\{1,\dots, n\}$; see Part One in Algorithm~\ref{Algo:DistRHop}.

Similar to the full communication version, the second two parts of the ``crude R-Hop'' solver run two loops. In the first, the $k^{th}$ component of $\bm{b}_{i}$ is computed by looping forward through the inverse chain, while in the second the $k^{th}$ component of the crude solution is determined by looping backwards. 
\begin{algorithm}[t!]
  \caption{$\text{f}_1([\bm{M}_0]_{k1},\ldots, [\bm{M}_0]_{kn}, R)$}
  \label{Alg_5}
  \begin{algorithmic}[1]
	\FOR {$l = 1$ to $R - 1$}
		\FOR {$j$ s.t.$\bm{v}_j\in \mathbb{N}_{l+1}(\bm{v}_k)$}
			\STATE$
			\left[(\bm{D}^{-1}_0\bm{A}_0)^{l+1}\right]_{kj} = \sum\limits_{r:\bm{v}_r\in \mathbb{N}_1(\bm{v}_j)}\frac{[\bm{D}_0]_{jj}}{[\bm{D}_0]_{rr}}[(\bm{D}^{-1}_0\bm{A}_0)^l]_{kr}[\bm{D}^{-1}_0\bm{A}_0]_{jr}$
		\ENDFOR
	\ENDFOR
	\STATE \textbf{return} $\bm{c}_1 = \{[(\bm{D}^{-1}_0\bm{A}_0)^{R}]_{k1},\ldots,[(\bm{D}^{-1}_0\bm{A}_0)^{R}]_{kn} \}$
  \end{algorithmic}
\end{algorithm}

The second part of the solver is better depicted in the flow diagrams of Figures~\ref{Fig:FlowDiagramOne} and~\ref{Fig:FlowDiagramTwo}. Within the first loop running through the length of the inverse chain, the condition $i-1 < \rho$ is checked. In case this condition is true, $\bm{A}_{0}$, $\bm{D}_{0}$, and the previous iteration vector $\bm{b}_{i-1}$ are used to update $\bm{b}_{i}$ as shown in Figure~\ref{Fig:FlowDiagramOne}. 

\begin{equation*}
\left[\bm{b}_{i}\right]_{k}=\left[\bm{b}_{i-1}\right]_{k}+\left[\bm{u}_{2^{i-1}}^{(i-1)}\right]_{k}
\end{equation*}

This is performed using another loop constructing a series of $[\bm{u}_{j}^{(i-1)}]_{k}$
vectors for $j=1,\dots, 2^{i-1}$, used to update the $k^{th}$ component of $\bm{b}_{i}$ at the $i^{th}$ iteration:
At the next $i^{th}$ iteration, the condition $i-1 < \rho$ is checked again. In case this condition is met, the previous computations are executed again. Otherwise, the commands depicted in Figure~\ref{Fig:FlowDiagramTwo} run. Here, a temporary variable denoting the fraction to the communication bound $R$, $l_{i-1}=\frac{2^{i-1}}{R}$, is used to determine the upper iterate bound in $j$. Again throughout this loop, a series of $\bm{u}$ vectors used to update $[\bm{b}_{i}]_{k}$ are constructed. 

Having terminated the forward loop, the $k^{th}$ component of the ``crude'' solution is computed by looping backward through the inverse approximated chain (see part three in Algorithm~\ref{Algo:DistRHop}). For $i$ running backwards to $1$, an R-Hop condition, $\rho$, is checked. In case $i < \rho$, the following update is performed:
\begin{equation*}
[\bm{\eta}_{j}^{(i+1)}]_{k}=\left[\bm{D}^{-1}_{0}\bm{A}_{0}\bm{\eta}_{j-1}^{(i+1)}\right]_{k},
\end{equation*}
for $j=2,\dots, 2^{i}$, $i=d-1,\dots, 1$, and 
\begin{equation*}
\left[\bm{\eta}_{1}^{(i+1)}\right]_{k}=\left[\bm{D}^{-1}_{0}\bm{A}_{0}\bm{x}_{i+1}\right]_{k}. 
\end{equation*}
\begin{figure*}[t!]
\centering
\vspace{-.5em}
\subfigure[]{
	\label{Fig:FlowDiagramOne}
\includegraphics[width=0.43\textwidth]{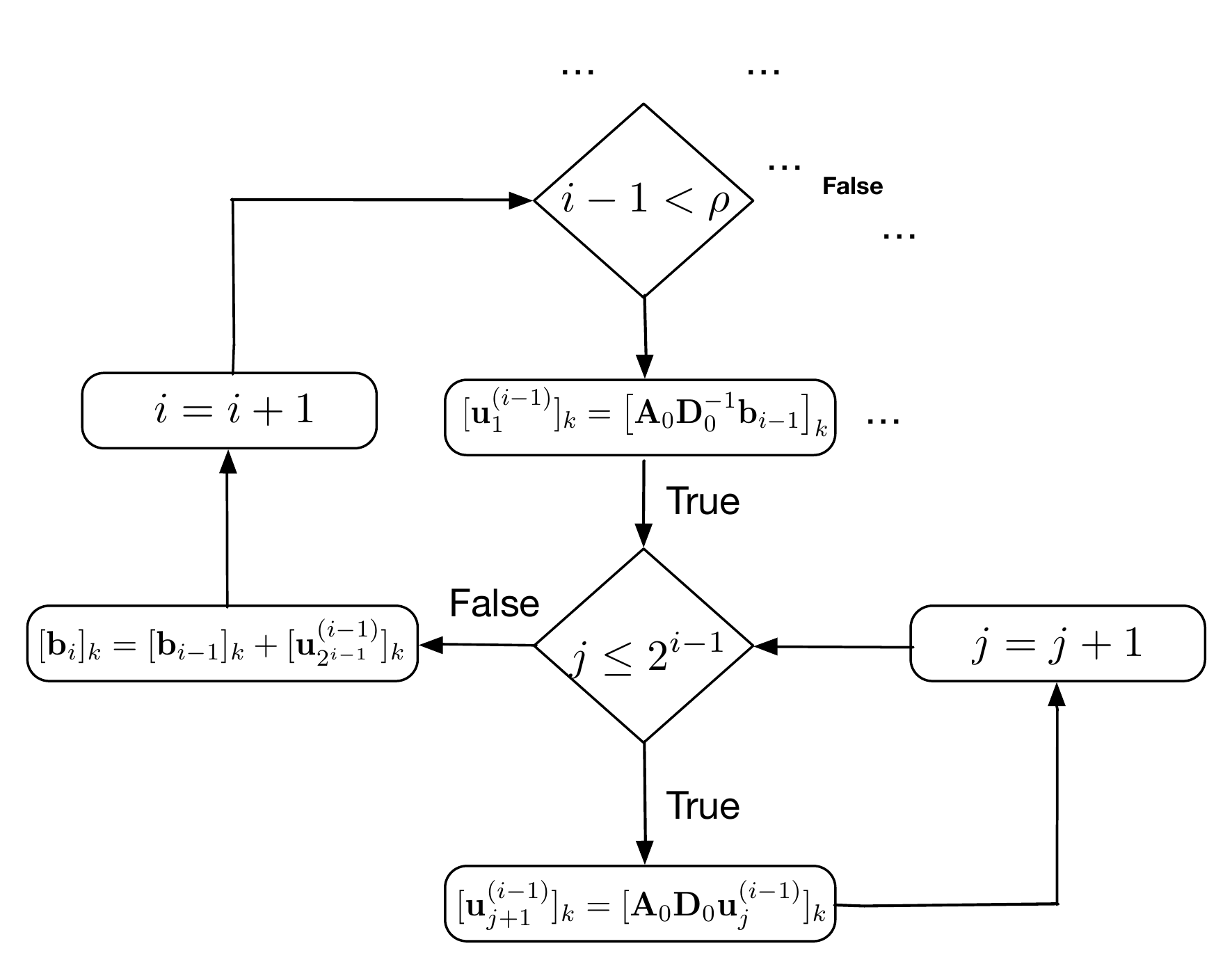}
}
\hfill\hspace{-.3in}\hfill
\subfigure[]{
	\label{Fig:FlowDiagramTwo}
\includegraphics[width=0.5\textwidth]{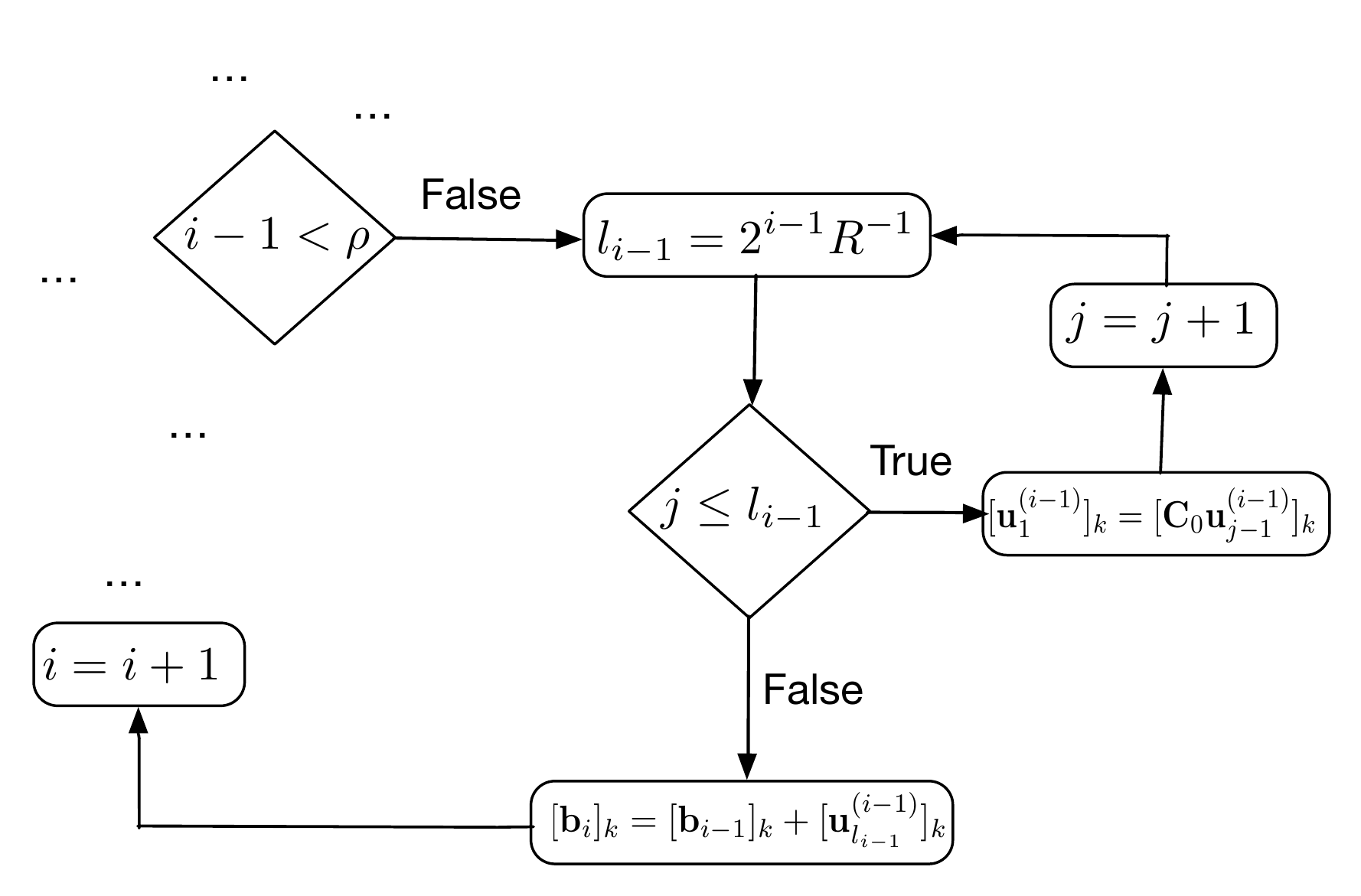}
}
\caption{The left figure depicts the first condition of part two in Algorithm~\ref{Algo:DistRHop} which has to be checked. In case $i-1 < \rho$, the computations shown in the figure are executed. The right figure handles the case when $i-1 \geq \rho$. Here, the computations shown in the figure are executed. Note, that the chains are constructed using the computations returned by $\text{f}_{0}(\cdot)$ and $\text{f}_{1}(\cdot)$.}
\end{figure*}



The role of $\bm{\eta}$ are backward intermediate solutions needed for updating the ``crude'' solution to $\bm{M}_{0}\bm{x}=\bm{b}_{0}$: 
\begin{equation*}
[\bm{x}_{i}]_{k}=\frac{1}{2}\left[\frac{[\bm{b}_{i}]_{k}}{[\bm{D}_{0}]_{kk}}+[\bm{x}_{1}]_{k}+[\bm{\eta}^{(i+1)}_{2^{i}}]\right]_{k},
\end{equation*}
for a node $\bm{v}_{k} \in \mathcal{V}$. In case $i \geq \rho$ a similar set of computations are executed for updating the crude solution using: 
\begin{equation*}
[\bm{x}_{i}]_{k}=\frac{1}{2}\left[\frac{[\bm{b}_{i}]_{k}}{[\bm{D}_{0}]_{kk}}+[\bm{x}_{1}]_{k}+[\bm{D}_{0}^{-1}\bm{A}\bm{x}_{1}]\right]_{k}.
\end{equation*}

\textbf{Analysis of Algorithm~\ref{Algo:DistRHop}} Similar to the previous section, we next provide the theoretical analysis needed for quantifying the performance of the crude R-Hop solver. The following Lemma shows that $\text{RDistRSolve}$ computes the $k^{th}$ component of the ``crude'' approximation of $\bm{x}^{\star}$ and provides the algorithm's time complexity:
 
\begin{lemma}\label{r_hop_rude_lemma}
Let $\bm{M}_0 = \bm{D}_0 - \bm{A}_0$ be the standard splitting and let $\bm{Z}^{\prime}_0$ be the operator defined by $\text{RDistRSolve}$, namely, $\bm{x}_0 = \bm{Z}^{\prime}_0\bm{b}_0$, then
$\bm{Z}^{\prime}_0\approx_{\epsilon_d} \bm{M}^{-1}_0$.  
Furthermore, $\text{RDistRSolve}$ requires 
\[\mathcal{O}\left(\frac{2^d}{R}\alpha + \alpha Rd_{max}\right),\] where $\alpha = \min\left\{n, \frac{\left(d^{R+1}_{\text{max}} - 1\right)}{\left(d_{\text{max}} - 1\right)}\right\}$ to arrive at $\bm{x}_{0}$. 
\end{lemma}

\begin{proof}
The proof of the above Lemma can be obtained by proving a collection of claims:
\begin{claim}\label{claim_1}
Matrices $\left(\bm{D}^{-1}_0\bm{A}_0\right)^r$ and $\left(\bm{A}_0\bm{D}^{-1}_0\right)^{r}$ have sparsity patterns corresponding to the $r$-Hop neighborhood for any $r\in \mathbb{N}$.
\end{claim}

The above claim is proved by induction on $R$. We start with the base case: for $R = 1$,
\begin{align*}
[\bm{A}_0\bm{D}^{-1}_0]_{ij} =
\begin{cases}
 &\frac{[\bm{A}_0]_{ij}}{[\bm{D}_0]_{ii}} \ \ \  \text{if } j: \bm{v}_j\in \mathbb{N}_{1}(\bm{v}_i) \\
 &0 \ \ \ \text{otherwise.}
\end{cases}
\end{align*}
Therefore, $\bm{A}_0\bm{D}^{-1}_0$ has a sparsity pattern corresponding to the $1$-Hop neighborhood. 
Assume that for all $1\le p\le R-1$, $\left(\bm{A}_0\bm{D}^{-1}_0\right)^p$ has a sparsity pattern corresponding to the $p-hop$ neighborhood. Consider, $\left(\bm{A}_0\bm{D}^{-1}_0\right)^{r}$
\begin{equation}\label{sparisty_lemma_2}
[(\bm{A}_0\bm{D}^{-1}_0)^{R}]_{ij} = \sum_{k=1}^{n}[(\bm{A}_0\bm{D}^{-1}_0)^{R-1}]_{ik}[\bm{A}_0\bm{D}^{-1}_0]_{kj}
\end{equation}
Since $\bm{A}_0\bm{D}^{-1}_0$ is non negative, then $[(\bm{A}_0\bm{D}^{-1}_0)^{R}]_{ij} \ne 0$ iff there exists $k$ such that $\bm{v}_k\in \mathbb{N}_{R-1}(\bm{v}_i)$ and $\bm{v}_k\in \mathbb{N}_1(\bm{v}_j)$, namely, $\bm{v}_j\in \mathbb{N}_{R}(\bm{v}_i)$. The proof can be done in a similar fashion for $\bm{D}^{-1}_0\bm{A}_0$.
\end{proof}

The next claim provides complexity guarantees for $\text{f}_{0}(\cdot)$ and $\text{f}_{1}(\cdot)$ described in Algorithms~\ref{Alg_4} and~\ref{Alg_5}, respectively. 
\begin{claim}\label{claim_2}
Algorithms~\ref{Alg_4} and~\ref{Alg_5} use only the R-hop information to compute the $k^{th}$ row of $\left(\bm{D}^{-1}_0\bm{A}_0\right)^{R}$ and $\left(\bm{A}_0\bm{D}^{-1}_0\right)^{R}$, respectively, in $\mathcal{O}\left(\alpha Rd_{max}\right)$ time steps, where $\alpha = \min\left\{n, \frac{\left(d^{R+1}_{\text{max}} - 1\right)}{\left(d_{\text{max}} - 1\right)}\right\}$.  
\end{claim}
\begin{proof}
The proof will be given for $\text{f}_{0}(\cdot)$ described in Algorithm~\ref{Alg_4} as that for $\text{f}_{1}(\cdot)$ can be performed similarly. Due to Claim~\ref{claim_1}, we have
\begin{align}\label{eq_1}
\left[\left(\bm{A}_0\bm{D}^{-1}_0\right)^{l+1}\right]_{kj} &= \sum\limits_{r=1}^{n}\left[\left(\bm{A}_0\bm{D}^{-1}_0\right)^{l}\right]_{kr}\left[\bm{A}_0\bm{D}^{-1}_0\right]_{rj} = \sum\limits_{r: \bm{v}_r\in \mathbb{N}_1(\bm{v}_j)}\left[\left(\bm{A}_0\bm{D}^{-1}_0\right)^{l}\right]_{kr}\left[\bm{A}_0\bm{D}^{-1}_0\right]_{rj}
\end{align} 
Therefore at iteration $l+1$, $\bm{v}_k$ computes the $k^{th}$ row of $\left(\bm{A}_0\bm{D}^{-1}_0\right)^{l+1}$ using: 
\begin{enumerate}
\item[] (1) the $k^{th}$ row of $(\bm{A}_0\bm{D}^{-1}_0)^{l}$, and
\item[] (2) the $r^{th}$ column of $\bm{A}_0\bm{D}^{-1}_0$.
\end{enumerate}
Node $\bm{v}_{r}$, however, can only send the $r^{th}$ row of $\bm{A}_{0}\bm{D}^{-1}_{0}$ making $\bm{A}_{0}\bm{D}^{-1}_{0}$ non-symmetric. Noting that
$\sfrac{[\bm{A}_0\bm{D}^{-1}_0]_{rj}}{[\bm{D}_0]_{rr}} = \sfrac{[\bm{A}_0\bm{D}^{-1}_0]_{jr}}{[\bm{D}_0]_{jj}}$, since $\bm{D}^{-1}_0\bm{A}_0\bm{D}^{-1}$ is symmetric, leads to 
$[(\bm{A}_0\bm{D}^{-1}_0)^{l+1}]_{kj} = 
\sum\limits_{r: \bm{v}_r\in \mathbb{N}_1(\bm{v}_j)}\frac{[\bm{D}_0]_{rr}}{[\bm{D}_0]_{jj}}[(\bm{A}_0\bm{D}^{-1}_0)^{l}]_{kr}[\bm{A}_0\bm{D}^{-1}_0]_{jr}$. 
To prove the time complexity guarantee, at each iteration $\bm{v}_k$ computes at most $\alpha$ values, where $\alpha  = \min\left\{n, \frac{\left(d^{R+1}_{\text{max}} - 1\right)}{\left(d_{\text{max}} - 1\right)}\right\}$ is the upper bound on the size of the R-hop neighborhood $\forall \bm{v}\in \mathcal{V}$.  Each such computation requires at most $\mathcal{O}(d_{max})$ operations. Thus, the overall time complexity is given by $\mathcal{O}(\alpha Rd_{max})$.
\end{proof}

We are now ready to provide the proof of Lemma \ref{r_hop_rude_lemma}. 
\begin{proof}
From \textbf{Parts Two} and \textbf{Three} of Algorithm~\ref{Algo:DistRHop}, it is clear that node $\bm{v}_k$ computes $[\bm{b}_1]_k,[\bm{b}_2]_k,\ldots, [\bm{b}_d]_k$ and $[\bm{x}_d]_k, [\bm{x}_{d-1}]_k,\ldots, [\bm{x}_0]_k$, respectively. These are determined using the inverse approximated chain as follows
\begin{align}\label{eq_4}
\bm{b}_i &= (\bm{I} + (\bm{A}_{i-1}\bm{D}^{-1}_{i-1})\bm{b}_{i-1}  
=\bm{b}_{i-1} + (\bm{A}_0\bm{D}^{-1}_0)^{2^{i-1}}\bm{b}_{i-1} \\ \nonumber
\bm{x}_{i} &= \frac{1}{2}[\bm{D}^{-1}_i\bm{b}_i + (\bm{I} +\bm{D}^{-1}_i\bm{A}_i)x_{i+1}] =\frac{1}{2}[\bm{D}^{-1}_0\bm{b}_i + \bm{x}_{i+1} + (\bm{D}^{-1}_0\bm{A}_0)^{2^i}\bm{x}_{i+1}]
\end{align}

Considering the computation of $[\bm{b}_1]_k,\ldots, [\bm{b}_d]_k$ for $\rho> i - 1$, we have
\begin{align*}
[\bm{b}_{i}]_k &= [\bm{b}_{i-1}]_k + [(\bm{A}_0\bm{D}^{-1}_0)^{2^{i-1}}\bm{b}_{i-1}]_k = [\bm{b}_{i-1}]_k + [\underbrace{\bm{A}_0\bm{D}^{-1}_0 \ldots \bm{A}_0\bm{D}^{-1}_0}_{2^{i-1}}\bm{b}_{i-1}]_k \\
&= [\bm{b}_{i-1}]_k + [\underbrace{\bm{A}_0\bm{D}^{-1}_0 \ldots \bm{A}_0\bm{D}^{-1}_0}_{2^{i-1}-1}\bm{u}^{(i-1)}_1]_k \dots = [\bm{b}_{i-1}]_k + \left[\bm{u}^{(i-1)}_{2^{i-1}}\right]_k
\end{align*}
\text{with $\bm{u}^{(i-1)}_{j+1} = \bm{A}_0\bm{D}^{-1}_0\bm{u}^{(i-1)}_{j}$ for $j = 1,\ldots 2^{i-1}-1$.} Since $\bm{A}_0\bm{D}^{-1}_0$ has a sparsity pattern corresponding to 1-hop neighborhood (see Claim~\ref{claim_1}),  node $\bm{v}_k$ computes $\left[\bm{u}^{(i-1)}_{j+1}\right]_k$, based on $\bm{u}^{(i-1)}_{j}$, acquired from its 1-Hop neighbors. It is easy to see that $ \forall i \ \text{such that} \ i - 1 < \rho$ the computation of $[\bm{b}_i]_k$ requires $\mathcal{O}\left(2^{i-1}d_{max}\right)$ time steps. Thus, the computation of $[\bm{b}_1]_k, \ldots, [\bm{b}_{\rho}]_k$ requires $\mathcal{O}(2^{\rho}d_{\text{max}}) = \mathcal{O}(Rd_{\text{max}})$. Now, consider the computation of $[\bm{b}_i]_k$ but for $i - 1 \geq \rho$ 
\begin{align*}
[\bm{b}_{i}]_k &= [\bm{b}_{i-1}]_k + [(\bm{A}_0\bm{D}^{-1}_0)^{2^{i-1}}\bm{b}_{i-1}]_k  =[\bm{b}_{i-1}]_k + [\underbrace{\bm{C}_0\ldots \bm{C}_0}_{l_{i-1}}\bm{b}_{i-1}]_k \\
&=[\bm{b}_{i-1}]_k + [\underbrace{\bm{C}_0 \ldots \bm{C}_0}_{l_{i-1}-1}\bm{u}^{(i-1)}_1]_k  =[\bm{b}_{i-1}]_k + \left[\bm{u}^{(i-1)}_{l_{i-1}}\right]_k
\end{align*}
with $\bm{C}_0 = (\bm{A}_0\bm{D}^{-1}_0)^R$,  $l_{i-1} = \frac{2^{i-1}}{R}$, and $\bm{u}^{(i-1)}_{j+1} = \bm{C}_0\bm{u}^{(i-1)}_{j}$ for $j=1,\ldots, l_{i-1} - 1$.
Since $\bm{C}_0$ has a sparsity pattern corresponding to R-hop neighborhood (see Claim~\ref{claim_1}), node $\bm{v}_k$ computes $[\bm{u}^{(i-1)}_{j+1}]_k$ based on the components of $\bm{u}^{(i-1)}_{j}$ attained from its R-hop neighbors. For each $i$ such that $i - 1 \geq \rho $ the computing $[\bm{b}_i]_k$ requires $\mathcal{O}\left(\frac{2^{i-1}}{R}\alpha\right)$ time steps, where $\alpha = \min\left\{n, \frac{\left(d^{R+1}_{\text{max}} - 1\right)}{\left(d_{\text{max}} - 1\right)}\right\}$ being the upper bound on the number of nodes in the $R-$ hop neighborhood $\forall \ \bm{v}\in \mathcal{V}$. Therefore, the  overall computation of $[\bm{b}_{\rho + 1}]_k, [\bm{b}_{\rho + 2}]_k, \ldots, [\bm{b}_d]_k$ is achieved in $\mathcal{O}\left(\frac{2^d}{R}\alpha\right)$ time steps. Finally, the time complexity for the computation of all of the values $[\bm{b}_1]_k,[\bm{b}_2]_k,\ldots, [\bm{b}_d]_k$ is $\mathcal{O}\left(\frac{2^d}{R}\alpha + Rd_{max}\right)$. Similar analysis can be applied to determine the computational complexity of $[\bm{x}_d]_k,[\bm{x}_{d-1}]_k,\ldots, [\bm{x}_1]_k$, i.e., \textbf{Part Three} of Algorithm~\ref{Algo:DistRHop}. We arrive at
$\bm{Z}^{\prime}_0\approx_{\epsilon_d} \bm{M}^{-1}_0$. 
Finally, using Claim \ref{claim_2}, the time complexity of $\text{RDistRSolve}$ (Algorithm~\ref{Algo:DistRHop}) is $\mathcal{O}\left(\frac{2^d}{R}\alpha + \alpha Rd_{max}\right)$.
\end{proof} 
\subsubsection{``Exact'' R-Hop Distributed Solver}
Having developed an R-hop version which computes a ``crude'' approximation to the solution of $\bm{M}_{0}\bm{x}=\bm{b}_{0}$, now we provide an exact R-hop solver presented in Algorithm~\ref{Alg_ExactRHop}. Similar to $\text{RDistRSolve}$, each node $\bm{v}_k$ receives the $k^{th}$ row $\bm{M}_0$, $[\bm{b}_0]_k$, $d$, $R$, and a precision parameter $\epsilon$ as inputs, and outputs the $k^{th}$ component of the $\epsilon$ close approximation of vector $\bm{x}^{\star}$. 

\begin{algorithm}\label{Algo:EDistR}
  \caption{ \hspace{-.5em} $\text{EDistRSolve}\left(\{[\bm{M}_0]_{k1},\ldots, [\bm{M}_0]_{kn}\}, [\bm{b}_0]_k, d, R,\epsilon\right)$}
  \label{Alg_ExactRHop}
  \begin{algorithmic}
\STATE \textbf{Initialize}: $[\bm{y}_0]_k = 0$, and $[\bm{\chi}]_k = \text{RDistRSolve}(\{[M_0]_{k1},\ldots, [M_0]_{kn}\}, [b_0]_k, d, R)$	
	\FOR {$t=1$ to $q$}
		\STATE  \hspace{-2em} $[\bm{u}^{(1)}_{t}]_k = [\bm{D}_0]_{kk}[\bm{y}_{t-1}]_k - \sum_{j: \bm{v}_j\in \mathbb{N}_{1}(\bm{v}_k)}[\bm{A}_{0}]_{kj}[\bm{y}_{t-1}]_j$
		 \STATE \hspace{-2em} $[\bm{u}^{(2)}_{t}]_k = \text{RDistRSolve}(\{[\bm{M}_0]_{k1},\ldots, [\bm{M}_0]_{kn}\}, [\bm{u}^{(1)}_{t}]_k, d, R)$
		\STATE  \hspace{-2em} $[\bm{y}_t]_k = [\bm{y}_{t-1}]_k - [\bm{u}^{(2)}_{t}]_k + [\bm{\chi}]_k$ 
	\ENDFOR 
    \STATE \textbf{return} $[\tilde{\bm{x}}]_k = [\bm{y}_q]_k$
  \end{algorithmic}
\end{algorithm}

\textbf{Analysis of Algorithm~\ref{Alg_ExactRHop}:} The following Lemma shows that $\text{EDistRSolve}$ computes the $k^{th}$ component of the $\epsilon$ close approximation to $\bm{x}^{\star}$ and provides the time complexity analysis. 

\begin{lemma}\label{Dist_Exact_algorithm_guarantee_lemma_Bla}
Let $\bm{M}_0 = \bm{D}_0 - \bm{A}_0$ be the standard splitting. Further, let $\epsilon_d < \sfrac{1}{3}\ln2$. Then Algorithm~\ref{Alg_ExactRHop} requires $\mathcal{O}\left(\log\frac{1}{\epsilon}\right)$ iterations to return the $k^{th}$ component of the $\epsilon$ close approximation to $\bm{x}^{\star}$.  
\end{lemma}

The following Lemma provides the time complexity analysis of $\text{EDistRSolve}$:

\begin{lemma}\label{time_complexity_of_distresolve_Bla}
Let $\bm{M}_0 = \bm{D}_0 - \bm{A}_0$ be the standard splitting and let $\epsilon_d < \sfrac{1}{3}\ln 2 $, then $\text{EDistRSolve}$ requires 
$\mathcal{O}\left(\left(\sfrac{2^d}{R}\alpha + \alpha Rd_{max}\right)\log\left(\sfrac{1}{\epsilon}\right)\right)$ time steps. Moreover, for each node $\bm{v}_k$, $\text{EDistRSolve}$ only uses information from the R-hop neighbors.
\end{lemma}


\textbf{Length of the Inverse Chain:} Again these introduced algorithms depend on the length of the inverse approximated chain, $d$. The analysis in Section~\ref{Sec:Length} can be applied again to determine the $d =  \lceil \log \left(2\ln\left(\frac{\sqrt[3]{2}}{\sqrt[3]{2} - 1}\right)\kappa\right)\rceil$ as the length of the inverse chain. 



\subsection{Comparison to Existing Literature}
As mentioned before, the proposed solver is a distributed version of the parallel SDDM solver of~\cite{c11}. Our approach is capable of acquiring $\epsilon$-close solutions for arbitrary $\epsilon>0$ in $\mathcal{O}\left(n^{3}\frac{\bm{\alpha}}{R}\frac{\bm{W}_{\text{max}}}{\bm{W}_{\text{min}}}\log\left(\frac{1}{\epsilon}\right)\right)$, with $n$ the number of nodes in graph $\mathcal{G}$, $\bm{W}_{\text{max}}$ and $\bm{W}_{\text{min}}$ denoting the largest and smaller weights of the edges in $\mathcal{G}$, respectively, $\bm{\alpha}=\min\left\{n,\frac{d_{\text{max}}^{R+1}-1}{d_{\text{max}}-1}\right\}$ representing the upper bound on the size of the R-Hop neighborhood $\forall \bm{v} \in \mathcal{V}$, and $\epsilon \in (0,\frac{1}{2}]$ as the precision parameter. After developing the full communication version, we proposed a generalization to the R-Hop case where communication is restricted. 

Our method is faster than state-of-the-art methods for iteratively solving linear systems. Typical linear methods, such as Jacobi iteration~\cite{c16}, are guaranteed to converge if the matrix is \emph{strictly} diagonally dominant. We proposed a distributed algorithm that generalizes this setting, where it is guaranteed to converge in the SDD/SDDM scenario. Furthermore, the time complexity of linear techniques is $\mathcal{O}(n^{1+\beta} \log n)$, hence, a case of strictly diagonally dominant matrix $\bm{M}_{0}$ can be easily constructed to lead to a complexity of $\mathcal{O}(n^{4}\log n)$. Consequently, our approach not only generalizes the assumptions made by linear methods, but is also faster by a factor of $\log n$. Furthermore, such algorithms require average consensus to decentralize vector norm computations. Contrary to these methods which lead to additional approximation errors to the real solution, our approach resolves these issues by eliminating the need for such a consensus framework. 

In centralized solvers, nonlinear methods (e.g., conjugate gradient descent~\cite{c37,c36}, etc.) typically offer computational advantages over linear methods (e.g., Jacobi Iteration) for iteratively solving linear systems. These techniques, however, can not be easily decentralized. For instance, the stopping criteria for nonlinear methods require the computation of weighted norms of residuals (e.g., $||\bm{p}_{k}||_{\bm{M}_{0}}$ with $\bm{p}_{k}$ being the search direction at iteration $k$). To the best of our knowledge, the distributed computation of weighted norms is difficult. Namely using the approach in~\cite{c38}, this requires the calculation of the top singular value of $\bm{M}_{0}$ which amounts to a power iteration on $\bm{M}_{0}^{\mathsf{T}}\bm{M}_{0}$ leading to the loss of sparsity. Furthermore, conjugate gradient methods require global computations of inner products. 

Another existing method which we compare our results to is the recent work of the authors~\cite{c34} where a local and asynchronous solution for solving systems of linear equations is considered. In their work, the authors derive a complexity bound, for one component of the solution vector, of $\mathcal{O}\left(\min\left(d\epsilon^{\frac{\ln d}{\ln ||\bm{G}||_{2}}}, \frac{d n \ln \epsilon}{\ln ||\bm{G}||_{2}}\right)\right)$, with $\epsilon$ being the precision parameter, $d$ a constant bound on the maximal degree of $\mathcal{G}$, and $\bm{G}$ is defined as $\bm{x}=\bm{G}\bm{x}+\bm{z}$ which can be directly mapped to $\bm{A}\bm{x} = \bm{b}$. The relevant scenario to our work is when $\bm{A}$ is PSD and $\bm{G}$ is symmetric. Here, the bound on the number of multiplications is given by $\mathcal{O}\left(\min\left(d^{\frac{\kappa(\bm{A})+1}{2}\ln\frac{1}{\epsilon}}, \frac{\kappa(\bm{A})+1}{2}n d \ln\frac{1}{\epsilon}\right)\right)$, with $\kappa(\bm{A})$ being the condition number of $\bm{A}$. In the general case, when the degree depends on the number of nodes (i.e., $d = d(n)$), the minimum in the above bound will be the result of the second term ( $\frac{\kappa(\bm{A})+1}{2}n d \ln\frac{1}{\epsilon}$) leading to $\mathcal{O}\left(d(n) n \kappa(\bm{A})\ln \frac{1}{\epsilon}\right)$. Consequently, in such a general setting, our approach outperforms~\cite{c34} by a factor of $d(n)$. 

\textbf{Special Cases:} To better understand the complexity of the proposed SDDM solvers, next we detail the complexity for three specific graph structures. Before deriving these special cases, however, we first note the following simple yet useful connection between weighted and unweighted Laplacians of a graph $\mathcal{G}$. Denoting by $\bm{B}$ the incidence matrix of $\mathcal{G}$ and $\bm{W}$ a diagonal matrix with edge weights as diagonal elements, we can write: 
\begin{align*}
\mathcal{L}_{\mathcal{G}}&=\bm{B}^{\mathsf{T}}\bm{B} \ \ \ \text{and} \ \ \ \mathcal{L}_{\mathcal{G}}^{(\text{weighted})} = \bm{B}^{\mathsf{T}}\bm{W}\bm{B}.
\end{align*}
Hence, we can easily establish:
\begin{align*}
\mu_{n}\left(\mathcal{L}_{\mathcal{G}}^{(\text{weighted})}\right)\leq \bm{w}_{\text{max}}\mu_{n}\left(\mathcal{L}_{\mathcal{G}}\right) \ \ \ \ \ \ \ \ \ \ \text{and} \ \ \ \ \ \ \ \ \ \ \mu_{2}\left(\mathcal{L}_{\mathcal{G}}^{(\text{weighted})}\right) \geq \bm{w}_{\text{min}}\left(\mathcal{L}_{\mathcal{G}}\right).
\end{align*}

This implies that the condition number of the weighted Laplacian satisfies: 
\begin{align*}
\kappa\left(\mathcal{L}_{\mathcal{G}}^{(\text{weighted})}\right)=\frac{\mu_{n}\left(\mathcal{L}_{\mathcal{G}}^{(\text{weighted})}\right)}{\mu_{2}\left(\mathcal{L}_{\mathcal{G}}^{(\text{weighted})}\right)} \leq \frac{\bm{w}_{\text{max}}}{\bm{w}_{\text{min}}}\kappa\left(\mathcal{L}_{\mathcal{G}}\right),
\end{align*}
with $\bm{w}_{\text{min}}$ and $\bm{w}_{\text{max}}$ are the minimal and maximal edge weights of $\mathcal{G}$. Using the above, we now consider four different graph topologies: 
\subsubsection{Path Graph}
Similar to the hitting time of a Markov chain on a path graph which is given by $\mathcal{O}(n^{2})$, the time complexity of the R-Hop SDDM solver is given by\footnote{Due to space constraints, the proofs can be found in the appendix.}: 
\begin{corollary}
Given a path graph $\mathcal{P}_{n}$ with $n$ nodes, the time complexity of the R-Hop SDDM solver is given by:
\begin{equation*}
\text{T}_{\text{SDDM}}(\mathcal{P}_{n})=\mathcal{O}\left(n^{2} \log \frac{1}{\epsilon}\right),
\end{equation*}
for any $\epsilon > 0$ and for $k=m=\mathcal{O}(\sqrt{n})$.
\end{corollary}


\subsubsection{Grid Graph} Recognizing that a grid graph $\mathcal{G}_{k \times m}$ can be represented as a product of two path graphs, $\mathcal{G}_{k \times m}=\mathcal{P}_{k} \times \mathcal{P}_{m}$, our solver's computational time can be summarized by: 
\begin{corollary}
Given a grid graph, $\mathcal{G}_{k \times m}$, the time complexity of the distributed SDDM-solver can be bounded by: 
\begin{equation*}
\text{T}_{\text{SDDM}}\left(\mathcal{G}_{k \times m}\right) = \mathcal{O}\left(n\log\frac{1}{\epsilon}\right),
\end{equation*}
for any $\epsilon>0$. 
\end{corollary}

\subsubsection{Scale-Free Networks (Polya-Urn Graphs)} Using the results developed in~\cite{c39,c40} the total time complexity of the distributed R-Hop solver is bounded by: 
\begin{corollary}
Given a scale-free network, $\mathcal{G}_{\text{SN}(n)}$, the time complexity of the R-Hop SDD solver for R=1 is given by: 
\begin{equation*}
\text{T}_{\text{SDD}}\left(\mathcal{L}_{\mathcal{G}_{\text{SN}(n)}}\right)=\mathcal{O}\left(n^{2} \log n \log \frac{1}{\epsilon}\right).
\end{equation*}
\end{corollary}

\subsubsection{$d$- Regular Ramanujan Expanders} For $d$ regular Ramanujan expanders in which $d$ does not depend on $n$, we have $\mu_{n}\left(\mathcal{L}_{\text{RExp}(d)}\right) \leq 2d$ and $\mu_{2}\left(\mathcal{L}_{\text{RExp}(d)}\right) \geq d-2\sqrt{d-1}$. Hence, the time complexity of the SDD-solver is given by constant time.

The developed R-Hop distributed SDDM solver is a fundamental contribution with wide ranging applicability. Next, we develop one such application from. We apply our solver for proposing an efficient and accurate distributed Newton method for network flow optimization. Namely, the distributed SDDM solver is used for computing the Newton direction in a distributed fashion up-to any arbitrary $\epsilon >0$. This results in a novel distributed Newton method outperforming state-of-the-art techniques in both computational complexity and accuracy. 

\section{Distributed Newton Method for Network Flow Optimization}\label{Sec:Newton}


Conventional methods for distributed network optimization are based on sub-gradient descent in either the primal
or dual domains, see. For a large class of problems, these techniques yield iterations that can be
implemented in a distributed fashion using only local information. Their applicability, however, is limited by increasingly slow convergence rates. Second order Newton methods~\cite{c4} are known to overcome this limitation leading to improved convergence rates.

Unfortunately, computing exact Newton directions based only on local information is challenging. Specifically, to determine the Newton direction, the inverse of the dual Hessian is needed. Determining this inverse, however, requires global information. Consequently, authors in~\cite{c5,c6,c7} proposed approximate algorithms for determining these Newton iterates in a distributed fashion. Accelerated Dual Descent (ADD)~\cite{c5}, for instance, exploits the fact that the dual Hessian is the weighted Laplacian of the network and performs a truncated Neumann expansion of the inverse to determine a local approximate to the exact direction. ADD allows for a tradeoff between accurate Hessian approximations and communication costs through the N-Hop design, where increased N allows for more accurate inverse approximations arriving at increased cost, and lower values of N reduce accuracy but improve computational times. Though successful, the effectiveness of these approaches highly depend on the accuracy of the truncated Hessian inverse which is used to approximate the Newton direction. As shown later, the approximated iterate can resemble high variation to the
real Newton direction, decreasing the applicability of these techniques.

\textbf{Contributions:} Exploiting the sparsity pattern of the dual Hessian, here we tackle the above problem and propose a Newton method for network optimization that is both faster and more accurate. Using the above developed solvers for SDDM linear equations, we approximate the Newton direction up-to any arbitrary precision $\epsilon> 0$. This leads to a distributed second-order method which performs almost identically the exact Newton method.  Contrary to current distributed Newton methods, our algorithm is the first which is capable of attaining an $\epsilon$-close approximation to the Newton direction up to any arbitrary $\epsilon >0$. We analyze the properties of the proposed algorithm and show that, similar to conventional Newton methods, superlinear
convergence within a neighborhood of the optimal value
is attained. 

We finally demonstrate the effectiveness of the approach in a set of experiments on randomly generated and Barbell networks. Namely, we show that our method is capable of significantly outperforming state-of-the-art methods in both the convergence speeds and in the accuracy of approximating the Newton direction.

\subsection{Network Flow Optimization}

We consider a network represented by a directed graph $\mathcal{G}=\left(\mathcal{N},\mathcal{E}\right)$ with node set $\mathcal{N}=\{1,\dots, N\}$ and edge set $\mathcal{E}=\{1,\dots, E\}$. The flow vector is denoted by $\bm{x}=\left[x^{(e)}\right]_{e\in\mathcal{E}}$, with $x^{(e)}$ representing the flow on edge $e$. The flow conservation conditions at nodes can be compactly represented as 
\begin{equation*}
\bm{A}\bm{x}=\bm{b},
\end{equation*}
where $\bm{A}$ is the $N \times E$ node-edge incidence matrix of $\mathcal{G}$ defined as 

\begin{displaymath}
   \bm{A}_{i,j} = \left\{
     \begin{array}{lr}
       1 & \text{if edge $j$ leaves node $i$} \\
      - 1 & \text{if edge $j$ enters node $i$} \\
      0 & \text{otherwise,}
     \end{array}
   \right.
\end{displaymath} 
and the vector $\bm{b} \in \bm{1}^{\perp}$ denotes the external source, i.e., $b^{(i)} > 0$ (or $b^{(i) } < 0$) indicates $b^{(i)}$ units of external flow enters (or leaves) node $i$. A cost function $\bm{\Phi}_{e}: \mathbb{R} \rightarrow \mathbb{R}$ is associated with each edge $e$. Namely, $\bm{\Phi}_{e}(x^{(e)})$ denotes the cost on edge $e$ as a function of the edge flow $x^{(e)}$. We assume that the cost functions $\bm{\Phi}_{e}$ are strictly convex and twice differentiable. Consequently, the minimum cost network optimization problem can be written as 
\begin{align}
\label{Eq:OptimAll}
&\min_{\bm{x}} \sum_{e=1}^{E} \bm{\Phi}_{e}(\bm{x}^{(e)}) \\ \nonumber
&\text{s.t. $\bm{A}\bm{x}=\bm{b}$}
\end{align}

Our goal is to investigate Newton type methods for solving the problem in~\ref{Eq:OptimAll} in a distributed fashion. Before diving into these details, however, we next present basic ingredients needed for the remainder of the paper. 

\subsection{Dual Subgradient Method} 
The dual subgradient method optimizes the problem in Equation~\ref{Eq:OptimAll} by descending in the dual domain. The Lagrangian, $l(\cdot): \mathbb{R}^{E} \times \mathbb{R}^{N} \rightarrow \mathbb{R}$, is given by 
\begin{equation*}
l (\bm{x},\bm{\lambda}) = -\sum_{e=1}^{E} \bm{\Phi}_{e}({x}^{(e)}) + \bm{\lambda}^{\mathsf{T}} (\bm{A}\bm{x}-\bm{b}). 
\end{equation*}

The dual function $q(\bm{\lambda})$ is then derived as 
\begin{align*}
q(\bm{\lambda}) &= \inf_{\bm{x} \in \mathbb{R}^{E}} l(\bm{x},\bm{\lambda}) = \inf_{\bm{x} \in \mathbb{R}^{E}} \left(-\sum_{e=1}^{E} \bm{\Phi}_{e} (x^{(e)}) + \bm{\lambda}^{\mathsf{T}}\bm{A}\bm{x}\right) - \bm{\lambda}^{\mathsf{T}}\bm{b} \\ 
& = \sum_{e=1}^{E} \inf_{x^{(e)} \in \mathbb{R}} \left(-\bm{\Phi}_{e}(x^{(e)}) + \left(\bm{\lambda}^{\mathsf{T}}\bm{A}\right)^{(e)}x^{(e)}\right) - \bm{\lambda}^{\mathsf{T}}\bm{b}.
\end{align*}
Hence, it can be clearly seen that the evaluation of the dual function $q(\bm{\lambda})$ decomposes into E one-dimensional optimization problems. We assume that each of these optimization problems have an optimal solution, which is unique by the strict convexity of the functions $\bm{\Phi}_{e}$. Denoting the solutions by $x^{(e)} (\bm{\lambda})$ and using the first order optimality conditions, it can be seen that for each edge, e, $x^{(e)}(\bm{\lambda})$ is given by\footnote{Note that if the dual is not continuously differentiable, the a generalized Hessian can be used.}
\begin{equation}
\label{Eq:MapBack}
x^{(e)}(\bm{\lambda})=[\dot{\bm{\Phi}}_{e}]^{-1}\left(\lambda^{(i)}-\lambda^{(j)}\right),
\end{equation}
where $i \in  \mathcal{N}$ and $j \in \mathcal{N}$ denote the source and destining nodes of edge $e=(i,j)$, respectively (see~\cite{c6} for details). Therefore, for an edge $e$, the evaluation of $x^{(e)}(\bm{\lambda})$ can be performed based on local information about the edge's cost function and the dual variables of the incident nodes, $i$ and $j$. 

The dual problem is defined as $\min_{\bm{\lambda}\in \mathbb{R}^{N}} q(\bm{\lambda})$. Since the dual function is convex, the optimization problem can be solved using gradient descent according to
\begin{equation}
\label{Eq:GD}
\bm{\lambda}_{k+1} = \bm{\lambda}_{k} - \alpha_{k}\bm{g}_{k} \hfill \ \ \text{for all $k \geq 0$,}
\end{equation}
with $k$ being the iteration index, and $\bm{g}_{k}=\bm{g}\left(\bm{\lambda}_{k}\right)=\nabla q(\bm{\lambda}_{k})$ denoting the gradient of the dual function evaluated at $\bm{\lambda}=\bm{\lambda}_{k}$. Importantly, the computation of the gradient can be performed as $\bm{g}_{k}=\bm{A}\bm{x}\left(\bm{\lambda}_{k}\right)-\bm{b}$, with $\bm{x}(\lambda_{k})$ being a vector composed of $x^{(e)}(\bm{\lambda}_{k})$ as determined by Equation~\ref{Eq:MapBack}. Further, due to the sparsity pattern of the incidence matrix $\bm{A}$, the $i^{th}$  element, $g_{k}^{(i)}$, of the gradient $\bm{g}_{k}$ can be computed as 
\begin{equation}
\label{Eq:GDDist}
g_{k}^{(i)}=\sum_{e=(i,j)} x^{(e)}(\bm{\lambda}_{k}) - \sum_{e=(j,i)} x^{(e)}(\bm{\lambda}_{k}) - b^{(i)}.
\end{equation}

Clearly, the algorithm in Equation~\ref{Eq:GD} can be implemented in a distributed fashion, where each node, $i$, maintains information about its dual, $\bm{\lambda}_{k}^{(i)}$, and primal, $x^{(e)}(\bm{\lambda}_{k})$, iterates of the outgoing edges $e=(i,j)$. Gradient components can then be evaluated as per~\ref{Eq:GDDist} using only local information. Dual variables can then be updated using~\ref{Eq:GD}. Given the updated dual variables, the primal variables can be computed using~\ref{Eq:MapBack}. 

Although the distributed implementation avoids the cost and fragility of collecting all information at centralized location, practical applicability of gradient descent is hindered by slow convergence rates. This motivates the consideration of Newton methods discussed next. 

\subsection{Newton's Method for Dual Descent}
Newton's method is a descent algorithm along a scaled version of the gradient. Its iterates are typically given by 
\begin{equation}
\bm{\lambda}_{k+1} = \bm{\lambda}_{k}+\alpha_{k}\bm{d}_{k} \hfill \ \ \  \text{for all $k \geq 0$,}
\end{equation}
with $\bm{d}_{k}$ being the Newton direction at iteration $k$, and $\alpha_{k}$ denoting the step size. The Newton direction satisfies
\begin{equation}
\label{Eq:Hessian}
\bm{H}_{k}\bm{d}_{k}=-\bm{g}_{k},
\end{equation} 
with $\bm{H}_{k}=\bm{H}(\bm{\lambda}_{k})=\nabla^{2}q(\bm{\lambda}_{k})$ being the Hessian of the dual function at the current iteration $k$.

\subsubsection{Properties of the Dual and Assumptions}
Here, we detail some assumptions needed by our approach. We also derive essential Lemmas quantifying properties of the dual Hessian. 
\begin{assumption}
The graph, $\mathcal{G}$, is connected, non-bipartite and has algebraic connectivity lower bound by a constant $\bm{\omega}$. 
\end{assumption}

\begin{assumption}\label{Ass:Two}
The cost functions, $\bm{\Phi}_{e}(\cdot)$, in Equation~\ref{Eq:OptimAll} are 
\begin{enumerate}
\item twice continuously differentiable satisfying 
\begin{equation*}
\gamma \leq \ddot{\bm{\Phi}}_{e}(\cdot) \leq \Gamma,
\end{equation*}
with $\gamma$ and $\Gamma$ are constants; and
\item Lipschitz Hessian invertible for all edges $e\in \mathcal{E}$
\begin{equation*}
\left|\frac{1}{\ddot{\bm{\Phi}}_{e}(\bm{x})} - \frac{1}{\ddot{\bm{\Phi}}_{e}(\hat{\bm{x}})}\right| \leq {\delta}\left|\bm{x}-\hat{\bm{x}}\right|.
\end{equation*}
\end{enumerate}
\end{assumption}

The following two lemmas~\cite{c5,c6} quantify essential properties of the dual Hessian which we exploit through our algorithm to determine the approximate Newton direction. 

\begin{lemma}\label{lemma:Crap}
The dual objective $q(\bm{\lambda})=\bm{\lambda}^{\mathsf{T}}(\bm{A}\bm{x}(\bm{\lambda})-b)-\sum_{e}\bm{\Phi}_{e}(\bm{x}(\lambda))$ abides by the following two properties~\cite{c5,c6}:
\begin{enumerate}
\item The dual Hessian, $\bm{H}(\bm{\lambda})$, is a weighted Laplacian of $\mathcal{G}$:
\begin{equation*}
\bm{H}(\bm{\lambda})=\nabla^{2}q(\bm{\lambda})=\bm{A}\left[\nabla^{2}f(\bm{x}(\bm{\lambda}))\right]^{-1}\bm{A}^{\mathsf{T}}. 
\end{equation*}
\item The dual Hessian $\bm{H}(\bm{\lambda})$ is Lispshitz  continuous with respect to the Laplacian norm (i.e., $||\cdot||_{\mathcal{L}}$) where $\mathcal{L}$ is the unweighted laplacian satisfying $\mathcal{L}=\bm{A}\bm{A}^{\mathsf{T}}$ with $\bm{A}$ being the incidence matrix of $\mathcal{G}$. Namely, $\forall \bm{\lambda}, \bar{\bm{\lambda}}$: 
\begin{equation*}
||\bm{H}(\bar{\bm{\lambda}})-\bm{H}(\bm{\lambda})||_{\mathcal{L}} \leq B ||\bar{\bm{\lambda}}-\bm{\lambda}||_{\mathcal{L}},
\end{equation*}
with $B=\frac{\mu_{n}(\mathcal{L})\bm{\delta}}{\gamma \sqrt{\mu_{2}(\mathcal{L})}}$ where $\mu_{n}(\mathcal{L})$ and $\mu_{2}(\mathcal{L})$ denote the largest and second smallest eigenvalues of the Laplacian $\mathcal{L}$. 
\end{enumerate}
\end{lemma}
\begin{proof}
See Appendix.
\end{proof}

The following lemma follows from the above and is needed in the analysis later: 
\begin{lemma}\label{lemma:B}
If the dual Hessian $\bm{H}(\bm{\lambda})$ is Lipschitz continuous with respect to the Laplacian norm $||\cdot||_{\mathcal{L}}$ (i.e., Lemma~\ref{lemma:Crap}), then for any $\bm{\lambda}$ and $\hat{\bm{\lambda}}$ we have 
\begin{equation*}
||\nabla q(\hat{\bm{\lambda}})-\nabla q({\bm{\lambda}}) - \bm{H}(\bm{\lambda})(\hat{\bm{\lambda}}-\bm{\lambda})||_{\mathcal{L}} \leq \frac{B}{2}||\hat{\bm{\lambda}} - \bm{\lambda}||_{\mathcal{L}}^{2}.
\end{equation*}
\end{lemma}
\begin{proof}
See Appendix. 
\end{proof}

As detailed in~\cite{c6}, the exact computation of the inverse of the Hessian needed for determining the Newton direction can not be attained exactly in a distributed fashion. Authors in~\cite{c5,c6} proposed approximation techniques for computing this direction. The effectiveness of these algorithms, however, highly depends on the accuracy of such an approximation. In this work, we propose a distributed approximator for the Newton direction capable of acquiring $\epsilon$-close solutions for any arbitrary $\epsilon$. Our results show that this new algorithm is capable of significantly surpassing others in literature where its performance accurately traces that of the standard centralized Newton approach. 
\subsection{Accurate Distributed Newton Methods}
Using the results of the distributed R-Hop solver, we propose a novel technique requiring only R-Hop communication for the distributed approximation of the Newton direction. Given the results of Lemma~\ref{lemma:Crap}, we can determine the approximate Newton direction by solving a system of linear equations represented by an SDD matrix\footnote{For ease of presentation, we refrain some of the proofs to the appendix.} with $\bm{M}_{0}=\bm{H}_{k}=\bm{H}(\bm{\lambda}_{k})$.  

Formally, we consider the following iteration scheme:
\begin{equation}
\bm{\lambda}_{k+1}=\bm{\lambda}_{k} + \alpha_{k}\tilde{\bm{d}}_{k},
\end{equation}
with $k$ representing the iteration number, $\alpha_{k}$ the step-size, and $\tilde{\bm{d}}_{k}$ denoting the approximate Newton direction. We determine $\tilde{\bm{d}}_{k}$ by solving $\bm{H}_{k}\bm{d}_{k}=-\bm{g}_{k}$ using Algorithm~\ref{Algo:EDistR}. It is easy to see that our approximation of the Newton direction, $\tilde{\bm{d}}_{k}$, satisfies
\begin{align*}
||\tilde{\bm{d}}_{k} - \bm{d}_{k}||_{\bm{H}_{k}} &\leq \epsilon ||\bm{d}_{k}||_{\bm{H}_{k}} \  \ \ \text{with} \ \ \ \tilde{\bm{d}}_{k} =-\bm{Z}_{k}\bm{g}_{k},
\end{align*}
where $\bm{Z}_{k}$ approximates $\bm{H}^{\dagger}_{k}$ according to the routine of Algorithm~\ref{Algo:EDistR}. The accuracy of this approximation is quantified in the following Lemma 

\begin{lemma}\label{Lemma:Bla}
Let $\bm{H}_{k}=\bm{H}(\bm{\lambda}_{k})$ be the Hessian of the dual function, then for any arbitrary $\epsilon > 0$ we have  
\begin{equation*}
e^{-\epsilon^{2}} \bm{v}^{\mathsf{T}} \bm{H}_{k}^{\dagger} \bm{v} \leq \bm{v}^{\mathsf{T}}\bm{Z}_{k}\bm{v}\leq e^{\epsilon^{2}}\bm{v}^{\mathsf{T}}\bm{H}_{k}^{\dagger}\bm{v}, \ \ \ \ \ \ \forall \bm{v} \in \bm{1}^{\perp}.
\end{equation*}
\end{lemma}
\begin{proof}
See Appendix. 
\end{proof}

\subsubsection{Convergence Guarantees}
Given such an accurate approximation, next we analyze the iteration scheme of our proposed method showing that similar to standard Newton methods, we achieve superlinear convergence within a neighborhood of the optimal value. We start by analyzing the change in the Laplacian norm of the gradient between two successive iterations 
\begin{lemma}
Consider the following iteration scheme $\bm{\lambda}_{k+1}=\bm{\lambda}_{k} + \alpha_{k}\tilde{\bm{d}}_{k}$ with $\alpha_{k} \in (0,1]$, then, for any arbitrary $\epsilon>0$, the Laplacian norm of the gradient, $||\bm{g}_{k+1}||_{\mathcal{L}}$, follows:
\begin{align}\label{Eq:NormG}
||\bm{g}_{k+1}||_{\mathcal{L}} &\leq \left[1-\alpha_{k} +\alpha_{k}\epsilon\frac{\mu_{n}(\mathcal{L})}{\mu_{2}(\mathcal{L})}\sqrt{\frac{\Gamma}{\gamma}}\right]||\bm{g}_{k}||_{\mathcal{L}}  + \frac{\alpha_{k}^{2}B\Gamma^{2}(1+\epsilon)^{2}}{2\mu^{2}_{2}(\mathcal{L})}||\bm{g}_{k}||_{\mathcal{L}}^{2},
\end{align}
with $\mu_{n}(\mathcal{L})$ and $\mu_{2}(\mathcal{L})$ being the largest and second smallest eigenvalues of $\mathcal{L}$, $\Gamma$ and $\gamma$ denoting the upper and lower bounds on the dual's Hessian, and $B \in \mathbb{R}$ is defined in Lemma~\ref{lemma:B}. 
\end{lemma}
\begin{proof}
See Appendix. 
\end{proof}

At this stage, we are ready to present the main results quantifying the convergence phases exhibited by our approach:
\begin{theorem}
Let $\gamma$, $\Gamma$, $B$ be the constants defined in Assumption~\ref{Ass:Two} and Lemma~\ref{lemma:Crap}, $\mu_{n}(\mathcal{L})$ and $\mu_{2}(\mathcal{L})$ representing the largest and second smallest eigenvalues of the normalized laplacian $\mathcal{L}$, $\epsilon \in \left(0, \frac{\mu_{2}(\mathcal{L}}{\mu_{n}(\mathcal{L})}\sqrt{\frac{\Gamma}{\gamma}}\right)$ the precision parameter for the SDDM solver, and letting the optimal step-size parameter $\alpha^{*}=\frac{e^{-\epsilon^{2}}}{(1+\epsilon)^{2}}\left(\frac{\gamma}{\Gamma}\frac{\mu_{2}(\mathcal{L})}{\mu_{n}(\mathcal{L})}\right)^{2}$. Then the proposed algorithm given by the $\bm{\lambda}_{k+1}=\bm{\lambda}_{k}+\alpha^{*}\tilde{\bm{d}}_{k}$ exhibits the following three phases of convergence:
\begin{enumerate}
\item \textbf{Strict Decreases Phase:} While $||\bm{g}_{k}||_{\mathcal{L}} \geq \eta_{1}$:
\begin{equation*}
q(\bm{\lambda}_{k+1})-q(\bm{\lambda}_{k}) \leq -\frac{1}{2} \frac{e^{-2\epsilon^{2}}}{(1+\epsilon)^{2}}\frac{\gamma^{3}}{\Gamma^{2}}\frac{\mu_{2}^{2}(\mathcal{L})}{\mu_{n}^{4}(\mathcal{L})}\eta_{1}^{2}.
\end{equation*} 
\item \textbf{Quadratic Decrease Phase:} While $\eta_{0} \leq ||\bm{g}_{k}||_{\mathcal{L}}\leq \eta_{1}$:
\begin{equation*}
||\bm{g}_{k+1}||_{\mathcal{L}} \leq \frac{1}{\eta_{1}}||\bm{g}_{k}||_{\mathcal{L}}^{2}.
\end{equation*}
\item \textbf{Terminal Phase:} When $||\bm{g}_{k}||_{\mathcal{L}} \leq \eta_{0}$: 
\begin{equation*}
||\bm{g}_{k+1}||_{\mathcal{L}}\leq \sqrt{\left[1-\alpha^{*}+\alpha^{*}\epsilon\frac{\mu_{n}(\mathcal{L})}{\mu_{2}(\mathcal{L})}\sqrt{\frac{\Gamma}{\gamma}}\right]}||\bm{g}_{k}||_{\mathcal{L}},
\end{equation*}
\end{enumerate}
where $\eta_{0}=\frac{\bm{\xi}(1-\bm{\xi})}{\bm{\zeta}}$ and $\eta_{1}=\frac{1-\bm{\xi}}{\bm{\zeta}}$, with 
\begin{align}\label{Eq:Aux}
\bm{\xi}&=\sqrt{\left[1-\alpha^*+\alpha^*\epsilon\frac{\mu_{n}(\mathcal{L})}{\mu_{2}(\mathcal{L})}\sqrt{\frac{\Gamma}{\gamma}}\right]} \ \ \ \text{with} \ \ \ \bm{\zeta}=\frac{B(\alpha^{*}\Gamma(1+\epsilon))^{2}}{2\mu_{2}^{2}(\mathcal{L})}
\end{align}
\end{theorem}

\begin{proof}
We will proof the above theorem by handling each of the cases separately. We start by considering the case when $||\bm{g}_{k}||_{\mathcal{L}} > \eta_{1}$ (i.e., \textbf{Strict Decrease Phase}). We have:
\begin{align*}
q(\lambda_{k+1}) & = q(\bm{\lambda}_{k}) + \bm{g}_{k}^{\mathsf{T}}(\bm{\lambda}_{k+1}-\bm{\lambda}_{k}) +\frac{1}{2}(\bm{\lambda}_{k+1}-\bm{\lambda}_{k})^{\mathsf{T}}\bm{H}(\bm{z})(\bm{\lambda}_{k+1}-\bm{\lambda}_{k}) \\ 
& = q(\bm{\lambda}_{k}) + \alpha_{k} \bm{g}_{k}^{\mathsf{T}}\tilde{\bm{d}}_{k} +\frac{\bm{\alpha}_{k}^{2}}{2} \tilde{\bm{d}}_{k}^{\mathsf{T}}\bm{H}(\bm{z})\tilde{\bm{d}}_{k} \leq q(\bm{\lambda}_{k}) + \alpha_{k}\bm{g}_{k}^{\mathsf{T}}\tilde{\bm{d}}_{k}+\frac{\alpha_{k}^{2}}{2\gamma}\tilde{\bm{d}}^{\mathsf{T}}_{k} \mathcal{L}\tilde{\bm{d}}_{k},
\end{align*} 
where the last steps holds since $\bm{H}(\cdot) \preceq \frac{1}{\gamma}\mathcal{L}$. Noticing that $||\tilde{\bm{d}}_{k}||_{\mathcal{L}}^{2} \leq \frac{\Gamma^{2}(1+\epsilon)^{2}}{\mu_{2}^{2}(\mathcal{L})}||\bm{g}_{k}||_{\mathcal{L}}^{2}$ (see Appendix), the only remaining step needed is to evaluate $\bm{g}_{k}^{\mathsf{T}}\tilde{\bm{d}}_{k}$. Knowing that $\tilde{\bm{d}}_{k} = -\bm{Z}_{k}\bm{g}_{k}$, we recognize
\begin{align*}
\bm{g}_{k}^{\mathsf{T}}\tilde{\bm{d}}_{k} &= - \bm{g}_{k}^{\mathsf{T}}\bm{Z}_{k}\bm{g}_{k} \leq e^{-\epsilon^{2}}\bm{g}_{k}^{\mathsf{T}}\bm{H}_{k}^{\dagger}\bm{g}_{k} \ (\text{Lemma~\ref{Lemma:Bla}}) \\
&\leq -\frac{e^{-\epsilon^{2}}}{\mu_{n}(\bm{H}_{k})}\bm{g}_{k}^{\mathsf{T}}\bm{g}_{k} \leq -\frac{e^{-\epsilon}}{\mu_{n}(\mathcal{L})}\bm{g}_{k}^{\mathsf{T}}\bm{g}_{k} \leq -\frac{e^{-\epsilon^{2}}\gamma}{\mu_{n}(\mathcal{L})}\frac{\bm{g}_{k}^{\mathsf{T}}\mathcal{L}\bm{g}_{k}}{\mu_{n}(\mathcal{L})} = \frac{e^{-\epsilon^{2}}\gamma}{\mu_{n}^{2}(\mathcal{L})}||\bm{g}_{k}||_{\mathcal{L}}^{2},
\end{align*}
where the last step follows from the fact that $\forall \bm{v}\in \mathbb{R}^{n}: \bm{v}^{\mathsf{T}}\bm{v} \geq \frac{\bm{v}^{\mathsf{T}}\mathcal{L}\bm{v}}{\mu_{n}(\mathcal{L})}$. Therefore, we can write
\begin{equation*}
q(\bm{\lambda}_{k+1})-q(\bm{\lambda}_{k}) \leq -\left[\alpha_{k}\frac{e^{-\epsilon^{2}}\gamma}{\mu_{n}^{2}(\mathcal{L})}-\alpha_{k}^{2}\frac{\Gamma^{2}(1+\epsilon)^{2}}{2\gamma\mu_{2}^{2}(\mathcal{L})}\right]||\bm{g}_{k}||_{\mathcal{L}}^{2}.
\end{equation*}
It is easy to see that $\alpha_{k}=\alpha^{*}=\frac{e^{-\epsilon^{2}}}{(1+\epsilon)^{2}}\left(\frac{\gamma}{\Gamma}\frac{\mu_{2}(\mathcal{L})}{\mu_{n}(\mathcal{L})}\right)^{2}$ minimizes the right-hand-side of the above equation. Using $||\bm{g}_{k}||_{\mathcal{L}}$ gives the constant decrement in the dual function between two successive iterations as 
\begin{equation*}
q(\bm{\lambda}_{k+1})-q(\bm{\lambda}_{k}) \leq -\frac{1}{2} \frac{e^{-2\epsilon^{2}}}{(1+\epsilon)^{2}}\frac{\gamma^{3}}{\Gamma^{2}}\frac{\mu_{2}^{2}(\mathcal{L})}{\mu_{n}^{4}(\mathcal{L})}\eta_{1}^{2}.
\end{equation*}
Considering the case when $\eta_{0} \leq ||\bm{g}_{k}||_{\mathcal{L}}^{2}\eta_{1}$ (i.e., \textbf{Quadratic Decrease Phase}), Equation~\ref{Eq:NormG} can be rewritten as
\begin{equation*}
||\bm{g}_{k+1}||_{\mathcal{L}}\leq \bm{\xi}^{2}||\bm{g}_{k}||_{\mathcal{L}} + \bm{\zeta}||\bm{g}_{k}||_{\mathcal{L}}^{2},
\end{equation*}
with $\bm{\xi}$ and $\bm{\zeta}$ defined as in Equation~\ref{Eq:Aux}. Further, noticing that since $||\bm{g}_{k}||_{\mathcal{L}} \geq \eta_{0}$ then $||\bm{g}_{k}||_{\mathcal{L}}\leq \frac{1}{\eta_{0}}||\bm{g}_{k}||_{\mathcal{L}}^{2}=\frac{\bm{\zeta}}{\bm{\xi}(1-\bm{\xi})}||\bm{g}_{k}||_{\mathcal{L}}^{2}$. Consequently the quadratic decrease phase is finalized by
\begin{align*}
||\bm{g}_{k+1}||_{\mathcal{L}} \leq \bm{\zeta}\left(\frac{\bm{\xi}}{1-\bm{\xi}}+1\right)||\bm{g}_{k}||_{\mathcal{L}}^{2}&=\frac{\bm{\zeta}}{1-\bm{\xi}}||\bm{g}_{k}||_{\mathcal{L}}^{2} =\frac{1}{\eta_{1}}||\bm{g}_{k}||_{\mathcal{L}}^{2}.
\end{align*}
Finally, we handle the case where $||\bm{g}_{k}||_{\mathcal{L}}\leq \eta_0$ (i.e., \textbf{Terminal Phase}). Since $||\bm{g}_{k}||_{\mathcal{L}}^{2} \leq \eta_{0}||\bm{g}_{k}||_{\mathcal{L}}$, it is easy to see that 
\begin{align*}
||\bm{g}_{k+1}||_{\mathcal{L}} &\leq (\bm{\xi}^{2}+\bm{\zeta}\eta_{0})||\bm{g}_{k}||_{\mathcal{L}}=(\bm{\xi}^{2}+\bm{\xi}(1-\bm{\xi})) ||\bm{g}_{k}||_{\mathcal{L}} \\
& =\bm{\xi}||\bm{g}_{k}||_{\mathcal{L}} = \sqrt{\left[1-\alpha^*+\alpha^*\epsilon\frac{\mu_{n}(\mathcal{L})}{\mu_{2}(\mathcal{L})}\sqrt{\frac{\Gamma}{\gamma}}\right]}||\bm{g}_{k}||_{\mathcal{L}}.
\end{align*}
\end{proof}

\subsubsection{Iteration Count and Message Complexity}
Having proved the three convergence phases of our algorithm, we next analyze the number of iterations needed by each phase. These results are summarized in the following lemma: 

\begin{lemma}\label{Lemma:Iter}
Consider the algorithm given by the following iteration protocol: $\bm{\lambda}_{k+1}=\bm{\lambda}_{k}+\alpha^{*}\tilde{\bm{d}}_{k}$. Let $\bm{\lambda}_{0}$ be the initial value of the dual variable, and $q^{*}$ be the optimal value of the dual function. Then, the number of iterations needed by each of the three phases satisfy: 
\begin{enumerate}
\item The {\textbf{strict decrease phase}} requires the following number of iterations to achieve the quadratic phase: 
\begin{equation*}
N_{1} \leq C_{1} \frac{\mu_{n}(\mathcal{L})^{2}}{\mu_{2}^{3}(\mathcal{L})} \left[1-\epsilon\frac{\mu_{n}(\mathcal{L})}{\mu_{2}(\mathcal{L})}\sqrt{\frac{\Gamma}{\gamma}}\right]^{-2},
\end{equation*}
where $C_{1}=C_{1}\left(\epsilon, \gamma, \Gamma, \bm{\delta}, q(\bm{\lambda}_{0}), q^{\star} \right)=2\bm{\delta}^{2}(1+\epsilon)^{2}\left[q(\bm{\lambda}_{0})-q^{\star}\right]\frac{\Gamma^{2}}{\gamma}$.
\item The \textbf{quadratic decrease phase} requires the following number of iterations to terminate: 
\begin{equation*}
N_{2} = \log_{2}\left[\frac{\frac{1}{2}\log_{2}\left(\left[1-\alpha^{*}\left(1-\epsilon\frac{\mu_{n}(\mathcal{L})}{\mu_{2}(\mathcal{L})}\sqrt{\frac{\Gamma}{\gamma}}\right)\right]\right)}{\log_{2}(r)}\right],
\end{equation*}
where $r=\frac{1}{\eta_{1}}||\bm{g}_{k^{\prime}}||_{\mathcal{L}}$, with $k^{\prime}$ being the first iteration of the quadratic decrease phase. 
\item The radius of the {\textbf{terminal phase}} is characterized by:
\begin{equation*}
\rho_{\text{terminal}} \leq \frac{2\left[1-\epsilon\frac{\mu_{n}(\mathcal{L})}{\mu_{2}(\mathcal{L})}\sqrt{\frac{\Gamma}{\gamma}}\right]}{e^{-\epsilon^{2}\gamma\bm{\delta}}}\mu_{n}(\mathcal{L})\sqrt{\mu_{2}(\mathcal{L})}.
\end{equation*}
\end{enumerate}
\end{lemma}
\begin{proof}
See Appendix. 
\end{proof}

Given the above result, the total message complexity can then be derived as:
\begin{equation*}
\mathcal{O}\left(\left(N_{1}+N_{2}\right)n\beta\left(\kappa(\bm{H}_{k})\frac{1}{R}+R\text{d}_{\max}\right)\log\left(\frac{1}{\epsilon}\right)\right).
\end{equation*} 

\subsubsection{Comparison to Existing Literature}
Recent progress towards distributed second-order methods applied to network flow optimization adopt an approximation scheme based on the properties of the Hessian. Most of these methods (e.g.,~\cite{c7}) handle the simpler consensus setting and thus are not directly applicable to our more general setting. Closest to this work is that proposed in~\cite{c5}, where the authors approximate the Newton direction by truncating the Neumann expansion of the pseudo-inverse of the Hessian. This truncation, however, introduces additional error to the computation of the Newton direction leading to inaccurate results especially on large networks. The primary contrast to this work is that our method is capable of acquiring $\epsilon$-close (for any arbitrary $\epsilon$) approximation to the Newton direction. In fact, the proposed method is almost identical to the exact Newton direction computed in a centralized manner (see Section~\ref{Sec:ExpNewton}). 

Next, we formally derive the iteration counts for our method on four-special cases as benchmark comparisons. Since the main contribution (due to the presence of $\log \log (\cdot)$ term in $N_{2}$) in the iteration count is given by the strict decrease phase (see Lemma~\ref{Lemma:Iter}), we develop these results in terms of $N_{1}$. Also note that the corollaries provided below are substantially harder to acquire when compared to the standard Newton analysis due to the dependence of some constants, e.g, $m$ and $M$ on the graph structure. Opposed to the analysis performed by other methods, our analysis explicitly handles such dependencies leading to more realistic and accurate mathematical insights. This explains the relatively high dependency on $n$ in some cases. Note, that in such case where these relations (i.e., parameter dependency on the graph structure) were not taken into account, substantial decrease in the complexity can be achieved at the compensate of accurate mathematical description.  

\textbf{Path Graph:} 
\begin{corollary}
Given a path graph $\mathcal{P}_{n}$ with $n$ nodes, the strict-decrease phase of the distributed Newton method is given by:
\begin{equation*}
N_{1}\left(\mathcal{P}_{n}\right)=\mathcal{O}\left(n^{6}\left[1-\epsilon\frac{4n^{2}}{\pi^{2}}\sqrt{\frac{\Gamma}{\gamma}}\right]^{-2}\right)
\end{equation*}
\end{corollary}

\textbf{Grid Graph:} 
\begin{corollary}
For a grid graph $\mathcal{G}_{k \times m}$ the strict-decrease phase of the distributed Newton method is given by:
\begin{equation*}
N_{1}\left(\mathcal{G}_{k \times m}\right)=\mathcal{O}\left(n^{3}\left[1-\epsilon\frac{8n^{2}}{\pi^{2}\left(k^{2}+\frac{n^{2}}{k^{2}}\right)}\sqrt{\frac{\Gamma}{\gamma}}\right]^{-2}\right)
\end{equation*}
\end{corollary}

\textbf{Scale-Free Graph:} 
\begin{corollary}
For a scale free grid graph $\mathcal{G}_{\text{SN}(n)}$ with $n$, the strict-decrease phase is given by:
\begin{equation*}
N_{1}\left(\mathcal{G}_{\text{SN}(n)}\right)=\mathcal{O}\left(n^{4}\log n\left[1-\epsilon\frac{n^{\frac{3}{2}}\log n}{4}\sqrt{\frac{\Gamma}{\gamma}}\right]^{-2}\right)
\end{equation*}
\end{corollary}

\textbf{d-Regular Ramanujan Expanders:} For such expanders the iteration count for the strict-decrease phase can be bounded by a constant. 

\subsection{Experiments and Results}\label{Sec:ExpNewton}
We evaluated the proposed distributed second-order method in three sets of experiments on randomly generated and Barbell networks. The goal was to assess the performance on networks exhibiting good and bad mixing times. We compared our algorithm's performance to: 1) exact-newton computed in a centralized fashion, 2) Accelerated Dual Descent (ADD) with two different splittings~\cite{c5,c6}, 3) dual sub-gradients, and 4) the fully distributed algorithms for convex optimization~\cite{c41} (FDA). An $\epsilon$ of $\frac{1}{10,000}$, a feasibility threshold of $10^{-5}$, and an R-Hop of 1 were provided to our SDDM solver for determining the approximate Newton direction. For all other methods free parameters were chosen as specified by the relevant papers. Feasibility and objective values were used as performance measures.

\subsubsection{Experiments \& Results on Random Graphs} 
Two experiments on small (20 nodes, 60 edges) and large (50 nodes, 150 edges) random graphs were conducted. The random graphs were constructed in such a way that  edges were drawn uniformly at random. The flow vectors, $\bm{b}$, were chosen to place source and sink nodes $\text{diam}(\mathcal{G})$ apart.  
\begin{figure*}[t!]
\centering
\vspace{-.5em}
\subfigure[]{
	\label{fig:PerfSM}
\includegraphics[width=0.49\textwidth]{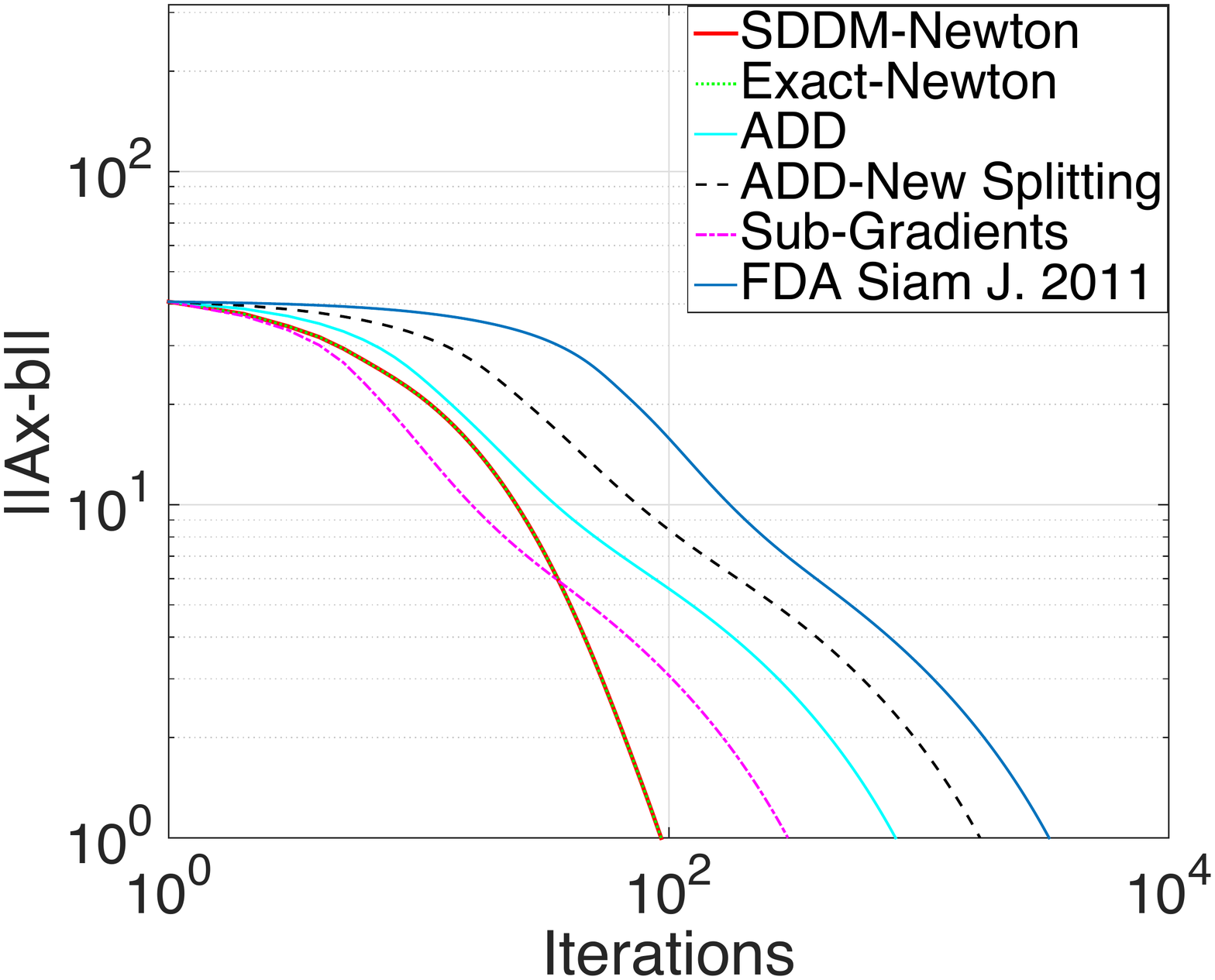}
}
\hfill\hspace{-.3in}\hfill
\subfigure[]{
	\label{fig:PerfCP}
\includegraphics[width=0.5\textwidth]{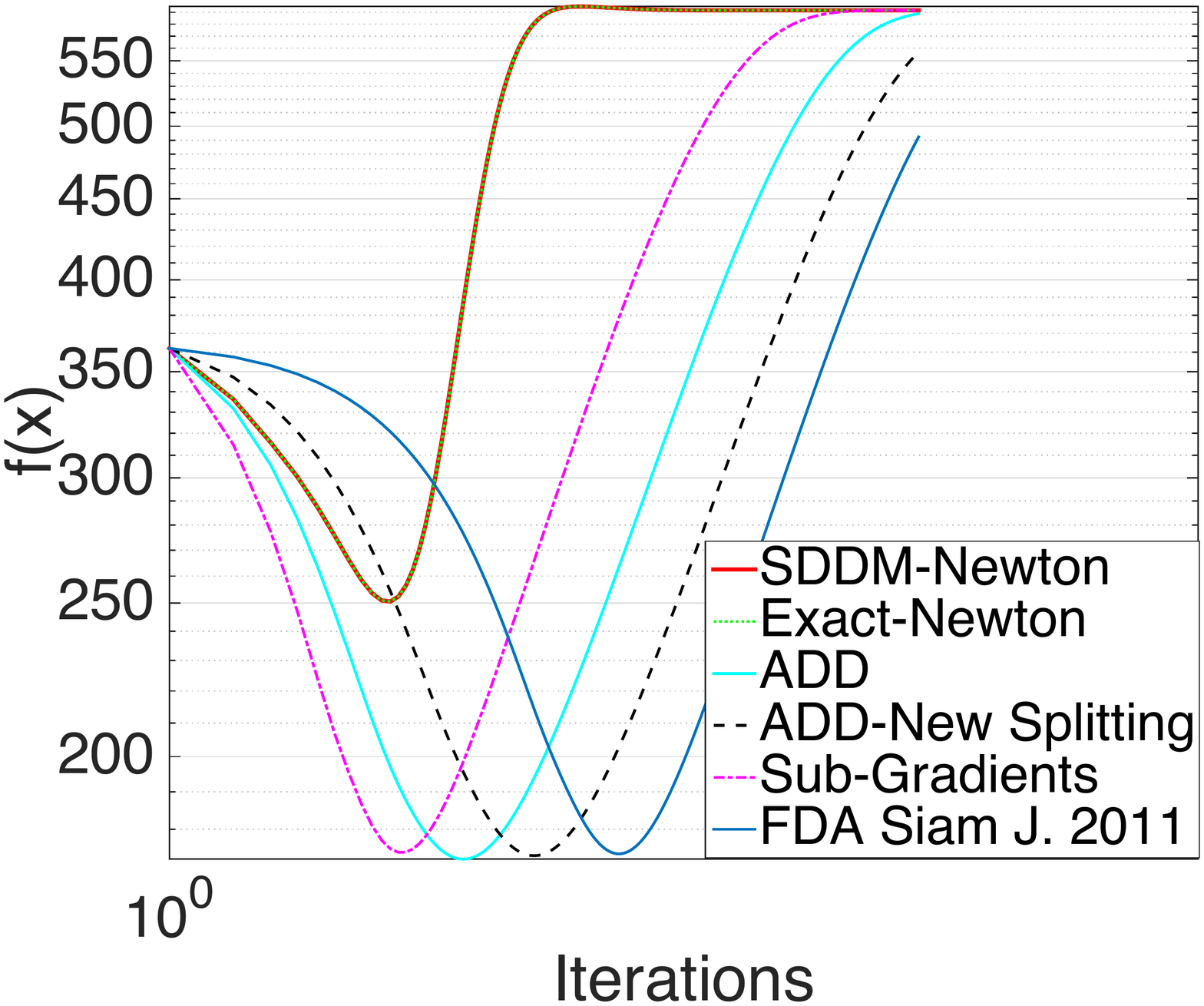}
}
\hfill\hspace{-.3in}\hfill
\subfigure[]{
	\label{fig:TrajSM}
\includegraphics[width=0.49\textwidth]{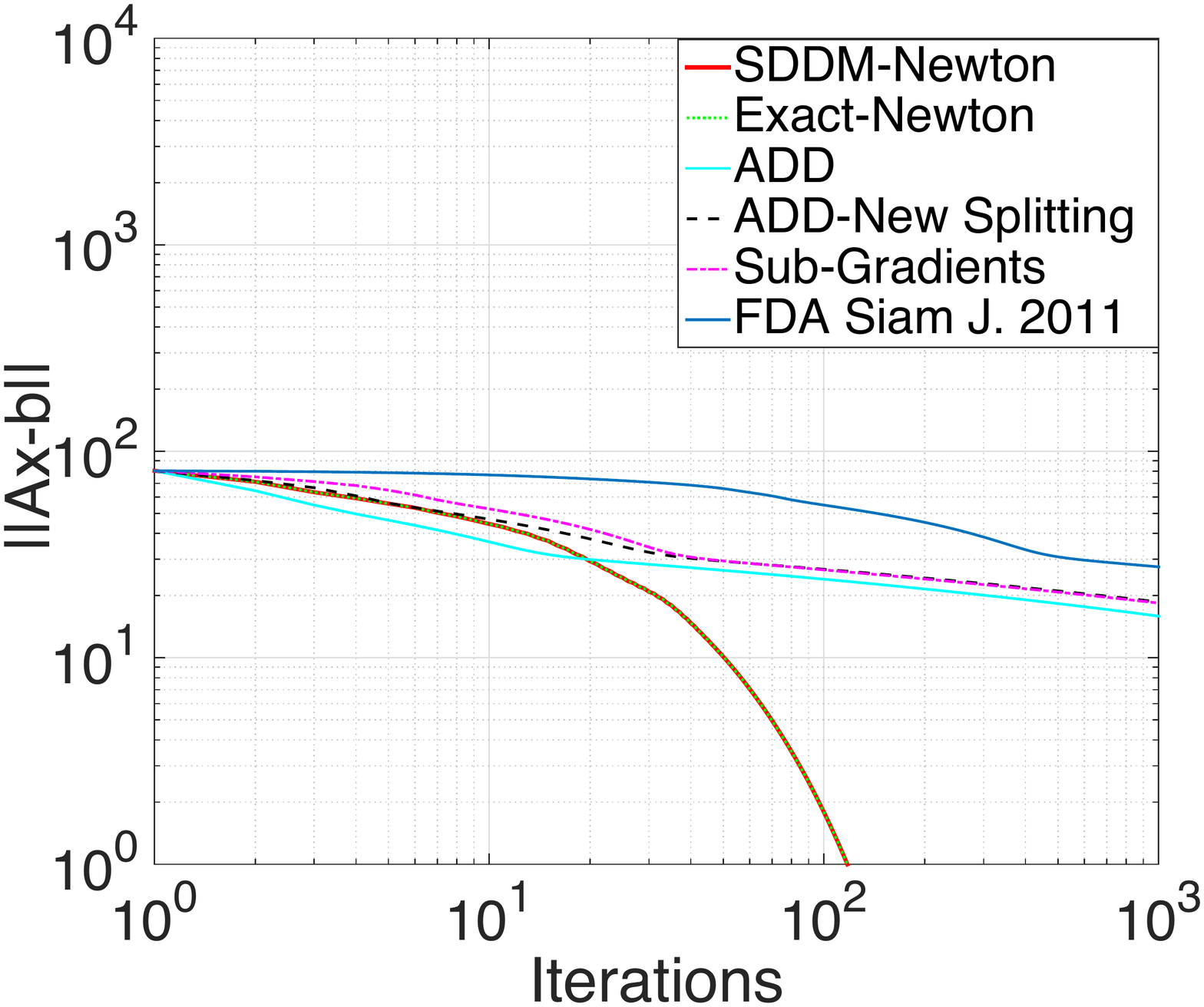}
}
\hfill\hspace{-.3in}\hfill
\subfigure[]{
	\label{fig:TrajCP}
\includegraphics[width=0.5\textwidth]{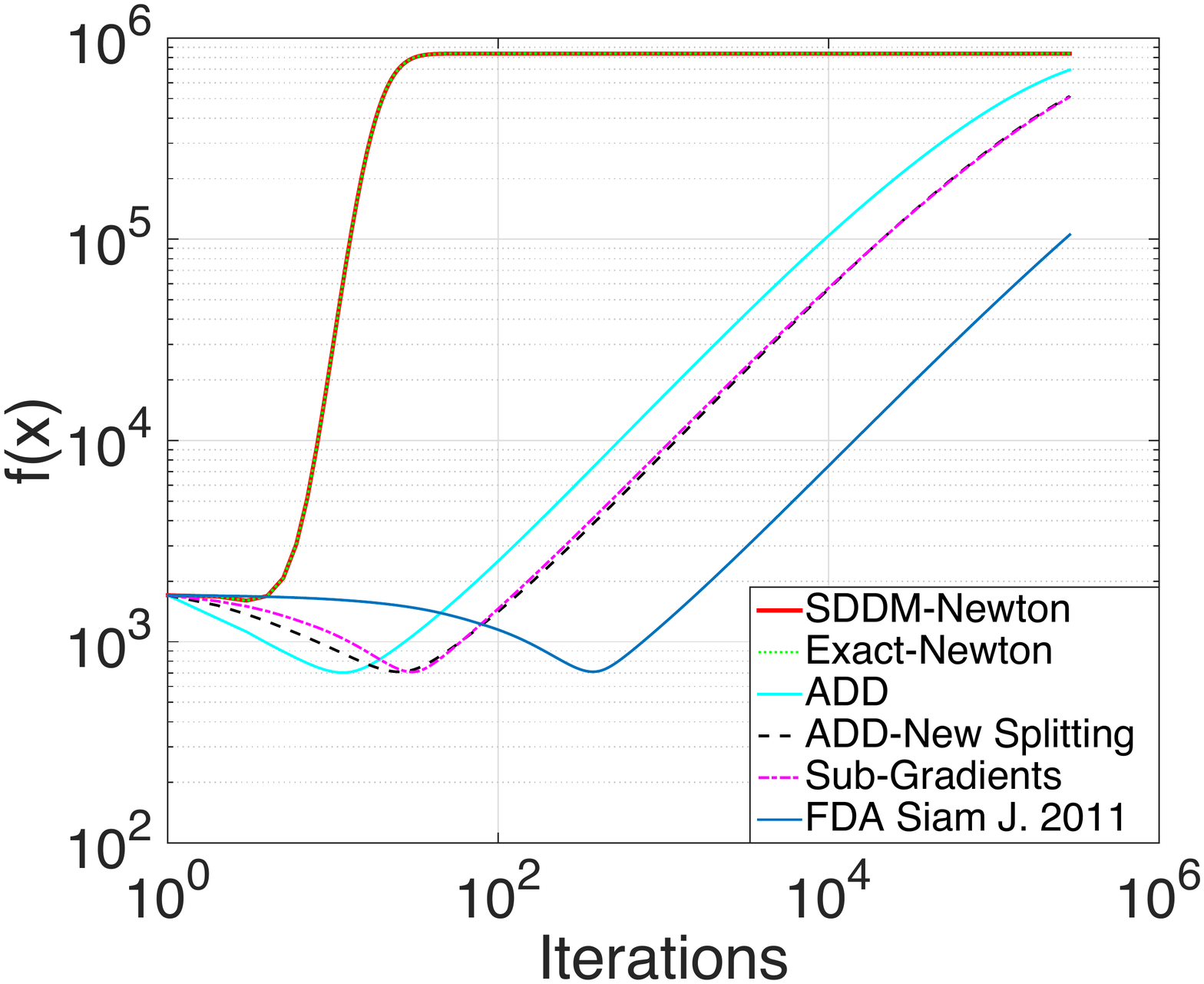}
}

\caption{Performance metrics on two randomly generated networks, showing the primal objective, $f\left(\bm{x}_{k}\right)$, and feasibility $||\bm{A}\bm{x}_{k}-\bm{b}||$ as a function of the number of iterations $k$ on a loglog scale. On a relatively small network (i.e., $20$ nodes and $60$ edges- Figures (a) and (b)) we outperform all method by an order of magnitude and trace the trajectory of exact Newton computed in a centralized fashion. On larger networks (i.e., 50 nodes and 150 edges-Figures (c) and (d)), SDDM-Newton is superior to all other algorithms, where it converges 5 orders of magnitude faster.}
\label{Fig:ResBenchmark}
\end{figure*}

Results summarizing the primal objective value, $f([\bm{x}]^{(e)})=\exp([\bm{x}]^{(e)})+\exp(-[\bm{x}]^{(e)})$, and feasibility, $||\bm{A}\bm{x}_{k} - \bm{b}||$, on both networks are shown in Figure~3. On small networks (i.e., 20 nodes and 60 edges), all algorithms perform relatively well. Clearly, our proposed approach (titled SDDM-Newton in the figures) outperforms, ADD, Sub-gradients, and the approach in~\cite{c41} with about an order of magnitude. Another interesting realization is that SDDM-Newton accurately tracks the exact Newton method with its direction computed in a centralized fashion. The reason for such positive results, is that our algorithm is capable of approximating the Newton direction up-to-any arbitrary $\epsilon >0$ while abiding by the R-Hop constraint. For larger networks, Figures~\ref{fig:TrajSM} and~\ref{fig:TrajCP}, SDDM-Newton is highly superior compared to other approaches. Here, our algorithm is capable of converging in about 3-5 orders of magnitude faster, i.e., we converge in about 3000 iterations as opposed to 6000 for ADD which showed the best performance among the other techniques. It is worth noting that we are again capable of tracing exact Newton due to the accuracy of our approximation.

\subsubsection{Experiments \& Results on Bar Bell Graphs} 
In the second set of experiments we assessed the performance of SDDM-Newton on a barbell graph, Figure~\ref{Fig:BarBellConstructions}, of 60 nodes. The graph consisted of two 20 node cliques connected by a 20 nodes line graph is depicted in. We assign directed edges on the graph arbitrary and solve the network flow optimization problem with a cost of $f([\bm{x}]^{(e)})=\exp([\bm{x}]^{(e)})+\exp(-[\bm{x}]^{(e)})$. 

\begin{wrapfigure}{l}{0.5\textwidth}
  \begin{center}
    \includegraphics[width=0.5\textwidth]{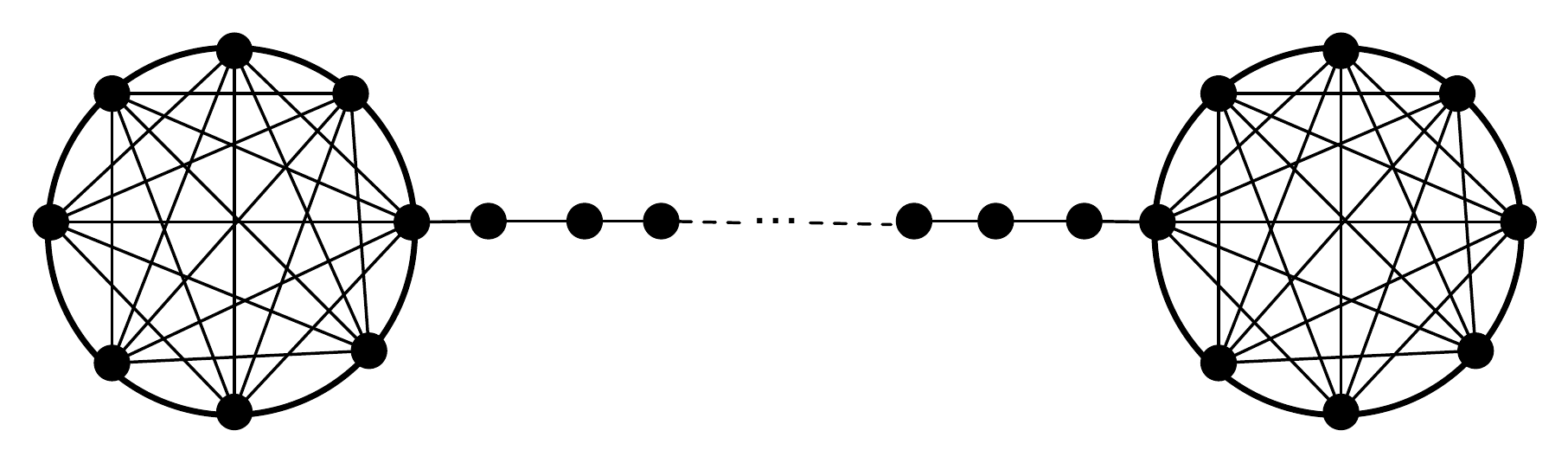}
  \end{center}
  \caption{A high-level depiction of a barbell graph showing two cliques connected by a line graph.}
  \label{Fig:BarBellConstructions}
\end{wrapfigure}

Results in Figure~\ref{fig:PerfSMTw} and~\ref{fig:PerfCPTw} demonstrate the superiority of our algorithm to state-of-the-art methods. Here, again we are capable of converting faster than other techniques in about 2 magnitudes faster. It is worth noting that the second-best performing algorithm is ADD which is capable of converging in almost 3000 iterations. 

\begin{figure*}[t!]
\centering
\vspace{-.5em}
\subfigure[]{
	\label{fig:PerfSMTw}
\includegraphics[width=0.5\textwidth]{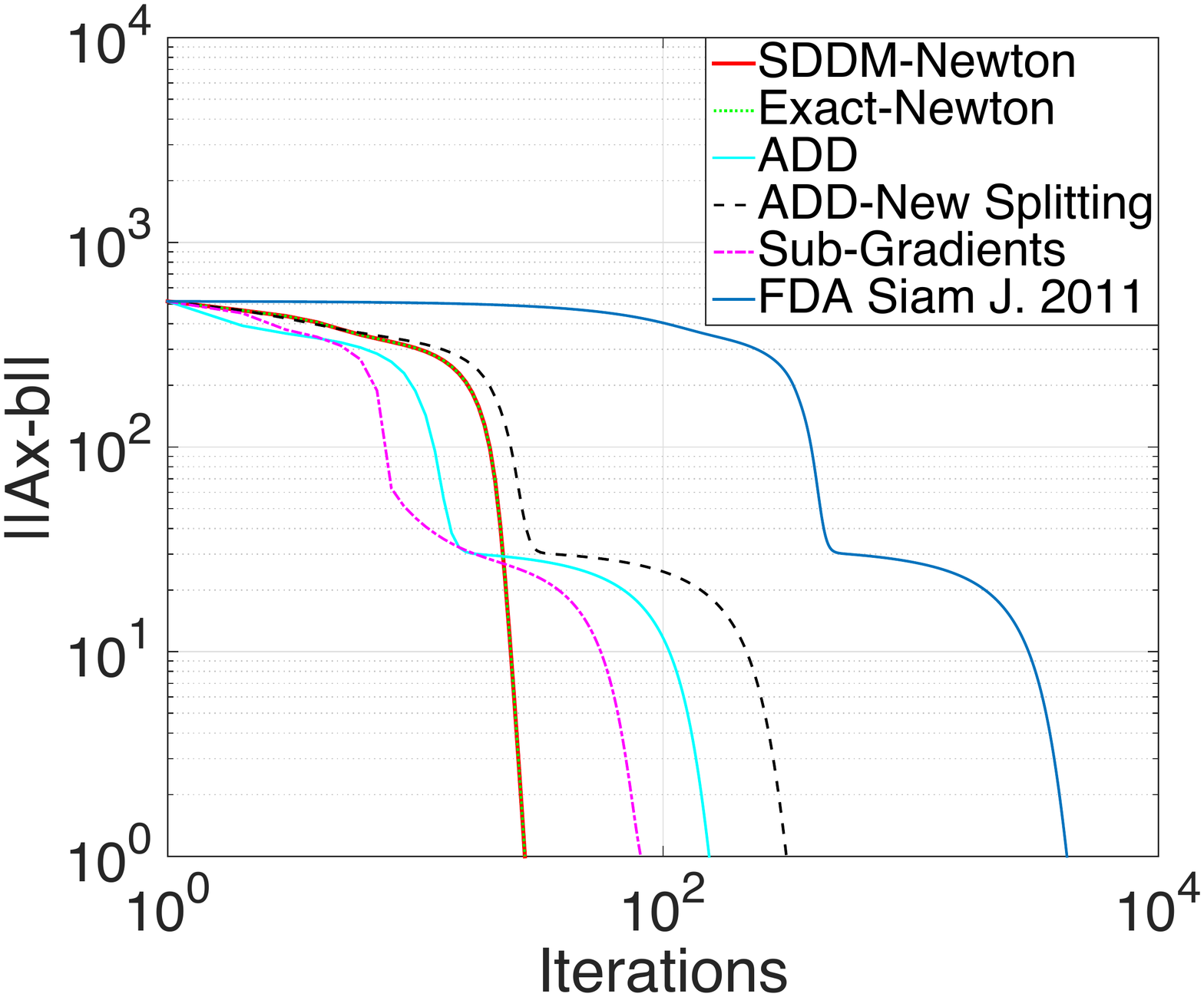}
}
\hfill\hspace{-.3in}\hfill
\subfigure[]{
	\label{fig:PerfCPTw}
\includegraphics[width=0.49\textwidth]{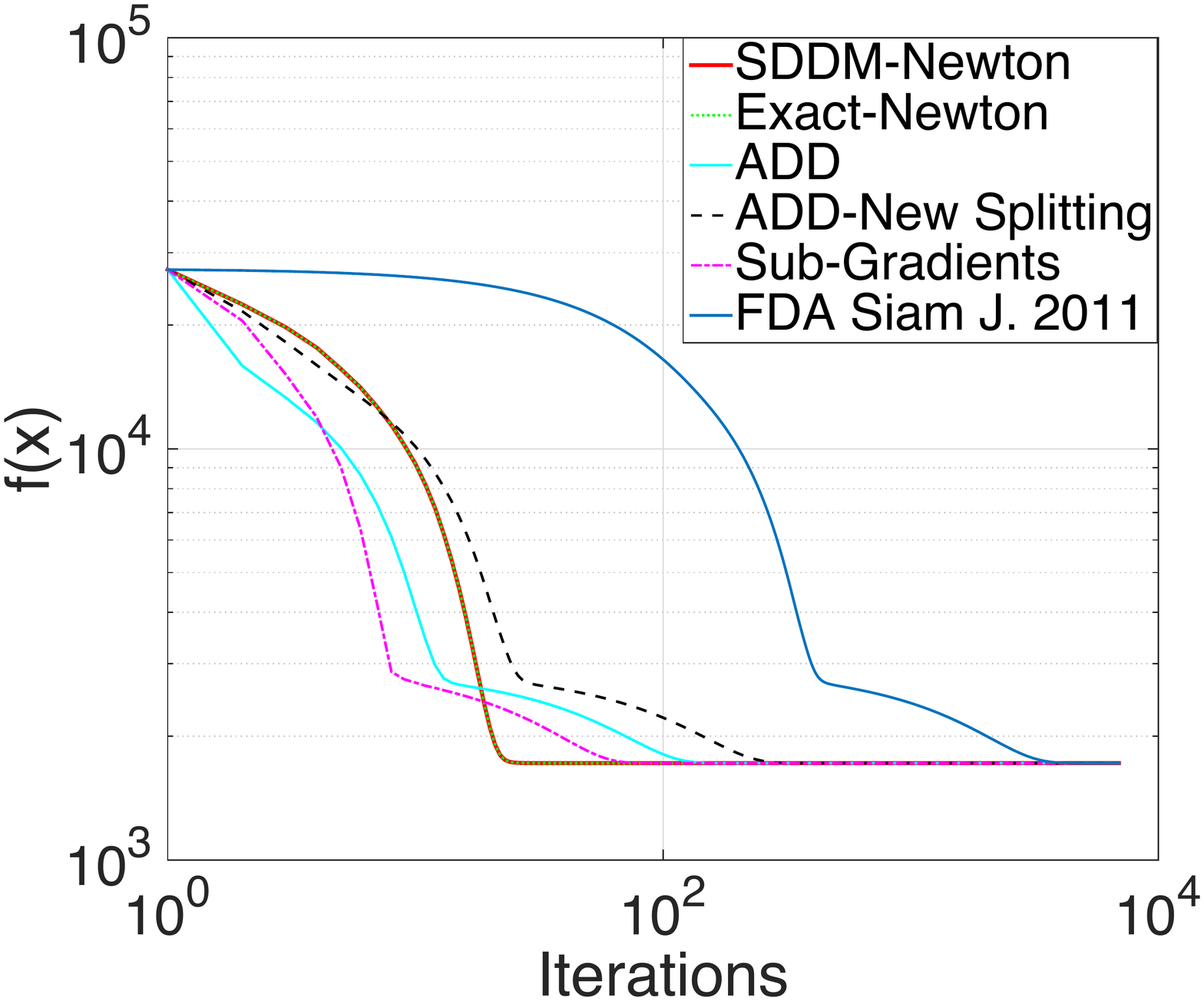}
}
\caption{Feasibility and objective results on a 60 node bar graph. Again in these sets of experiments, our method is capable of outperforming all other methods. Furthermore, SDDM-Newton is again capable of tracing the exact Newton method.}
\end{figure*}
Finally, we repeated the same experiments on a larger bar bell network formed of 120 nodes (two cliques with 40 nodes connected by a 40 node line graph). These results, demonstrated in Figures~\ref{fig:PerfSMTwo} and~\ref{fig:PerfCPTwo}, again validate the previous performance measures showing even better performance in both the objective value and feasibility. 
\begin{figure*}[t!]
\centering
\vspace{-.5em}
\subfigure[]{
	\label{fig:PerfSMTwo}
\includegraphics[width=0.5\textwidth]{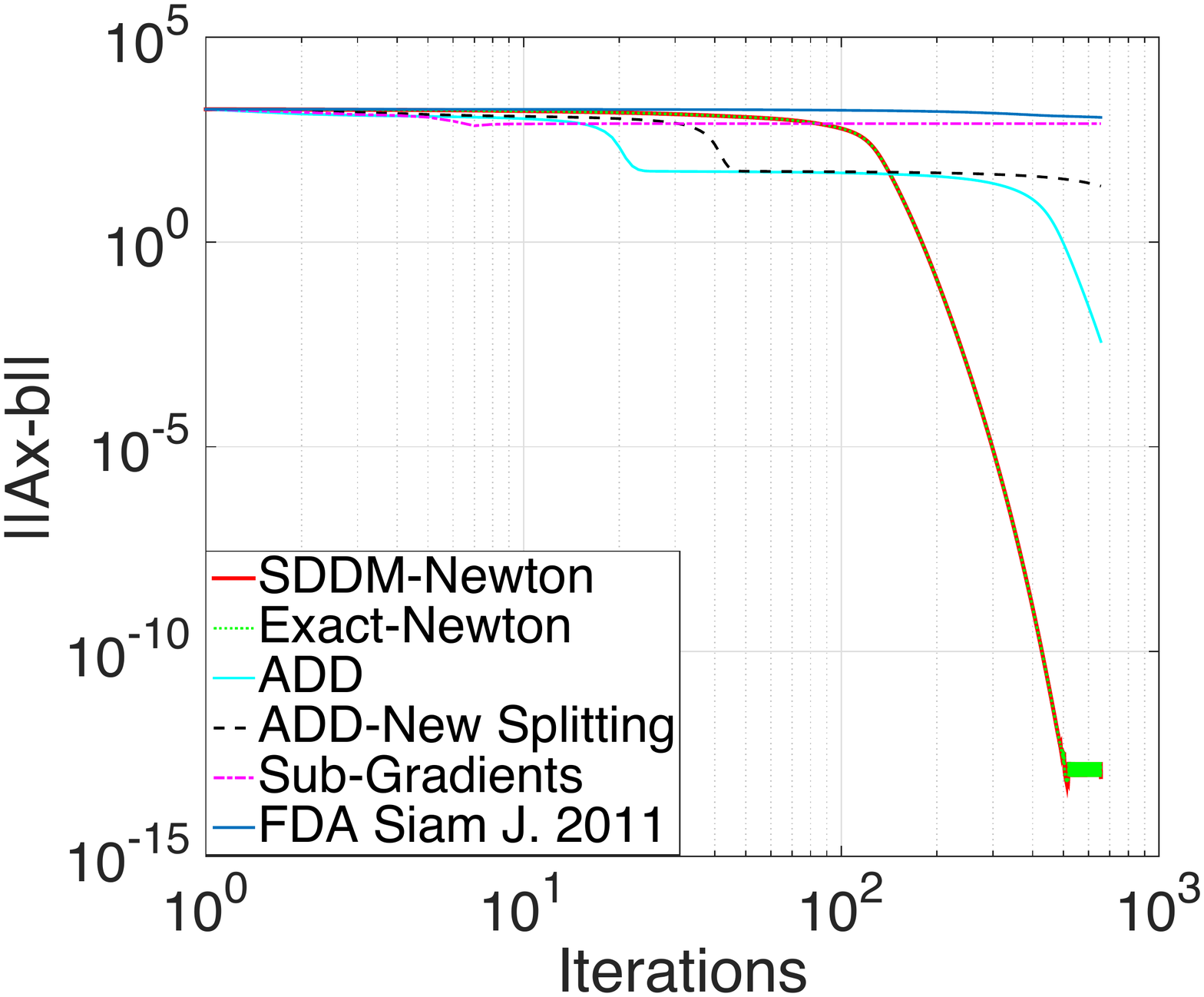}
}
\hfill\hspace{-.3in}\hfill
\subfigure[]{
	\label{fig:PerfCPTwo}
\includegraphics[width=0.49\textwidth]{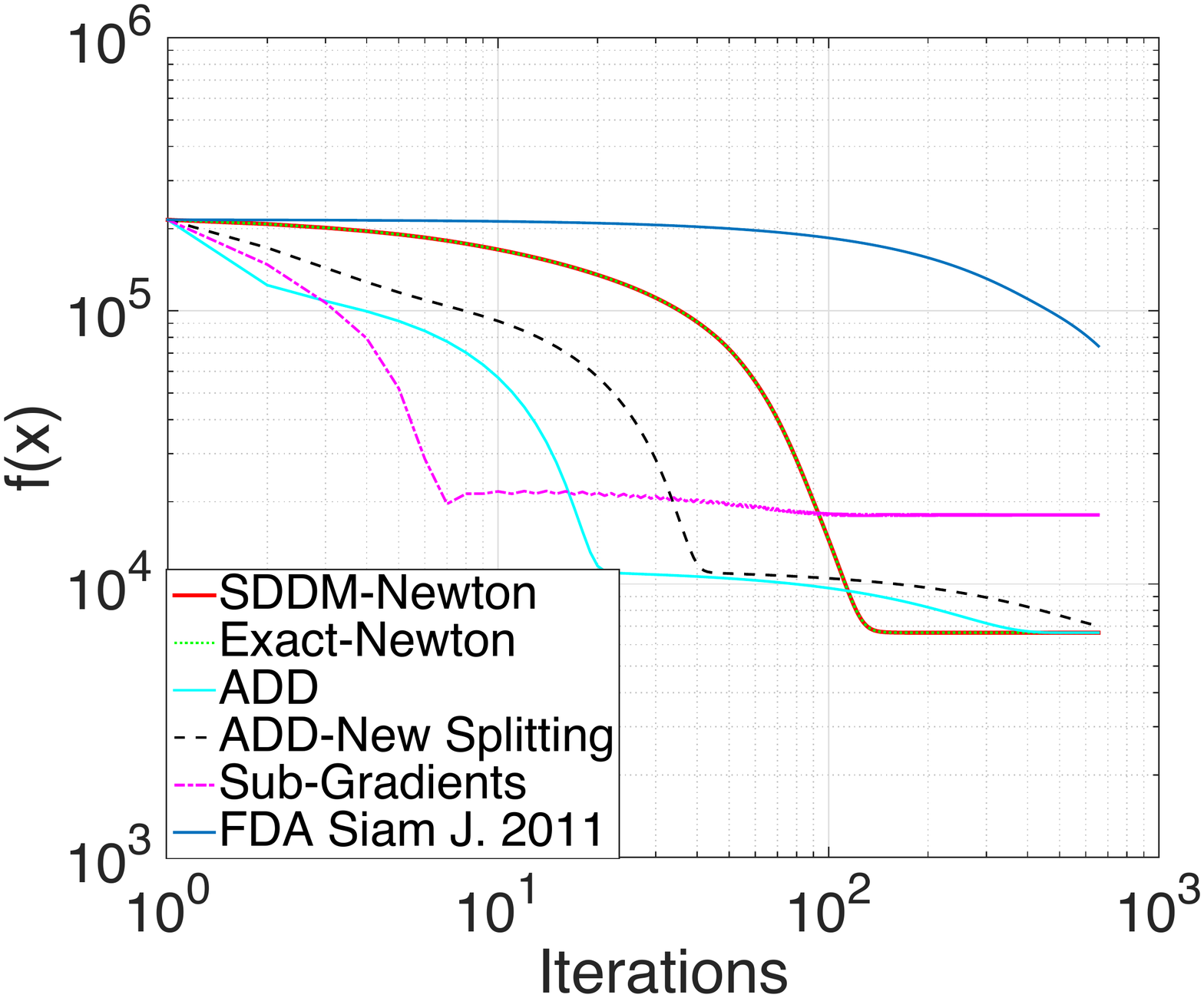}
}
\caption{Feasibility and objective results on a 120 node bar graph. Again in these sets of experiments, our method is capable of outperforming all other methods. Furthermore, SDDM-Newton is again capable of tracing the exact Newton method.}
\end{figure*}

\subsubsection{Measuring Message Complexity}
Though successful, the experiments performed in the previous section show accuracy improvements without demonstrating per-iteration message complexities needed. To have a fair comparison to state-of-the-art methods, in this section we report such results on four different network topologies. Here, we show that our method requires a relatively slight increase in message complexity to trace the exact Newton direction.

These per-iteration values were determined by deriving bounds to each of the benchmark algorithms. For all methods except SDDM-Newton and that in~\cite{c5}, such complexity can be bounded by $\mathcal{O}(d_{\text{max}})$ with $d_{\text{max}}$ being the maximal degree. As for SDDM-Newton, the per-iteration complexity is upper-bounded by $\mathcal{O}(\bm{\kappa}(\mathcal{L}_{\mathcal{G}}) d_{\text{max}})$, with $\bm{\kappa}(\mathcal{L}_{\mathcal{G}})$ being the condition number of the graph Laplacian. For the algorithm in~\cite{c5} the message complexity satisfies $\mathcal{O}(\bm{\kappa}(\mathcal{L}_{\mathcal{G}}) d_{\text{max}} \log n)$ with $n$ being the total number of nodes in $\mathcal{G}$. Immediately, we recognize that our algorithm is faster by a factor of $\log n$ compared to that in~\cite{c5}. Compared to other techniques, however, our method is slower by a factor of $\bm{\kappa}(\mathcal{L}_{\mathcal{G}})$. 
\begin{wrapfigure}{l}{0.6\textwidth}
\begin{center}
\includegraphics[width=0.6\textwidth, height=2.1in]{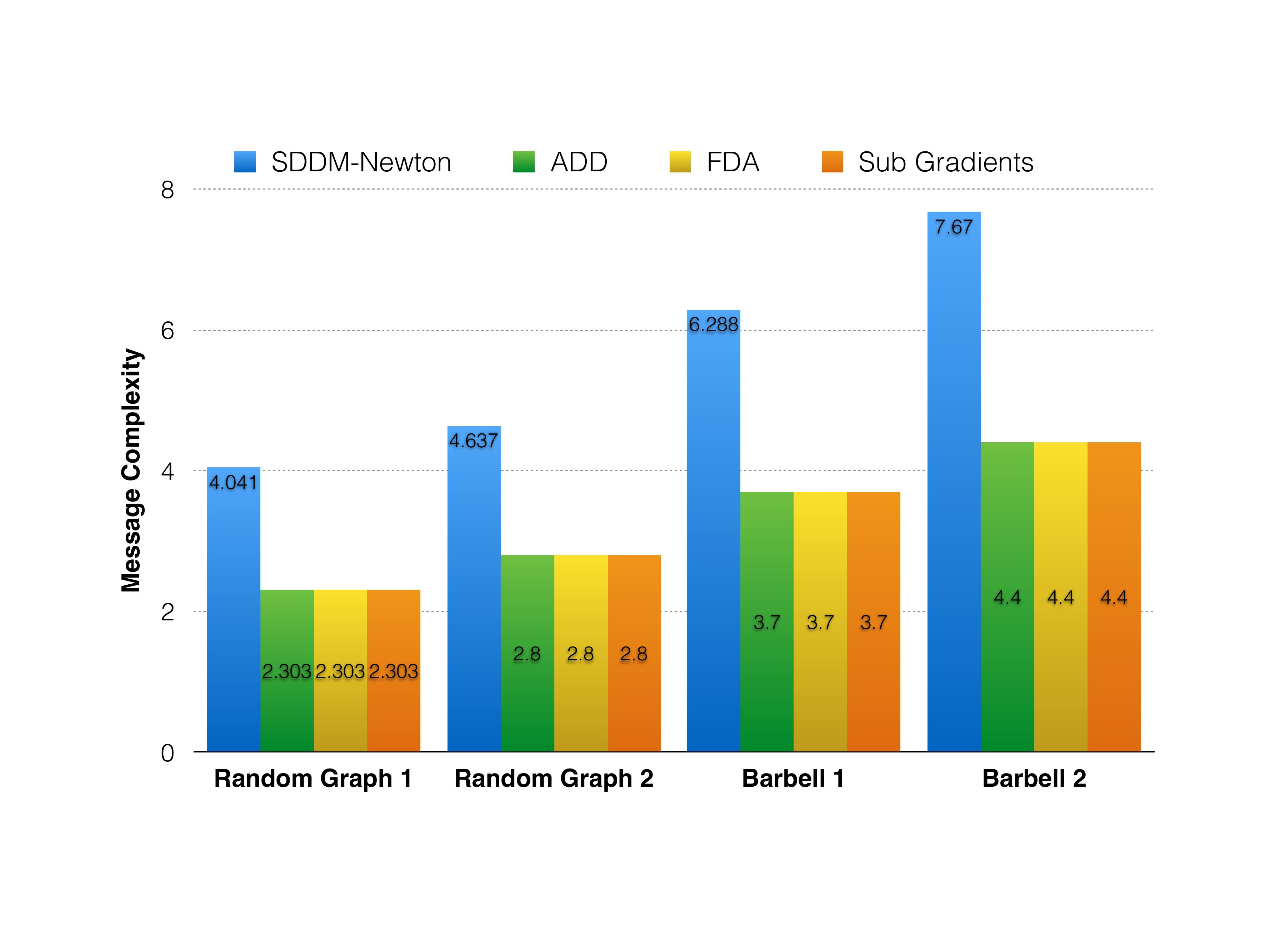}
\end{center}
\caption{Message complexity results comparing our method to state-of-the-art techniques on four different network topologies. Graphs 1 and 2 correspond to random networks with 20 nodes, 60 edge and 100 nodes and 500 edge, respectively. Graphs 3 and 4 report message complexities needed for two barbell graphs. }
\label{Fig:barbell}
\end{wrapfigure}

To better quantify such a difference, we perform four sets of experiments on two randomly generated and two barbell networks with varying size. The random networks consisted of 20 nodes, 60 edges, and 120 nodes and 500 edges, respectively. Moreover, the barbell graphs followed the same construction of the previous section with sizes varying from 60 to 120 nodes. Comparison results showing the logarithm of the message complexity are reported in the bar graph of Figure~\ref{Fig:barbell}. These demonstrate that SDDM-Newton requires a slight increase in the message complexity to trace the exact Newton direction. 

\section{Conclusion}

In this paper we proposed a distributed solver for linear systems described by SDDM matrices. Our approach distributes that in~\cite{c11} by proposing the usage of an inverse approximated chain which can be computed in a distributed fashion. Precisely, two solvers were proposed. The first required full communication in the network, while the second restricts communication to the R-Hop neighborhood between the nodes. 

We applied our solver to network flow optimization. This resulted in an efficient and accurate distributed second-order method capable of tracing exact Newton computed in a centralized fashion. We showed that similar to standard Newton, our methods are capable of achieving superlinear convergence in a neighborhood of the optimal solution. We extensively evaluated the proposed method on both randomly generated and barbell graphs. Results demonstrate that our method outperforms state-of-the-art techniques. 


\section{SDDM Solver Proofs}
In this appendix, we provide the complete and comprehensive proofs specific for the distributed SDDM solver. 
\begin{lemma}
Let $\bm{Z}_0\approx_{\epsilon}\bm{M}^{-1}_0$, and $\tilde{\bm{x}} = \bm{Z}_0\bm{b}_0$. Then $\tilde{\bm{x}}$ is $\sqrt{2^{\epsilon}(e^{\epsilon} - 1)}$ approximate solution of $\bm{M}_{0}\bm{x}=\bm{b}_{0}$.
\end{lemma}
\begin{proof}
Let $\bm{x}^{\star}\in \mathbb{R}^{n}$ be the solution of $\bm{M}_0\bm{x}=\bm{b}_{0}$, then
\begin{align}\label{app_solut_for_approx_inv}
||\bm{x}^{\star} - \tilde{\bm{x}}||^2_{\bm{M}_0} =(\bm{x}^{\star} - \tilde{\bm{x}})^{\mathsf{T}}\bm{M}_0(\bm{x}^{\star} - \tilde{\bm{x}}) = (\bm{x}^{\star})^{\mathsf{T}}\bm{M}_0\bm{x}^{\star} + (\tilde{\bm{x}})^{\mathsf{T}}\bm{M}_0\tilde{\bm{x}} - 2(\bm{x}^{\star})^{\mathsf{T}}\bm{M}_0\tilde{\bm{x}}
\end{align}
Now, consider each term in (\ref{app_solut_for_approx_inv}) separately:
\begin{enumerate}
\item $(\bm{x}^{\star})^{\mathsf{T}}\bm{M}_0\tilde{\bm{x}} = \bm{b}^{\mathsf{T}}_0\bm{M}^{-1}_0\bm{M}_0\bm{Z}_0\bm{b}_0 = \bm{b}^{\mathsf{T}}_0\bm{Z}_0\bm{b}_0$

\item $(\bm{x}^{\star})^{\mathsf{T}}\bm{M}_0\bm{x}^{\star} = \bm{b}^{\mathsf{T}}_0\bm{M}^{-1}_0\bm{M}_0\bm{M}^{-1}_0\bm{b}_0 = \bm{b}^{\mathsf{T}}_0\bm{M}^{-1}_0\bm{b}_0 \le e^{\epsilon}\bm{b}^{\mathsf{T}}_0\bm{Z}_0\bm{b}_0$

\item $\tilde{\bm{x}}^{\mathsf{T}}\bm{M}_0\tilde{\bm{x}} = \bm{b}^{\mathsf{T}}_0\bm{Z}_0\bm{M}_0\bm{Z}_0\bm{b}_0 \le e^{\epsilon}\bm{b}^{\mathsf{T}}_0\bm{Z}_0\bm{b}_0$
\end{enumerate}
Note that in the last step we used the fact that $\bm{Z}_0\approx_{\epsilon}\bm{M}^{-1}_0$ implies $\bm{M}_0 \approx_{\epsilon}\bm{Z}^{-1}_0$. Therefore, (\ref{app_solut_for_approx_inv}) can be rewritten as:
\begin{equation}\label{inter_1}
||\bm{x}^{\star} - \tilde{\bm{x}}||^2_{\bm{M}_0} \le 2(e^{\epsilon} - 1)\bm{b}^{\mathsf{T}}_0\bm{Z}_0\bm{b}_0
\end{equation}
Combining (\ref{inter_1}) with $\bm{b}^{\mathsf{T}}_0\bm{Z}_0\bm{b}_0 = (\bm{x}^{\star})^{\mathsf{T}}\bm{M}_0\bm{Z}_0\bm{M}_0\bm{x}^{\star}  \le e^{\epsilon}(\bm{x}^{\star})^{\mathsf{T}}\bm{M}_0\bm{x}^{\star}$:
\begin{align*}
&||\bm{x}^{\star} - \tilde{\bm{x}}||^2_{\bm{M}_0} \le 2(e^{\epsilon} - 1)e^{\epsilon}(\bm{x}^{\star})^{\mathsf{T}}\bm{M}_0\bm{x}^{\star} = 2(e^{\epsilon} - 1)e^{\epsilon}||\bm{x}^{\star}||^2_{\bm{M}_0}
\end{align*}
\end{proof}
\begin{lemma}
Let $\bm{M}_0 = \bm{D}_0 - \bm{A}_0$ be the standard splitting of $\bm{M}_{0}$. Let $\bm{Z}^{\prime}_0$ be the operator defined by $\text{DistrRSolve}([\{[\bm{M}_0]_{k1},\ldots, [\bm{M}_0]_{kn}\}, [\bm{b}_0]_k, d)$ (i.e., $\bm{x}_0 = \bm{Z}^{\prime}_0\bm{b}_0$). Then
\begin{equation*}
\bm{Z}^{\prime}_0\approx_{\epsilon_d} \bm{M}^{-1}_0
\end{equation*}
Moreover, Algorithm~\ref{Algo:DisRudeApprox} requires  $\mathcal{O}\left(dn^2\right)$ time steps. 
\end{lemma}
\begin{proof}
The proof commences by showing that $\left(\bm{D}^{-1}_0\bm{A}_0\right)^r$ and $\left(\bm{A}_0\bm{D}^{-1}_0\right)^{-r}$ have a sparsity pattern corresponding to the r-hop neighborhood for any $r\in \mathbb{N}$. This case be shown using induction as follows
\begin{enumerate}
\item If $r = 1$, we have
\begin{equation*}
[\bm{A}_0\bm{D}^{-1}_0]_{ij} = \begin{cases} \frac{[\bm{A}_0]_{ij}}{[\bm{D}_0]_{ii}} &\mbox{if } j: \bm{v}_j\in \mathbb{N}_{1}(\bm{v}_i), \\ 
0 & \mbox{otherwise }. \end{cases}
\end{equation*}
Therefore, $\bm{A}_0\bm{D}^{-1}_0$ has sparsity pattern corresponding to the 1-Hop neighborhood.
\end{enumerate}
\begin{enumerate}
\item[2.] Assume that $(\bm{A}_0\bm{D}^{-1}_0)^p$ has a sparsity patter corresponding to the p-hop neighborhood for all $p: 1\le p\le r-1$.
\end{enumerate}
\begin{enumerate}
\item[3.] Now, consider $(\bm{A}_0\bm{D}^{-1}_0)^{r}$, where
\begin{equation}\label{sparisty_lemma_Bla}
[(\bm{A}_0\bm{D}^{-1}_0)^{r}]_{ij} = \sum_{k=1}^{n}[(\bm{A}_0\bm{D}^{-1}_0)^{r-1}]_{ik}[\bm{A}_0\bm{D}^{-1}_0]_{kj}
\end{equation}
Since $\bm{A}_0\bm{D}^{-1}_0$ is non negative then it is easy to see that $[(\bm{A}_0\bm{D}^{-1}_0)^{r}]_{ij}\ne 0$ if and only if there exists $k$ such that $\bm{v}_k\in \mathbb{N}_{r-1}(\bm{v}_i)$ and $\bm{v}_k\in \mathbb{N}_1(\bm{v}_j)$ (i.e., $\bm{v}_j\in \mathbb{N}_{r}(\bm{v}_i)$).
\end{enumerate}
For $\bm{D}^{-1}_0\bm{A}_0$, the same results can be derived similarly.

Please notice that in \textbf{Part One} of $\text{DistrRSolve}$ algorithm node $\bm{v}_k$ computes (in a distributed fashion) the components $[\bm{b}_1]_k$ to $[\bm{b}_d]_{k}$ using the inverse approximated chain $\mathcal{C} = \{\bm{A}_0,\bm{D}_0,\bm{A}_1, \bm{D}_1,\ldots, \bm{A}_d, \bm{D}_d\}$. Formally, 
\begin{align*}
\bm{b}_{i} &= \left[\bm{I} + \left(\bm{A}_0\bm{D}^{-1}_0\right)^{2^{i-1}}\right]\bm{b}_{i-1} = \bm{b}_{i-1} + \left(\bm{A}_0\bm{D}^{-1}_0\right)^{2^{i-2}}\cdot\left(\bm{A}_0\bm{D}^{-1}_0\right)^{2^{i-2}}\bm{b}_{i-1}
\end{align*}
Clearly, at the $i^{th}$ iteration node $\bm{v}_k$ requires the $k^{th}$ row of $\left(\bm{A}_{0}\bm{D}^{-1}_{0}\right)^{2^{i-2}}$ (i.e., the $k^{th}$ row from the previous iteration) in addition to the $j^{th}$ row of $\left(\bm{A}_0\bm{D}^{-1}_0\right)^{2^{i-2}}$ from all nodes $\bm{v}_j\in \mathbb{N}_{2^{i-1}}\left(\bm{v}_k\right)$ to compute the $k^{th}$ row of $\left(\bm{A}_{0}\bm{D}^{-1}_{0}\right)^{2^{i-1}}$. 

For computing $\left[\left(\bm{A}_0\bm{D}^{-1}_0\right)^{2^{i-1}}\right]_{kj}$, node $\bm{v}_k$ requires the $k^{th}$ row and $j^{th}$ column of $\left(\bm{A}_0\bm{D}^{-1}_0\right)^{2^{i-2}}$. The problem, however, is that node $\bm{v}_j$ can only send the $j^{th}$ row of $\left(\bm{A}_0\bm{D}^{-1}_0\right)^{2^{i-2}}$ which can be easily seen not to be see that symmetric. To overcome this issue, node $\bm{v}_k$ has to compute the $j^{th}$ column of $\left(\bm{A}_0\bm{D}^{-1}_0\right)^{2^{i-2}}$ based on its $j^{th}$ row.  The fact that $\bm{D}^{-1}_0\left(\bm{A}_0\bm{D}^{-1}_0\right)^{2^{i-2}}$ is symmetric, manifests that for $r = 1,\ldots, n$

\begin{equation*}
\frac{\left[\left(\bm{A}_0\bm{D}^{-1}_0\right)^{2^{i-2}}\right]_{rj}}{[\bm{D}_0]_{rr}} = \frac{\left[\left(\bm{A}_0\bm{D}^{-1}_0\right)^{2^{i-2}}\right]_{jr}}{[\bm{D}_0]_jj}
\end{equation*}
Hence, for all $r=1,\ldots,n$
\begin{equation}\label{columns_from_rowsTwo}
\left[\left(\bm{A}_0\bm{D}^{-1}_0\right)^{2^{i-2}}\right]_{rj} = \frac{[\bm{D}_0]_{rr}}{[\bm{D}_0]_{jj}}\left[\left(\bm{A}_0\bm{D}^{-1}_0\right)^{2^{i-2}}\right]_{jr}
\end{equation}     
Now, lets analyze the time complexity of computing components ${[b_1]_k, [b_2]_k,\ldots, [b_d]_k}$.\\\newline
 \textbf{Time Complexity Analysis:} At each iteration $i$, node $\bm{v}_k$ receives the $j^{th}$ row of $\left(\bm{A}_0\bm{D}^{-1}_0\right)^{2^{i-2}}$ from all nodes $\bm{v}_j\in \mathbb{N}_{2^{i-1}}(\bm{v}_k)$. using Equation~\ref{columns_from_rowsTwo}, node $\bm{v}_{k}$ computes the corresponding columns as well as the product of these columns with the $k^{th}$ row of $\left(\bm{A}_0\bm{D}^{-1}_0\right)^{2^{i-2}}$. Therefore, the time complexity at the $i^{th}$ iteration is $\mathcal{O}\left(n^2 + \text{diam}\left(\mathcal{G}\right)\right)$, where $n^2$ is responsible for the $k^{th}$ row computation, and $\text{diam}\left(\mathcal{G}\right)$ represents the communication cost between the nodes.  Using the fact that $\text{diam}\left(\mathbb{G}\right)\le n$, the total complexity of $\textbf{Part One}$ in $\text{DistrRSolve}$ algorithm is $\mathcal{O}\left(dn^2\right)$.

In \textbf{Part Two}, node $\bm{v}_k$ computes (in a distributed fashion) ${[\tilde{\bm{x}}_{d-1}]_k, [\tilde{\bm{x}}_{d-2}]_k,\ldots, [\tilde{\bm{x}}_{0}]_k}$ using the same inverse approximated chain $\mathcal{C} = \{\bm{A}_0,\bm{D}_0,\bm{A}_1, \bm{D}_1,\ldots, \bm{A}_d, \bm{D}_d\}$. 
\begin{align}
\bm{x}_{i} = \frac{1}{2}\bm{D}^{-1}_0\bm{b}_{i} &+ \frac{1}{2}\left[\bm{I} + (\bm{D}^{-1}_0\bm{A}_0)^{2^i}\right]\bm{x}_{i+1} = \frac{1}{2}\bm{D}^{-1}_0\bm{b}_{i} + \frac{1}{2}\bm{x}_{i+1}  \\\nonumber
&\hspace{3em}+\frac{1}{2}\left(\bm{D}^{-1}_0\bm{A}_0\right)^{2^{i-1}}\left(\bm{D}^{-1}_0\bm{A}_0\right)^{2^{i-1}}\bm{x}_{i+1}
\end{align}
for $i=d-1, \ldots, 1$. Thus,
\begin{equation*}
\bm{x}_{0} = \frac{1}{2}\bm{D}^{-1}_0\bm{b}_{0} + \frac{\bm{x}_1}{2} + \frac{1}{2}\left(\bm{D}^{-1}_0\bm{A}_0\right)\bm{x}_{1}
\end{equation*}
Similar to the analysis of \textbf{Part One} of $\text{DistrRSolve}$ algorithm the time complexity of \textbf{Part Two} as well as the time complexity of the whole algorithm is $\mathcal{O}\left(dn^2\right)$.

Finally, using Lemma~\ref{Rude_Alg_guarantee_Lemma} for the inverse approximated chain $\mathcal{C} = \{\bm{A}_0,\bm{D}_0, \bm{A}_1, \bm{D}_1,\ldots, \bm{A}_d, \bm{D}_d\}$ yields:

\begin{equation*}
\bm{Z}^{\prime}_0\approx_{\epsilon_d} \bm{M}^{-1}_0.
\end{equation*}
\end{proof}

\begin{lemma}
Let $\bm{M}_0 = \bm{D}_0 - \bm{A}_0$ be the standard splitting. Further, let $\epsilon_d < \frac{1}{3}\ln2$ in the nverse approximated chain $\mathcal{C} = \{\bm{A}_0,\bm{D}_0,\bm{A}_1, \bm{D}_1,\ldots, \bm{A}_d, \bm{D}_d\}$. Then $\text{DistrESolve}(\{[\bm{M}_0]_{k1},\ldots, [\bm{M}_0]_{kn}\}, [b_0]_k, d, \epsilon)$ requires $\mathcal{O}\left(\log\frac{1}{\epsilon}\right)$ iterations to return the $k^{th}$ component of the $\epsilon$ close approximation for $\bm{x}^{\star}$.

\end{lemma}
\begin{proof}
Notice that iterations in $\text{DistrESolve}$ corresponds to Preconditioned Richardson Iteration:
\begin{equation*}
\bm{y}_t = \left[\bm{I} - \bm{Z}'_0\bm{M}_0\right]\bm{y}_{t-1} + \bm{Z}_0\bm{b}_0
\end{equation*}
where $\bm{Z}'_0$ is the operator defined by $\text{DistrRSolve}$ and $\bm{y}_0 = \bm{0}$. Therefore, from Lemma \ref{Rude_Dec_Alg_guarantee_Lemma}:
\begin{equation*}
\bm{Z}'_0 \approx_{\epsilon_d} \bm{M}^{-1}_0 
\end{equation*}
Finally, applying Lemma \ref{Exact_Alg_guarantee_lemma} gives that $\text{DistrESolve}$ algorithm needs $\mathcal{O}\left(\log\frac{1}{\epsilon}\right)$ iterations to $k^{th}$ component of the $\epsilon$ approximated solution for $\bm{x}^{\star}$.
\end{proof}

\begin{lemma}
Let $\bm{M}_0 =\bm{D}_0 - \bm{A}_0$ be the standard splitting. Further, let $\epsilon_d < \frac{1}{3}\ln2$ in the inverse approximated chain $\mathcal{C} = \{\bm{A}_0,\bm{D}_0,\bm{A}_1, \bm{D}_1,\ldots, \bm{A}_d, \bm{D}_d\}$. Then, $\text{DistrESolve}(\{[\bm{M}_0]_{k1},\ldots, [\bm{M}_0]_{kn}\}, [\bm{b}_0]_k, d, \epsilon)$ requires $\mathcal{O}\left(dn^2\log(\frac{1}{\epsilon})\right)$ time steps.
\end{lemma}
\begin{proof}
Each iteration of $\text{DistrESolve}$ algorithm calls $\text{DistRSolve}$ routine, therefore, using the above the total time complexity of f $\text{DistrESolve}$ algorithm is $\mathcal{O}\left(dn^2\log(\frac{1}{\epsilon})\right)$ time steps
\end{proof}

\begin{lemma}
Let $\bm{M}_0 = \bm{D}_0 - \bm{A}_0$ be the standard splitting. Further, let $\epsilon_d < \sfrac{1}{3}\ln2$. Then Algorithm~8 requires $\mathcal{O}\left(\log\frac{1}{\epsilon}\right)$ iterations to return the $k^{th}$ component of the $\epsilon$ close approximation to $\bm{x}^{\star}$.  
\end{lemma}
\begin{proof}
Please note that the iterations of $\text{EDistRSolve}$ correspond to a distributed version of the  preconditioned Richardson iteration scheme
\begin{equation*}
\bm{y}_{t} = [\bm{I} - \bm{Z}^{\prime}_0\bm{M}_0]\bm{y}_{t-1} + \bm{Z}^{\prime}_0\bm{b}_0
\end{equation*}
with $\bm{y}_0 = 0$ and $\bm{Z}^{\prime}_0$ being the operator defined by $\text{RDistRSolve}$. From Lemma~3.8 it is clear that
$\bm{Z}'_0 \approx_{\epsilon_d}\bm{M}^{-1}_0$. 
Applying Lemma~2.12, provides that $\text{EDistRSolve}$ requires $ \mathcal{O}\left(\log\sfrac{1}{\epsilon}\right)$ iterations to return the $k^{th}$ component of the $\epsilon$ close approximation to $\bm{x}^{\star}$. Finally, since $\text{EDistRSolve}$ uses procedure $\text{RDistRSolve}$ as a subroutine, it follows that for each node $\bm{v}_k$ only communication between the R-hope neighbors is allowed. 
\end{proof}

\begin{lemma}
Let $\bm{M}_0 = \bm{D}_0 - \bm{A}_0$ be the standard splitting and let $\epsilon_d < \sfrac{1}{3}\ln 2 $, then $\text{EDistRSolve}$ requires 
$\mathcal{O}\left(\left(\sfrac{2^d}{R}\alpha + \alpha Rd_{max}\right)\log\left(\sfrac{1}{\epsilon}\right)\right)$ time steps. Moreover, for each node $\bm{v}_k$, $\text{EDistRSolve}$ only uses information from the R-Hop neighbors.
\end{lemma}

\begin{proof}
Notice that at each iteration $\text{EDistRSolve}$ calls  $\text{RDistRSolve}$ as a subroutine, therefore, for each node $v_k$ only R-hop communication is allowed. Lemma~3.8 gives that the time complexity of each iteration is $\mathcal{O}\left(\frac{2^d}{R}\alpha + \alpha Rd_{max}\right)$, and using Lemma~3.9 immediately gives that the time complexity of $\mathcal{O}\left(\left(\sfrac{2^d}{R}\alpha + \alpha Rd_{max}\right)\log\left(\sfrac{1}{\epsilon}\right)\right)$.
\end{proof}

\section{Distributed Newton Lemmas}
\begin{lemma}
The dual objective $q(\bm{\lambda})=\bm{\lambda}^{\mathsf{T}}(\bm{A}\bm{x}(\bm{\lambda})-b)-\sum_{e}\bm{\Phi}_{e}(\bm{x}(\lambda))$ abides by the following two properties~[6]:
\begin{enumerate}
\item The dual Hessian, $\bm{H}(\bm{\lambda})$, is a weighted Laplacian of $\mathcal{G}$:
\begin{equation*}
\bm{H}(\bm{\lambda})=\nabla^{2}q(\bm{\lambda})=\bm{A}\left[\nabla^{2}f(\bm{x}(\bm{\lambda}))\right]^{-1}\bm{A}^{\mathsf{T}} 
\end{equation*}
\item The dual Hessian $\bm{H}(\bm{\lambda})$ is Lispshitz  continuous with respect to the Laplacian norm (i.e., $||\cdot||_{\mathcal{L}}$) where $\mathcal{L}$ is the unweighted laplacian satisfying $\mathcal{L}=\bm{A}\bm{A}^{\mathsf{T}}$ with $\bm{A}$ being the incidence matrix of $\mathcal{G}$. Namely, $\forall \bm{\lambda}, \bar{\bm{\lambda}}$: 
\begin{equation*}
||\bm{H}(\bar{\bm{\lambda}})-\bm{H}(\bm{\lambda})||_{\mathcal{L}} \leq B ||\bar{\bm{\lambda}}-\bm{\lambda}||_{\mathcal{L}}
\end{equation*}
with $B=\frac{\mu_{n}(\mathcal{L})\delta}{\gamma \sqrt{\mu_{2}(\mathcal{L})}}$ where $\mu_{n}(\mathcal{L})$ and $\mu_{2}(\mathcal{L})$ denote the largest and second smallest eigenvalues of the Laplacian $\mathcal{L}$. 
\end{enumerate}
\end{lemma}
\begin{proof}
For the first part see Lemma 1 in [6]. So, lets prove the second part:\\
Lets denote $\boldsymbol{W(\lambda)} = [\nabla^2f(x(\lambda))]^{-1}$, then:
\begin{equation}\label{eq_n_1}
\boldsymbol{H(\bar{\lambda})} - \boldsymbol{H(\lambda)} = \boldsymbol{A[W(\bar{\lambda}) - W(\lambda)]A^{\mathsf{T}}}
\end{equation}
Using that $||\boldsymbol{S}||_{\mathcal{L}} = \sup_{\boldsymbol{v}\in \bf{1}^{\perp}}\frac{||\boldsymbol{Sv}||_{\mathcal{L}}}{||\boldsymbol{v}||_{\mathcal{L}}}$ lets fix some $\boldsymbol{v}\in \bf{1}^{\perp}$ and consider the expression $||\boldsymbol{A[W(\bar{\lambda}) - W(\lambda)]A^{\mathsf{T}}v}||^2_{\mathcal{L}}$:
\begin{align*}
&\boldsymbol{||A[W(\bar{\lambda}) - W(\lambda)]A^{\mathsf{T}}v||}^2_{\mathcal{L}} = \\\nonumber
&\boldsymbol{v^{\mathsf{T}}A[W(\bar{\lambda}) - W(\lambda)]A^{\mathsf{T}}\mathcal{L}A[W(\bar{\lambda}) - W(\lambda)]A^{\mathsf{T}}v} = ^{1}\\\nonumber
&\boldsymbol{v^{\mathsf{T}}A[W(\bar{\lambda}) - W(\lambda)](A^{\mathsf{T}}A)^2[W(\bar{\lambda}) - W(\lambda)]A^{\mathsf{T}}v} \le^{^2}\\\nonumber
& \mu^2_n(\mathcal{L})\boldsymbol{v^{\mathsf{T}}A[W(\bar{\lambda}) - W(\lambda)]^2A^{\mathsf{T}}v} \le \\\nonumber
&\mu^2_n(\mathcal{L})\mu^2_n(\boldsymbol{|W(\bar{\lambda}) - W(\lambda)|})\boldsymbol{v^{\mathsf{T}}AA^{\mathsf{T}}v}  \\
&\hspace{7em}= 
\mu^2_n(\mathcal{L})\mu^2_n(\boldsymbol{|W(\bar{\lambda}) - W(\lambda)|})||\boldsymbol{v}||^2_{\mathcal{L}}
\end{align*}
We used in step ($^1$) $\mathcal{L} = \boldsymbol{AA^{\mathsf{T}}}$, and in step ($^2$) we used that $\mu_n(\boldsymbol{A^{\mathsf{T}}A}) = \mu_n(\boldsymbol{AA^{\mathsf{T}}}) = \mu_n(\mathcal{L})$. Therefore, we have:
\begin{equation}\label{eq_n_2}
||\boldsymbol{A[W(\bar{\lambda}) - W(\lambda)]A^{\mathsf{T}}}||_{\mathcal{L}} \le \mu_n(\mathcal{L})\mu_n(\boldsymbol{|W(\bar{\lambda}) - W(\lambda)|})
\end{equation}
Now, we upper bound the expression $\mu_n(\boldsymbol{|W(\bar{\lambda}) - W(\lambda)|})$:
\begin{align}\label{eq_n_3}
&\mu_n(\boldsymbol{|W(\bar{\lambda}) - W(\lambda)|}) \le \max_{e\in \mathcal{E}}\left|\frac{1}{\ddot{\boldsymbol{\Phi}}_e(\boldsymbol{x_e(\bar{\lambda})})} - \frac{1}{\ddot{\boldsymbol{\Phi}}_e(\boldsymbol{x_e(\lambda)})}\right| \le \\\nonumber
&\delta \max_{e\in \mathcal{E}}|\boldsymbol{x_e(\bar{\lambda}) - x_e(\lambda)}|
\end{align}
In the last transition we used Assumption 2. Now, using formulae for the derivative of the inverse function we have:
\begin{equation}
\left|\frac{\partial}{\partial \lambda_i}[\dot{\boldsymbol{\Phi}}]^{-1}(\boldsymbol{\lambda)}\right| = \frac{1}{\ddot{\boldsymbol{\Phi}}([\dot{\boldsymbol{\Phi}}]^{-1}(\boldsymbol{\lambda}))}\le \frac{1}{\gamma}
\end{equation}  
Hence, $[\dot{\boldsymbol{\Phi}}]^{-1}(\boldsymbol{\lambda)}$ is bounded, and therefore $[\dot{\boldsymbol{\Phi}}]^{-1}(\boldsymbol{\lambda)}$ is Lipshittz continuous with constant $L' = \frac{1}{\gamma}$. Now, because $\boldsymbol{x_e(\lambda)} = [\dot{\boldsymbol{\Phi}}]^{-1}(\lambda_i - \lambda_j)$, we have that $\boldsymbol{x_e(\lambda)}$ is Lipshitz continuous with corresponding constant $L'$. Hence, $\forall e\in \mathcal{E}$:
\begin{equation}\label{eq_n_4}
|\boldsymbol{x_e(\bar{\lambda}) - x_e(\lambda)}| \le \frac{1}{\gamma}||\boldsymbol{\bar{\lambda} - \lambda}||_2
\end{equation}
Now we are ready to prove the following

\begin{claim}
For all $e\in \mathcal{E}$ and for any $\boldsymbol{\bar{\lambda}, \lambda}$:
\begin{equation}\label{eq_n_5}
|\boldsymbol{x_e(\bar{\lambda}) - x_e(\lambda)}| \le \frac{1}{\gamma}\frac{||\boldsymbol{\bar{\lambda} - \lambda}||_{\mathcal{L}}}{\sqrt{\mu_2(\mathcal{L})}}
\end{equation}
\end{claim}

\begin{proof}
Consider three cases:
\begin{enumerate}
\item $\boldsymbol{\bar{\lambda} - \lambda} \in \bf{1}^{\perp}$. In this case, using that $\forall \boldsymbol{v}\in \bf{1}^{\perp}:$\\ 
$ \boldsymbol{v^{\mathsf{T}}v} \le \frac{\boldsymbol{v^{\mathsf{T}}\mathcal{L}v}}{\mu_2(\mathcal{L})}$ in (\ref{eq_n_4}):
\begin{equation*}
|\boldsymbol{x_e(\bar{\lambda}) - x_e(\lambda)}| \le \frac{1}{\gamma}||\boldsymbol{\bar{\lambda} - \lambda}||_2 \le \frac{1}{\gamma}\frac{\boldsymbol{||\bar{\lambda} - \lambda}||_{\mathcal{L}}}{\sqrt{\mu_2(\mathcal{L})}}
\end{equation*}
\item $\boldsymbol{\bar{\lambda} - \lambda} \in Span\{\bf{1}\}$ In this case $||\boldsymbol{\bar{\lambda} - \lambda}||_{\mathcal{L}} = 0$, and we have $\boldsymbol{\bar{\lambda}} = \boldsymbol{\lambda} + \alpha \bf{1}$, hence $\bar{\lambda}_i - \bar{\lambda}_j = \lambda_i + \alpha - (\lambda_j + \alpha) = \lambda_i - \lambda_j$, which gives:
\begin{equation*}
\boldsymbol{x_e(\bar{\lambda})} = [\dot{\boldsymbol{\Phi}}]^{-1}(\bar{\lambda}_i - \bar{\lambda}_j) = [\dot{\boldsymbol{\Phi}}]^{-1}(\lambda_i - \lambda_j) = \boldsymbol{x_e(\lambda)}
\end{equation*}
Therefore,
\begin{equation*}
|\boldsymbol{x_e(\bar{\lambda}) - x_e(\lambda)}| = 0 = \frac{1}{\gamma}\frac{||\boldsymbol{\bar{\lambda} - \lambda}||_{\mathcal{L}}}{\sqrt{\mu_2(\mathcal{L})}}.
\end{equation*}
Consequently, (\ref{eq_n_5}) is valid.
\item $\boldsymbol{\bar{\lambda} - \lambda} = \boldsymbol{u_1 + u_2}$, where $\boldsymbol{u_1}\in \bf{1}^{\perp},$ $\boldsymbol{u_2}\in Span\{\bf{1}\}$. In this case $||\boldsymbol{\bar{\lambda} - \lambda}||_{\mathcal{L}} = ||\boldsymbol{u_1}||_{\mathcal{L}}$, and $\bar{\lambda}_i - \bar{\lambda}_j = \lambda_i - \lambda_j + u_1(i) - u_1(j)$. Notice that the same expression for $\bar{\lambda}_i - \bar{\lambda}_j$ will be in the case when $\boldsymbol{\bar{\lambda}} = \boldsymbol{\lambda + u_1}$. Hence, using the first case which proves the claim:
\begin{align*}
&|\boldsymbol{x_e(\bar{\lambda}) - x_e(\lambda)}|  = |[\dot{\boldsymbol{\Phi}}]^{-1}(\bar{\lambda}_i - \bar{\lambda}_j) - [\dot{\boldsymbol{\Phi}}]^{-1}(\lambda_i - \lambda_j)| = \\\nonumber
&|[\dot{\boldsymbol{\Phi}}]^{-1}(\lambda_i - \lambda_j + u_1(i) - u_1(j))  - [\dot{\boldsymbol{\Phi}}]^{-1}(\lambda_i - \lambda_j)| \le \frac{1}{\gamma}||\boldsymbol{u_1}||_2 \le 
\frac{1}{\gamma}\frac{||\boldsymbol{u_1}||_{\mathcal{L}}}{\sqrt{\mu_2(\mathcal{L})}} = \frac{1}{\gamma}\frac{||\boldsymbol{\bar{\lambda} - \lambda}||_{\mathcal{L}}}{\sqrt{\mu_2(\mathcal{L})}}
\end{align*}

\end{enumerate}
\end{proof}
Combining the above claim with (\ref{eq_n_3}) gives:
\begin{equation}\label{eq_n_6}
\mu_n(|\boldsymbol{W(\bar{\lambda}) - W(\lambda)}|) \le \frac{\delta}{\gamma}\frac{||\boldsymbol{\bar{\lambda} - \lambda}||_{\mathcal{L}}}{\sqrt{\mu_2(\mathcal{L})}}
\end{equation}
Using (\ref{eq_n_6}) in (\ref{eq_n_2}) leads us to:
\begin{equation*}
||\boldsymbol{H(\bar{\lambda}) - H(\lambda)}||_{\mathcal{L}} \le B||\boldsymbol{\bar{\lambda} - \lambda}||_{\mathcal{L}},
\end{equation*}
where $B = \frac{\mu_n(\mathcal{L})\delta}{\gamma\sqrt{\mu_2(\mathcal{L})}}$.
\end{proof}

\begin{lemma}
If the dual Hessian $\bm{H}(\bm{\lambda})$ is Lipschitz continuous with respect to the Laplacian norm $||\cdot||_{\mathcal{L}}$ (i.e., Lemma 7), then for any $\bm{\lambda}$ and $\hat{\bm{\lambda}}$ we have 
\begin{equation*}
||\nabla q(\hat{\bm{\lambda}})-\nabla q({\bm{\lambda}}) - \bm{H}(\bm{\lambda})(\boldsymbol{\hat{\lambda}}-\bm{\lambda})||_{\mathcal{L}} \leq \frac{B}{2}||\hat{\bm{\lambda}} - \bm{\lambda}||_{\mathcal{L}}^{2}.
\end{equation*}
\end{lemma}
\begin{proof}
We apply the result of Fundamental Theorem of Calculus for the gradient $\nabla q$ which implies for any vectors $\boldsymbol{\lambda}$ and $\boldsymbol{\hat{\lambda}}$ in $\mathbb{R}^{n}$ we can write 
\begin{equation}\label{fundamental_result2}
\nabla q(\boldsymbol{\hat{\lambda}})= \nabla q(\boldsymbol{\lambda})+\int_{0}^1 \boldsymbol{H}(\boldsymbol{\lambda}+t(\boldsymbol{\hat{\lambda}-\lambda)}) (\boldsymbol{\hat{\lambda}-\lambda})\ dt,
\end{equation}
We proceed by adding and subtracting $\boldsymbol{H(\lambda)(\hat{\lambda}-\lambda)}$ to the integral in the right hand side of (\ref{fundamental_result2}). It follows that 
\begin{align}\label{taylor_first_two terms_44}
&\nabla q(\boldsymbol{\hat{\lambda}})= \nabla q(\boldsymbol{\lambda})+ \\\nonumber
&\int_{0}^1 \left[\boldsymbol{H}(\boldsymbol{\lambda}+t(\boldsymbol{\hat{\lambda}-\lambda}))-\boldsymbol{H(\lambda)}\right] (\boldsymbol{\hat{\lambda}-\lambda})+\boldsymbol{H(\lambda)(\hat{\lambda}-\lambda)}\ dt.
\end{align}
we can separate the integral in (\ref{taylor_first_two terms_44}) into two integrals as
\begin{align}\label{taylor_first_two terms_45}
&\nabla q(\boldsymbol{\hat{\lambda}})= \nabla q(\boldsymbol{\lambda})+\int_{0}^1 \left[\boldsymbol{H}(\boldsymbol{\lambda}+t(\boldsymbol{\hat{\lambda}-\lambda)})-\boldsymbol{H(\lambda)}\right] (\boldsymbol{\hat{\lambda}-\lambda})\ dt\\\nonumber
&+ \int_{0}^1 \boldsymbol{H(\lambda)(\hat{\lambda}-\lambda)}\ dt.
\end{align}
The second integral in the right hand side of (\ref{taylor_first_two terms_45}) does not depend on $t$ and we can simplify the integral as $ \boldsymbol{H(\lambda)(\hat{\lambda}-\lambda)}$. This simplification implies that we can rewrite (\ref{taylor_first_two terms_45}) as
\begin{align}\label{taylor_first_two terms_46}
&\nabla q(\boldsymbol{\hat{\lambda}})= \nabla q(\boldsymbol{\lambda})+ \boldsymbol{H(\lambda)(\hat{\lambda}-\lambda)}+\\\nonumber
&\int_{0}^1 \left[\boldsymbol{H}(\boldsymbol{\lambda}+t(\boldsymbol{\hat{\lambda}-\lambda}))-\boldsymbol{H(\lambda)}\right] (\boldsymbol{\hat{\lambda}-\lambda})\ dt,
\end{align}
By rearranging terms in (\ref{taylor_first_two terms_46}) and taking the norm of both sides we obtain 
\begin{align}\label{taylor_first_two terms_47}
&\|\nabla q(\boldsymbol{\hat{\lambda}})= \nabla q(\boldsymbol{\lambda}) - \boldsymbol{H(\lambda)(\hat{\lambda}-\lambda)}\|_{\mathcal{L}}= \\\nonumber
&\left|\left|\int_{0}^1 \left[\boldsymbol{H}(\boldsymbol{\lambda}+t(\boldsymbol{\hat{\lambda}-\lambda}))-\boldsymbol{H(\lambda)}\right] (\boldsymbol{\hat{\lambda}-\lambda})\ dt\right|\right|_{\mathcal{L}} \le
\\\nonumber 
&\int_{0}^1\left|\left|\left[\boldsymbol{H}(\boldsymbol{\lambda}+t(\boldsymbol{\hat{\lambda}-\lambda}))-\boldsymbol{H(\lambda)}\right] (\boldsymbol{\hat{\lambda}-\lambda}) \right|\right|_{\mathcal{L}} dt
\end{align}

Now we are ready to prove the following
\begin{claim}
Let $\boldsymbol{H(\lambda)}$ be the Hessian of the dual function $q(\boldsymbol{\lambda})$. Then, for any $\boldsymbol{v}\in \mathbb{R}^n$:
\begin{equation}\label{eq_n_7}
\left|\left|\left[\boldsymbol{H(\bar{\lambda})-H(\lambda)}\right]\boldsymbol{v} \right|\right|_{\mathcal{L}} \le \left|\left|\boldsymbol{H}(\boldsymbol{\bar{\lambda}})-\boldsymbol{H(\lambda)}\right|\right|_{\mathcal{L}} \left|\left|\boldsymbol{v} \right|\right|_{\mathcal{L}}
\end{equation}
\end{claim}
\begin{proof}
Consider three cases:
\begin{enumerate}
\item $\boldsymbol{v}\in \bf{1}^{\perp}$. In this case (\ref{eq_n_7}) follows immediately from the definition: $||\boldsymbol{S}||_{\mathcal{L}} = \sup_{\boldsymbol{v}\in \bf{1}^{\perp}}\frac{||\boldsymbol{Sv}||_{\mathcal{L}}}{||\boldsymbol{v}||_{\mathcal{L}}}$. 
\item $\boldsymbol{v}\in Span\{\bf{1}\}$. In this case $||\boldsymbol{v}||_{\mathcal{L}} = 0$ and $[\boldsymbol{H(\bar{\lambda}) - H(\lambda)}]\boldsymbol{v} = \boldsymbol{H(\bar{\lambda})v} - \boldsymbol{H(\lambda)v} = \bf{0} - \bf{0} = \bf{0}$ (because $\boldsymbol{H}(\cdot)\bf{1} = \bf{0}$). Hence, (\ref{eq_n_7}) is correct
\item $\boldsymbol{v} = \boldsymbol{u_1} + \boldsymbol{u_2}$, where $u_1\in \bf{1}^{\perp}$,$ u_2\in Span\{\bf{1}\}$. In this case $||\boldsymbol{v}||_{\mathcal{L}} = ||\boldsymbol{u_1}||_{\mathcal{L}}$, and
\begin{align*}
&\left|\left|\left[\boldsymbol{H(\bar{\lambda})-H(\lambda)}\right]\boldsymbol{v} \right|\right|_{\mathcal{L}} = \left|\left|\left[\boldsymbol{H(\bar{\lambda})-H(\lambda)}\right](\boldsymbol{u_1 + u_2}) \right|\right|_{\mathcal{L}} = \\\nonumber
&\left|\left|\left[\boldsymbol{H(\bar{\lambda})-H(\lambda)}\right]\boldsymbol{u_1} \right|\right|_{\mathcal{L}} \le^{^1} ||\boldsymbol{H(\bar{\lambda}) - H(\lambda)}||_{\mathcal{L}}||\boldsymbol{u_1}||_{\mathcal{L}} = \\\nonumber
&||\boldsymbol{H(\bar{\lambda}) - H(\lambda)}||_{\mathcal{L}}||\boldsymbol{v}||_{\mathcal{L}} 
\end{align*}
where in step ($^1$) we used the first case result.\\
This proves the claim
\end{enumerate}
\end{proof}
Applying the above claim to (\ref{taylor_first_two terms_47}) gives:
\begin{align*}
&\|\nabla q(\boldsymbol{\hat{\lambda}})- \nabla q(\boldsymbol{\lambda})-\boldsymbol{H(\lambda)(\hat{\lambda}-\lambda)}\|_{\mathcal{L}} \le \\\nonumber
&\int_{0}^1||\boldsymbol{H}(\boldsymbol{\lambda} + t(\boldsymbol{\hat{\lambda} - \lambda})) - \boldsymbol{H(\lambda)}||_{\mathcal{L}}t||\boldsymbol{\hat{\lambda} - \lambda}||_{\mathcal{L}}dt \le^{2}\\\nonumber
&\int_{0}^1B||\boldsymbol{\hat{\lambda} - \lambda}||_{\mathcal{L}}t||\boldsymbol{\hat{\lambda} - \lambda}||_{\mathcal{L}}dt = B||\boldsymbol{\hat{\lambda} - \lambda}||^2_{\mathcal{L}}\int_{0}^1tdt  \\
& = \frac{B}{2}||\boldsymbol{\hat{\lambda} - \lambda}||^2_{\mathcal{L}}
\end{align*}
where in step ($^1$) we used the fact that $\boldsymbol{H}(\cdot)$ is Lipshitz continuous with respect to the laplacian norm $||\cdot||_{\mathcal{L}}$.
\end{proof}

\begin{lemma}
Let $\bm{H}_{k}=\bm{H}(\bm{\lambda}_{k})$ be the Hessian of the dual function, then for any arbitrary $\epsilon > 0$ we have  
\begin{equation*}
e^{-\epsilon^{2}} \bm{v}^{\mathsf{T}} \bm{H}_{k}^{\dagger} \bm{v} \leq \bm{v}^{\mathsf{T}}\bm{Z}_{k}\bm{v}\leq e^{\epsilon^{2}}\bm{v}^{\mathsf{T}}\bm{H}_{k}^{\dagger}\bm{v}, \ \ \ \ \ \ \forall \bm{v} \in \bm{1}^{\perp}
\end{equation*}

\end{lemma}
\begin{proof}
Lets $\{\mu^{(k)}_i\}^{n}_{i=1}$ be the collection of eigenvalues of $\boldsymbol{H_k}$ and $\{\boldsymbol{u^{(k)}}_i\}$ are corresponding eigenvectors. Then
\begin{equation}
\boldsymbol{H_k} = \sum_{i=2}^{n}\mu^{(k)}_i\boldsymbol{u^{(k)}}_i\boldsymbol{u^{(k) \mathsf{T}}}_i \hspace{1cm} \boldsymbol{H^{\dagger}_k} = \sum_{i=2}^{n}\frac{1}{\mu^{(k)}_i}\boldsymbol{u^{(k)}}_i\boldsymbol{u^{(k) \mathsf{T}}}_i
\end{equation}  
where we use $\mu^{(k)}_1 = 0$ and $\boldsymbol{u^{(k)}}_1 = \bf{1}$. Now lets fix some $\delta > 0$ and consider the matrix $\boldsymbol{H}_{k,\delta} = \sum_{i=2}^{n}\mu^{(k)}_i\boldsymbol{u^{(k)}}_i\boldsymbol{u^{(k) \mathsf{T}}}_i + \delta \bf{11}^{\mathsf{T}} = $ $ \boldsymbol{H_k} + \delta \bf{11}^{\mathsf{T}}$. The corresponding linear system will have the form:
\begin{equation}\label{add_system}
\boldsymbol{H}_{k,\delta}\boldsymbol{d_k} = -\boldsymbol{g_k}
\end{equation}
and the operator $\boldsymbol{Z}_{k,\delta}$ defined by by \textit{EDistRSolve} routine for (\ref{add_system}) satisfies:
\begin{equation}\label{approx_expression_delta}
e^{-\epsilon^2}\boldsymbol{v}^{\mathsf{T}}\boldsymbol{H}^{-1}_{k,\delta}\boldsymbol{v} \le \boldsymbol{v}^{\mathsf{T}}\boldsymbol{Z}_{k,\delta}\boldsymbol{v}\le e^{\epsilon^2}\boldsymbol{v}^{\mathsf{T}}\boldsymbol{H}^{-1}_{k,\delta}\boldsymbol{v}, \hspace{1cm} \forall \boldsymbol{v}\in \mathbb{R}^n
\end{equation}
Notice that $\boldsymbol{H}^{-1}_{k,\delta} =  \sum_{i=2}^{n}\frac{1}{\mu^{(k)}_i}boldsymbol{u^{(k)}}_i\boldsymbol{u^{(k) \mathsf{T}}}_i + \frac{1}{\delta}\bf{11}^{\mathsf{T}} = $ $ \boldsymbol{H^{\dagger}_k} + \frac{1}{\delta}\bf{11}^{\mathsf{T}}$. Hence, taking $\boldsymbol{v}\in \bf{1}^{\perp}$ in (\ref{approx_expression_delta}):
\begin{equation}\label{approx_1_perp}
e^{-\epsilon^2}\boldsymbol{v}^{\mathsf{T}}\boldsymbol{H^{\dagger}_{k}}\boldsymbol{v} \le \boldsymbol{v}^{\mathsf{T}}\boldsymbol{Z}_{k,\delta}\boldsymbol{v}\le e^{\epsilon^2}\boldsymbol{v}^{\mathsf{T}}\boldsymbol{H^{\dagger}_{k}}\boldsymbol{v}, \hspace{1cm} \forall \boldsymbol{v}\in \bf{1}^{\perp}
\end{equation} 
The last step is to take the limit $\delta \to 0$ in (\ref{approx_1_perp}) and notice that $\boldsymbol{Z}_{k,\delta} \to \boldsymbol{Z_k}$:
\begin{equation*}
e^{-\epsilon^2}\boldsymbol{v}^{\mathsf{T}}\boldsymbol{H^{\dagger}_{k}}\boldsymbol{v} \le \boldsymbol{v}^{\mathsf{T}}\boldsymbol{Z_{k}}\boldsymbol{v}\le e^{\epsilon^2}\boldsymbol{v}^{\mathsf{T}}\boldsymbol{H^{\dagger}_{k}}\boldsymbol{v}, \hspace{1cm} \forall \boldsymbol{v}\in \bf{1}^{\perp}
\end{equation*} 
\end{proof}

\begin{lemma}
Consider the following iteration scheme $\bm{\lambda}_{k+1}=\bm{\lambda}_{k} + \alpha_{k}\tilde{\bm{d}}_{k}$ with $\alpha_{k} \in (0,1]$, then, for any arbitrary $\epsilon>0$, the Laplacian norm of the gradient, $||\bm{g}_{k+1}||_{\mathcal{L}}$, follows:
\begin{align}\label{Eq:NormGA}
||\bm{g}_{k+1}||_{\mathcal{L}} &\leq \left[1-\alpha_{k} +\alpha_{k}\epsilon\frac{\mu_{n}(\mathcal{L})}{\mu_{2}(\mathcal{L})}\sqrt{\frac{\Gamma}{\gamma}}\right]||\bm{g}_{k}||_{\mathcal{L}} +  \frac{\alpha_{k}^{2}B\Gamma^{2}(1+\epsilon)^{2}}{2\mu^{2}_{2}(\mathcal{L})}||\bm{g}_{k}||_{\mathcal{L}}^{2}
\end{align}
with $\mu_{n}(\mathcal{L})$ and $\mu_{2}(\mathcal{L})$ being the largest and second smallest eigenvalues of $\mathcal{L}$, $\Gamma$ and $\gamma$ are constants from Assumption 2, and $B \in \mathbb{R}$ is defined previously. 
\end{lemma}
\begin{proof}
Because the dual function $q(\bm{\lambda})$ has Hessian which is Lipschitz continuous with respect to the laplacian norm $||\cdot||_L$, we can write:
\begin{equation}
||g_{k+1} - g_{k} - H_k(\lambda_{k+1} - \lambda_k)||_\mathcal{L} \le \frac{B}{2}||\lambda_{k+1} - \lambda_k||^2_\mathcal{L}
\end{equation}
Using $\lambda_{k+1} = \lambda_k + \alpha_k\tilde{d}_k$ can be rewritten as:
\begin{equation}
||g_{k+1} - g_k - \alpha_kH_k\tilde{d}_k||_\mathcal{L} \le \frac{\alpha^2_kB}{2}||\tilde{d}_k||^2_\mathcal{L}
\end{equation}
Therefore,
\begin{equation}\label{eq_n_12}
||\tilde{d}_k||^2_\mathcal{L} \le \frac{\Gamma^2(1 + \epsilon)^2}{\mu^2_2(L)}||g_k||^2_\mathcal{L}
\end{equation}
Since $||g_k + \alpha_kH_k\tilde{d}_k||_\mathcal{L}$. Let $\tilde{d}_k = d_k + c_k$, then:
\begin{equation}\label{eq_n_13}
||c_k||_{H_k} \le \epsilon||d_k||_{H_k}
\end{equation}
and
\begin{align}\label{eq_n_14}
&||g_k + \alpha_kH_k\tilde{d}_k|||_\mathcal{L} = ||g_k + \alpha_kH_K(d_k + c_k)||_\mathcal{L} =  \\\nonumber
&||g_k - \alpha_kg_k + \alpha_kH_kc_k||_L \le (1 - \alpha_k)||g_k||_\mathcal{L} + \alpha_k||H_kc_k||_\mathcal{L}
\end{align}
Therefore, we need to evaluate $||H_kc_k||_\mathcal{L}$:
\begin{align*}
&||H_kc_k||^2_\mathcal{L} = c_k^TH_kLH_kc_k \le \mu_n(L)c^T_kH^2_kc_k \\\nonumber
&\le \mu_n(\mathcal{L})\mu_n(H_k)c^T_kH_kc_k \le^{^1} \mu_n(L)\frac{\mu_n(L)}{\gamma}\epsilon^2||d_k||^2_{H_k} = \\\nonumber
&\frac{\mu^2_n(\mathcal{L})\epsilon^2}{\gamma}d^\mathsf{T}_kH_kd_k = \frac{\mu^2_n(\mathcal{L})\epsilon^2}{\gamma}g^\mathsf{T}_kH^{\dagger}_kH_kH^{\dagger}_kg_k = \\\nonumber
&\frac{\mu^2_n(L)\epsilon^2}{\gamma}g^\mathsf{T}_kH^{\dagger}_kg_k \le^{2} \frac{\mu^2_n(\mathcal{L})\epsilon^2}{\gamma}\frac{\Gamma}{\mu^2_2(L)}||g_k||^2_\mathcal{L} = \epsilon^2\frac{\mu^2_n(L)}{\mu^2_2(\mathcal{L})}\frac{\Gamma}{\gamma}||g_k||^2_\mathcal{L}
\end{align*}
Hence,
\begin{equation}\label{eq_n_15}
||H_kc_k||_\mathcal{L} \le \epsilon\frac{\mu_n(\mathcal{L})}{\mu_2(\mathcal{L})}\sqrt{\frac{\Gamma}{\gamma}}||g_k||_\mathcal{L}
\end{equation}
Combining the above gives:
\begin{equation}\label{eq_n_16}
||g_k + \alpha_kH_k\tilde{d}_k||_\mathcal{L} \le \left[1 - \alpha_k + \alpha_k\epsilon\frac{\mu_n(\mathcal{L})}{\mu_2(\mathcal{L})}\sqrt{\frac{\Gamma}{\gamma}}\right]||g_k||_\mathcal{L}
\end{equation}
Therefore, we have:
\begin{align*}
&||g_{k+1}|| \le ||g_k + \alpha_kH_k\tilde{d}_k||_\mathcal{L} + \frac{\alpha^2_kB}{2}||\tilde{d}_k||^2_\mathcal{L} \le \\\nonumber
&\left[1 - \alpha_k + \alpha_k\epsilon\frac{\mu_n(\mathcal{L})}{\mu_2(\mathcal{L})}\sqrt{\frac{\Gamma}{\gamma}}\right]||g_k||_\mathcal{L} + \frac{\alpha^2_kB\Gamma^2(1 + \epsilon)^2}{2\mu^2_2(L)}||g_k||^2_\mathcal{L}
\end{align*}

\end{proof}
\begin{lemma}
Consider the algorithm given by the following iteration protocol: $\bm{\lambda}_{k+1}=\bm{\lambda}_{k+1}+\alpha^{*}\tilde{\bm{d}}_{k}$. Let $\bm{\lambda}_{0}$ be the initial value of the dual variable, and $q^{*}$ be the optimal value of the dual function. Then, the number of iterations needed by each of the three phases satisfy: 
\begin{enumerate}
\item The {\textbf{strict decrease phase}} requires the following number iterations to achieve the quadratic phase: 
\begin{equation*}
N_{1} \leq C_{1} \frac{\mu_{n}(\mathcal{L})^{2}}{\mu_{2}^{3}(\mathcal{L})} \left[1-\epsilon\frac{\mu_{n}(\mathcal{L})}{\mu_{2}(\mathcal{L})}\sqrt{\frac{\Gamma}{\gamma}}\right]^{-2},
\end{equation*}
where $C_{1}=C_{1}\left(\epsilon, \gamma, \Gamma, \bm{\delta}, q(\bm{\lambda}_{0}), q^{\star} \right)=2\bm{\delta}^{2}(1+\epsilon)^{2}\left[q(\bm{\lambda}_{0})-q^{\star}\right]\frac{\Gamma^{2}}{\gamma}$.
\item The \textbf{quadratic decrease phase} requires the following number of iterations to terminate: 
\begin{equation*}
N_{2} = \log_{2}\left[\frac{\frac{1}{2}\log_{2}\left(\left[1-\alpha^{*}\left(1-\epsilon\frac{\mu_{n}(\mathcal{L})}{\mu_{2}(\mathcal{L})}\sqrt{\frac{\Gamma}{\gamma}}\right)\right]\right)}{\log_{2}(r)}\right],
\end{equation*}
where $r=\frac{1}{\eta_{1}}||\bm{g}_{k^{\prime}}||_{\mathcal{L}}$, with $k^{\prime}$ being the first iteration of the quadratic decrease phase. 
\item The radius of the {\textbf{terminal phase}} is characterized by:
\begin{equation*}
\rho_{\text{terminal}} \leq \frac{2\left[1-\epsilon\frac{\mu_{n}(\mathcal{L})}{\mu_{2}(\mathcal{L})}\sqrt{\frac{\Gamma}{\gamma}}\right]}{e^{-\epsilon^{2}\gamma\bm{\delta}}}\mu_{n}(\mathcal{L})\sqrt{\mu_{2}(\mathcal{L})}.
\end{equation*}
\end{enumerate}
\end{lemma}

\begin{proof}
We will start with strict decrease phase. From Theorem~1:
\begin{equation}\label{lem_24_inter_eq_1}
q(\lambda_{k+1}) - q(\lambda_k) \le -\frac{1}{2}\frac{e^{-2\epsilon^2}}{(1 + \epsilon)^2}\frac{\gamma^3}{\Gamma^2}\frac{\mu^2_2(L)}{\mu^4_n(L)}\eta^2_1
\end{equation}
where $\eta_1 = \frac{1 - \mathcal{A}}{\mathcal{B}}$, with:
\begin{equation*}
\mathcal{A} = \sqrt{\left[1 - \alpha^* + \alpha^*\epsilon\frac{\mu_n(L)}{\mu_2(L)}\sqrt{\frac{\Gamma}{\gamma}}\right]};\hspace{0.2cm} \mathcal{B} = \frac{B(\alpha^*\Gamma (1 + \epsilon))^2}{2\mu^2_2(L)}
\end{equation*}
Hence:
\begin{align}
&q(\lambda_{k+1}) - q(\lambda_k) \le -\frac{1}{2}\frac{e^{-2\epsilon^2}}{(1 + \epsilon)^2}\frac{\gamma^3}{\Gamma^2}\frac{\mu^2_2(L)}{\mu^4_n(L)}\eta^2_1 \le \\\nonumber
& -\frac{1}{2}\frac{e^{-2\epsilon^2}}{(1 + \epsilon)^2}\frac{\gamma^3}{\Gamma^2}\frac{\mu^2_2(L)}{\mu^4_n(L)}\frac{(1 - \mathcal{A})^2}{\mathcal{B}^2} = \\\nonumber
&-\frac{1}{2}\frac{e^{-2\epsilon^2}}{(1 + \epsilon)^2}\frac{\gamma^3}{\Gamma^2}\frac{\mu^2_2(L)}{\mu^4_n(L)}\frac{(1 - \mathcal{A})^2}{\left(\frac{B(\alpha^*\Gamma (1 + \epsilon))^2}{2\mu^2_2(L)}\right)^2} = \\\nonumber
&-2\frac{e^{-2\epsilon^2}}{(1 + \epsilon)^6}\frac{\gamma^3}{\Gamma^6}\frac{\mu^6_2(L)}{\mu^4_n(L)}\frac{(1 - \mathcal{A})^2}{B^2(\alpha^*)^4} = \\\nonumber
&-2\frac{e^{-2\epsilon^2}}{(1 + \epsilon)^6}\frac{\gamma^3}{\Gamma^6}\frac{\mu^6_2(L)}{\mu^4_n(L)}\frac{(1 - \mathcal{A})^2}{B^2\left(\frac{e^{-\epsilon^2}}{(1 + \epsilon)^2}\left(\frac{\gamma}{\Gamma}\frac{\mu_2(L)}{\mu_n(L)}\right)^2 \right)^4} = \\\nonumber
& -2e^{2\epsilon^2}(1 + \epsilon)^2\frac{\Gamma^2}{\gamma^5}\frac{\mu^4_n(L)}{\mu^2_2(L)}\frac{1}{B^2}(1 - \mathcal{A})^2 = \\\nonumber
&-2e^{2\epsilon^2}(1 + \epsilon)^2\frac{\Gamma^2}{\gamma^5}\frac{\mu^4_n(L)}{\mu^2_2(L)}\frac{1}{\left(\frac{\mu_n(L)\xi}{\gamma\sqrt{\mu_2(L)}}\right)^2}(1 - \mathcal{A})^2 = \\\nonumber
& -2e^{2\epsilon^2}(1 + \epsilon)^2\frac{\Gamma^2}{\gamma^3}\frac{\mu^2_n(L)}{\mu_2(L)}\frac{(1 - \mathcal{A})^2}{\xi^2}
\end{align}
Now, notice that:
\begin{equation*}
1 - \mathcal{A} = \frac{1 - \mathcal{A}^2}{1 + \mathcal{A}}
\end{equation*}
Moreover, because $0 \le \mathcal{A} \le 1$, therefore:
\begin{equation}\label{lem_24_interm_eq_3}
\frac{1}{4}(1 - \mathcal{A}^2)^2 \le (1 - \mathcal{A})^2 \le (1 - \mathcal{A}^2)^2
\end{equation}
Therefore, we can write:
\begin{align}\label{lem_24_inter_eq_4}
&q(\lambda_{k+1}) - q(\lambda_k) \le  -2e^{2\epsilon^2}(1 + \epsilon)^2\frac{\Gamma^2}{\gamma^3}\frac{\mu^2_n(L)}{\mu_2(L)}\frac{(1 - \mathcal{A})^2}{\xi^2} \le \\\nonumber
&-2e^{2\epsilon^2}(1 + \epsilon)^2\frac{\Gamma^2}{\gamma^3}\frac{\mu^2_n(L)}{\mu_2(L)}\frac{1}{\xi^2}\frac{(1 - \mathcal{A}^2)^2}{4} = \\\nonumber
& -\frac{1}{2}e^{2\epsilon^2}(1 + \epsilon)^2\frac{\Gamma^2}{\gamma^3}\frac{\mu^2_n(L)}{\mu_2(L)}\frac{1}{\xi^2}(\alpha^*)^2\left[1 - \epsilon\frac{\mu_n(L)}{\mu_2(L)}\sqrt{\frac{\Gamma}{\gamma}}\right]^2 = \\\nonumber
&-\frac{1}{2}e^{2\epsilon^2}(1 + \epsilon)^2\frac{\Gamma^2}{\gamma^3}\frac{\mu^2_n(L)}{\mu_2(L)}\frac{1}{\xi^2}\left(\frac{e^{-\epsilon^2}}{(1 + \epsilon)^2}\left(\frac{\gamma}{\Gamma}\frac{\mu_2(L)}{\mu_n(L)}\right)^2\right)^2\times\\\nonumber
&\left[1 - \epsilon\frac{\mu_n(L)}{\mu_2(L)}\sqrt{\frac{\Gamma}{\gamma}}\right]^2 \\
&= -\frac{1}{2\xi^2}\frac{1}{(1 + \epsilon)^2}\frac{\gamma}{\Gamma^2}\frac{\mu^3_2(L)}{\mu^2_n(L)}\left[1 - \epsilon\frac{\mu_n(L)}{\mu_2(L)}\sqrt{\frac{\Gamma}{\gamma}}\right]^2 =\\\nonumber
&-\frac{1}{2\xi^2(1 + \epsilon)^2}\frac{\gamma}{\Gamma^2}\frac{\mu^3_2(L)}{\mu^2_n(L)}\left[1 - \epsilon\frac{\mu_n(L)}{\mu_2(L)}\sqrt{\frac{\Gamma}{\gamma}}\right]^2
\end{align}
Denote
\begin{equation*}
\delta^* = \frac{1}{2\xi^2(1 + \epsilon)^2}\frac{\gamma}{\Gamma^2}\frac{\mu^3_2(L)}{\mu^2_n(L)}\left[1 - \epsilon\frac{\mu_n(L)}{\mu_2(L)}\sqrt{\frac{\Gamma}{\gamma}}\right]^2,
\end{equation*}
then from (\ref{lem_24_inter_eq_4}) we have:
\begin{equation*}
q(\lambda_{k+1}) - q(\lambda_k) \le -\delta^* 
\end{equation*}
Hence, the number of iterations required by the algorithm for the strict decrease phase is upper-bounded by:
\begin{align*}
&N_1 \le \frac{(q(\lambda_0) - q^*)}{\delta^*} = \\\nonumber
&2\xi^2(1+ \epsilon)^2[q(\lambda_0) - q^*]\frac{\Gamma^2}{\gamma}\frac{\mu^2_n(L)}{\mu^3_2(L)}\left[1 - \epsilon\frac{\mu_n(L)}{\mu_2(L)}\sqrt{\frac{\Gamma}{\gamma}}\right]^{-2}
\end{align*}
where $q^*$ - optimal value of dual function.\\\newline
Now, lets analyze the quadratic decrease phase. We have, for $\eta_0 \le ||g_k||_L < \eta_1$:
\begin{equation*}
||g_{k+1}||_L \le \frac{1}{\eta_1}||g_k||^2_L
\end{equation*}
Hence,
\begin{equation}
\frac{1}{\eta_1}||g_{k+1}||_L \le \left(\frac{1}{\eta_1}||g_k||_L\right)^2
\end{equation}
Now, denote $l$ be the first iteration when quadratic phase is achieved, i.e $||g_l||_L < \eta_1$, therefore, for $k+l$ iteration:
\begin{align*}
\frac{1}{\eta_1}||g_{k+l}||_L \le \left(\frac{1}{\eta_1}||g_{k + l -1}||_L\right)^2 \le \ldots \le \left(\frac{1}{\eta_1}||g_l||_L\right)^{2^k} = r^{2^k}
\end{align*}
where we use notation $r = \frac{1}{\eta_1}||g_l||_L, r \in (0,1)$.
Hence, the number or iterations ADD-SDDM algorithm requires to reach
terminal phase is given by the following condition:
\begin{equation*}
\eta_1r^{2^k} < \eta_0
\end{equation*}
which immediately gives:
\begin{equation*}
k \ge \log_2\left[\frac{\log_2\left(\frac{\eta_0}{\eta_1}\right)}{\log_2r}\right] = \log_2\left[\frac{\log_2\mathcal{A}}{\log_2r}\right]
\end{equation*}
Hence,
\begin{equation*}
N_2 = \log_2\left[\frac{\frac{1}{2}\log_2\left(\left[1 - \alpha^*\left(1 -  \epsilon\frac{\mu_n(L)}{\mu_2(L)}\sqrt{\frac{\Gamma}{\gamma}}\right)\right] \right)}{\log_2r} \right]
\end{equation*}
Finally lets consider the radius of the terminal phase:
\begin{align*}
&\rho_{terminal} = \eta_0  = \frac{\mathcal{A}(1 - \mathcal{A})}{\mathcal{B}} \le \frac{1 - \mathcal{A}^2}{\mathcal{B}} = \frac{•\alpha^*\left[1 - \epsilon\frac{\mu_n(L)}{\mu_2(L)}\sqrt{\frac{\Gamma}{\gamma}}\right]}{\frac{B(\alpha^*\Gamma (1 + \epsilon))^2}{2\mu^2_2(L)}} = \\\nonumber
& \frac{2\left[1 - \epsilon\frac{\mu_n(L)}{\mu_2(L)}\sqrt{\frac{\Gamma}{\gamma}}\right]}{e^{-\epsilon^2}\gamma\xi}\mu_n(L)\sqrt{\mu_2(L)}
\end{align*}

\end{proof}

\end{document}